\documentclass{imsart}
\usepackage{comment}
\usepackage{mathrsfs}
\usepackage{csquotes}
\RequirePackage[OT1]{fontenc}
\RequirePackage{amsthm,amsmath,amsfonts,color,mathrsfs,graphicx}
\RequirePackage{natbib}
\RequirePackage[colorlinks,citecolor=blue,urlcolor=blue]{hyperref}

\startlocaldefs
\numberwithin{equation}{section}
\theoremstyle{plain}
\endlocaldefs
\theoremstyle{definition}

\newcommand{\etal}{{\em et al.}}

\usepackage{float,paralist,enumerate,enumitem}
\usepackage{ hyperref}

\makeatletter
\def\namedlabel#1#2{\begingroup
    #2%
    \def\@currentlabel{#2}%
    \phantomsection\label{#1}\endgroup
}
\makeatother

\newtheorem{thm}{Theorem}[section]
 \newtheorem{cor}{Corollary}[section]
 \newtheorem{lem}{Lemma}[section]
 \newtheorem{prop}{Proposition}[section]
 \theoremstyle{definition}
 \newtheorem{defn}{Definition}[section]
 \newtheorem{rmk}{Remark}[section]

\DeclareMathOperator{\var}{Var}   \DeclareMathOperator{\cov}{Cov}
\DeclareMathOperator{\tr}{tr}

\renewcommand{\(}{\left(}
\renewcommand{\)}{\right)}
\newcommand{\lj}{\left|}
\newcommand{\rj}{\right|}
\newcommand{\De}{\Delta}
\newcommand{\la}{\lambda}

\newcommand{\de}{\delta}
\newcommand{\al}{\alpha}
\renewcommand{\th}{\theta}
\newcommand{\si}{\sigma}

\newcommand{\ga}{\gamma}
\newcommand{\pmu}{\pmb\mu}
\newcommand{\pTh}{\pmb\Theta}
\newcommand{\pLa}{\pmb\Lambda}
\newcommand{\ep}{\epsilon}
\newcommand{\pSi}{\pmb\Sigma}

\newcommand{\pvep}{\pmb\varepsilon}
\newcommand{\vep}{\varepsilon}
\newcommand{\pep}{\pmb\epsilon}

\def\ICV{\mbox{ICV}}

\def\RCV{\mbox{RCV}}
\def\PAV{\mbox{PAV}}
\def\$\cB_m${\mbox{$\cB_m$}}
\DeclareMathOperator*{\argmin}{arg\,min}

\newcommand{\A}{{\bf A}}
\newcommand{\B}{{\bf B}}

\newcommand{\I}{{\bf I}}
\newcommand{\e}{{\bf e}}
\newcommand{\bE}{{\bf E}}
\newcommand{\R}{{\bf R}}
\newcommand{\X}{{\bf X}}
\newcommand{\Y}{{\bf Y}}
\newcommand{\Z}{{\bf Z}}
\newcommand{\W}{{\bf W}}
\newcommand{\V}{{\bf V}}
\renewcommand{\S}{{\bf S}}
\newcommand{\x}{{\bf x}}

\newcommand{\toop}{\stackrel{p}{\longrightarrow}}

\newcommand{\eqD}{\stackrel{\mathcal{D}}{=}}
\newcommand{\toD}{\stackrel{\mathcal{D}}{\to}}

\newcommand{\cF}{{\mathcal F}}
\newcommand{\cL}{{\mathcal L}}
\newcommand{\cA}{{\mathcal A}}
\newcommand{\cB}{{\mathcal B}}
\newcommand{\bC}{{\mathbb C}}
\newcommand{\bR}{{\mathbb R}}
\newcommand{\sD}{{\mathscr{D}}}
\newcommand{\Om}{{\Omega}}

\newcommand{\zz}[1]{\mathbb{#1}}

\def\tSi {{\widetilde{\Sigma}}}
\def\tpSi {{\widetilde{\pmb\Sigma}}}

\def\bpSi {{\breve{\pmb\Sigma}}}
\newcommand{\ol}{\overline}
\newcommand{\ul}{\underline}
\newcommand{\wt}{\widetilde}
\newcommand{\wh}{\widehat}
\def\eps{\varepsilon}

\def\q{\quad}

\newcommand{\ip}[1]{\langle #1 \rangle}

\definecolor{darkviolet}{rgb}{0.58, 0.0, 0.83}

\renewcommand{\theenumi}{\Alph{enumi}}

 \makeatletter
 \renewcommand{\p@enumii}{\theenumi.}
 \makeatother

\begin{document}
\pagenumbering{Alph}
\begin{titlepage}
\maketitle
\thispagestyle{empty}
\end{titlepage}
\pagenumbering{arabic}

\thispagestyle{empty}

\begin{frontmatter}

\title{On the inference about the spectral distribution of high-dimensional covariance matrix based on high-frequency noisy observations}
\runtitle{Infer spectral dist'n of HD Cov based on noisy obs}

\begin{aug}
\author{\fnms{Ningning} \snm{Xia}\ead[label=e1]{xia.ningning@mail.shufe.edu.cn}\thanksref{t1}}
\and
\author{\fnms{Xinghua} \snm{Zheng}\ead[label=e2]{xhzheng@ust.hk}\thanksref{t2}}

\thankstext{t1}{Research partially supported by GRF 606811 and 16305315 of the HKSAR, NSFC 11501348, Shanghai Pujiang Program 15PJ1402300, IRTSHUFE and the state key program in the major research plan of NSFC 91546202.}
\thankstext{t2}{Research partially supported by DAG (HKUST) and GRF 606811 and 16305315 of the HKSAR.}
\runauthor{N. Xia and X. Zheng}

\affiliation{Shanghai University of Finance and Economics, and Hong Kong University of Science and Technology}

\address{School of Statistics and Management, Shanghai\\
	Key Laboratory of Financial Information Technology,\\
	Shanghai University of Finance and Economics\\
	777 Guo Ding Road, China, 200433\\
	\printead{e1}}

\address{Department of information systems,\\
	business statistics and operations management\\
	Hong Kong University of Science and Technology\\
	Clear Water Bay, Kowloon, Hong Kong\\
	\printead{e2}}
\end{aug}


\begin{abstract}
In practice, observations are often contaminated by noise, making the resulting sample covariance matrix a signal-plus-noise sample covariance matrix.
Aiming to make inferences about the spectral distribution of the population covariance matrix under such a situation,  we establish an asymptotic relationship that describes how the limiting spectral distribution of (signal) sample covariance matrices depends on that of signal-plus-noise-type sample covariance matrices.
As an application, we consider inferences about the spectral distribution of integrated covolatility (ICV) matrices of high-dimensional diffusion processes based on high-frequency data with microstructure noise.  The (slightly modified) pre-averaging estimator is a signal-plus-noise sample covariance matrix, and the aforementioned result, together with a (generalized) connection between the spectral distribution of signal sample covariance matrices and that of the population covariance matrix, enables us to propose a two-step procedure to consistently estimate the spectral distribution of ICV for a class of diffusion processes. An alternative approach is further proposed, which possesses several desirable properties: it is more robust, it eliminates the effects of microstructure noise, and the asymptotic relationship that enables consistent estimation of the  spectral distribution of ICV  is the standard Mar\v{c}enko-Pastur equation.
The performance of the two approaches  is examined via simulation studies under both synchronous and asynchronous observation settings.
\end{abstract}


\end{frontmatter}

\section{Introduction}\label{sec:intro}

\subsection{Motivation}\label{ssec:motivation}
Covariance structure is of fundamental importance in multivariate analysis and applications. While in the classical low-dimensional setting, a usually unknown covariance structure can be estimated by the sample covariance matrix, in the high-dimensional setting, it is now well understood that the sample covariance matrix is not a consistent estimator. Furthermore, in many applications the observations are contaminated. Below, we explain one such setting that motivates this work. Similar situations  arise in many other settings, especially in signal processing (see, e.g., \cite{DS2007a},   \cite{ElKaroui10b}, and \cite{HLMNV2012}).

Our motivating question arises in the context of estimating the so-called integrated covariance matrix of the high-dimensional diffusion process and has applications to the study of stock price processes. More specifically, suppose that we have $p$ stocks whose (latent) log price processes are denoted by $(X_t^{j})$ for $j=1,\ldots,p$. Let $\X_t=(X_t^{1},\ldots,X_t^{p})^T$, where $T$ denotes the transpose.
A widely used model for $(\X_t)$ is
\begin{eqnarray}
d\X_t=\pmu_t\, dt + \pTh_t\, d\W_t, ~~~~~ t\in [0,1],
\label{diffusion}
\end{eqnarray}
where $(\pmu_t)=(\mu_t^{1},\ldots,\mu_t^{p})^T$ is a $p$-dimensional drift process,  $(\pTh_t)$ is a $p\times p$ matrix for any~$t$ called the covolatility process, and $(\W_t)$ is a $p$-dimensional standard Brownian motion. Both $(\pmu_t)$ and $(\pTh_t)$ can be stochastic and depend on the  Brownian motion $(\W_t)$. The interval $[0,1]$ is the time period of interest, say, one trading day (=six and a half hours). The integrated covariance (ICV) matrix refers to
\[
\ICV:= \ \int_0^1 {\pmb\Theta}_t{\pmb\Theta}_t^T \, dt.
\]
The ICV matrix, in particular, its spectrum (i.e., its set of eigenvalues), plays an important role in financial applications such as factor analysis and risk management.

A classical estimator of the ICV matrix is the so-called realized covariance (RCV) matrix, which relies on the assumption that $(\X_t)$ could be observed at high frequency. More specifically, suppose that $(\X_t)$ could be observed at time points $t_i=i/n$ for $i=0,1,\ldots,n$. Then, the RCV matrix is defined as
\[
\RCV=\ \sum_{i=1}^n \De\X_i\(\De\X_i\)^T,
\]
where
\[
\De\X_i= \( \begin{array}{c}
\De X_i^{1}\\
\vdots\\
\De X_i^{p} \end{array} \)
:= \( \begin{array}{c}
X_{t_i}^{1}-X_{t_{i-1}}^{1}\\
\vdots\\
X_{t_i}^{p}-X_{t_{i-1}}^{p} \end{array} \)
\]
stands for the vector of log returns over the period $[(i-1)/n, i/n]$.
The consistency and central limit theorems for the RCV matrix under such a setting \emph{and} when the dimension $p$ is fixed are well known;  see, for example, \cite{AB98},
\cite{ABDL01}, \cite{BNS02}, \cite{jacodprotter98}, \cite{myklandzhang06},
among others.

To obtain a better understanding of the above setting, it is instructive to connect it with the usual multivariate analysis setting. In the simplest case, when $\pmu_t \equiv 0$ and $\pTh_t \equiv~\pTh$, we have (1) $\ICV=\pTh \pTh^T$; (2) the returns $\De\X_i\eqD \Z_i/\sqrt{n}$, where $\Z_i\sim_{i.i.d.}  N(0, \ICV)$, and $\RCV \eqD  \frac{1}{n}\sum_{i=1}^n \Z_i \Z_i^T $. In other words, this simplest setting is equivalent to an \hbox{i.i.d.} observation setting in multivariate analysis, where the sample covariance matrix is used to estimate the population covariance matrix. In general, this situation is more complicated because both $(\pmu_t)$ and $(\pTh_t)$ can be stochastic and dependent on the underlying Brownian motion.

The so-called  market microstructure noise presents another challenge.
In practice, the observed prices are always contaminated versions of the latent prices, the error being referred to as market microstructure noise. Such noise is induced by various frictions in the trading process such as the bid-ask spread and the discreteness of price. Despite its small size,
market microstructure noise accumulates at high frequency and badly affects inferences about the latent price processes.
\cite{LPS15} compare various volatility estimators and point out that microstructure noise is not negligible when the sampling frequency is higher than one observation per five minutes.

The following additive model has been widely adopted in recent studies on volatility estimation:
\begin{equation}\label{eq:noise_additive}
\Y_{t_i}=\X_{t_i}+\pvep_i, ~~~ i=1,\cdots,n,
\end{equation}
where $\Y_{t_i}=(Y_t^{1},\ldots,Y_t^{p})^T$ denotes the observations and $\pvep_i=(\eps_i^{1},\ldots,\eps_i^{p})^T$ denotes the noise, which is \hbox{i.i.d.}
independent of $(\X_t)$ with $E(\pvep_i)=0$ and certain covariance matrix $\pSi_e$.
Observe that under \eqref{eq:noise_additive}, the observed log-returns $\Delta \Y_{t_i}:=\Y_{t_i} - \Y_{t_{i-1}}$ relate to the true log-returns $\Delta \X_{t_i}$ by the following equation:
\begin{equation}\label{eq:signal_noise_additive}
\Delta \Y_{t_i}=\Delta \X_{t_i}+\Delta\pvep_i,  ~~~ i=1,\cdots,n,
\end{equation}
where, as usual, $\Delta \pvep_{i}:=\pvep_{i} - \pvep_{{i-1}}$. We are therefore in a noisy observation setting in which the observations are contaminated by additive noise. Such a setting forms the basis of the current work.

One striking feature in \eqref{eq:signal_noise_additive} that differs from most noisy observation settings is that as the observation frequency $n$ goes to infinity, the signal, namely, the true log-return $\Delta \X_{t_i}$ becomes diminishingly small, while the noise $\Delta\pvep_i$ remains being of the same order of magnitude. Therefore, the signal-to-noise ratio goes to 0. A direct consequence is that even when the dimension is $p=1$, the optimal rate for estimating $\ICV$ is only $n^{1/4}$ instead of the usual~$\sqrt{n}$; see \cite{gloterjacod01}. In other words, due to the \emph{dominance of noise over signal}, the ``effective sample size'' is only $O(\sqrt{n})$ rather than~$n$. This can be clearly seen from the preaveraging method that we will explain in Section \ref{ssec:preaveraging}.

While the problem above constitutes our main motivation for considering a signal-plus-noise observation setting, our results are not restricted to this particular application. Our first main result, Theorem \ref{thm:LSD_signal_noise}, applies to a general setting where the signal and noise are of the same order of magnitude.

\subsection{Summary of main results}\label{ssec:summary}
Our main goal is to make inferences about the spectral distribution of the underlying population covariance matrix, in the setting above, the $\ICV$ matrix, based on the noisy observations~$(\Y_{t_i})$ as in \eqref{eq:noise_additive}. We provide two approaches, which we summarize as follows.

Approach I requires two steps. We shall introduce an intermediate matrix, $\cA_m$ as defined in \eqref{eq:conv_noiseless_pav} below. Think of the $\ICV$ matrix as the underlying population covariance matrix. Because our observations are contaminated, the resulting sample covariance matrix is a signal-plus-noise sample covariance matrix. The intermediate matrix $\cA_m$ is a sample covariance matrix based on only signals. The two steps are then
\begin{itemize}
  \item Step 1: derive the Stieltjes transform of the spectral distribution of~$\cA_m$ based on the signal-plus-noise sample covariance matrix.
  \item Step 2: based on the derived Stieltjes transform of $\cA_m$ in Step 1, further consistently estimate the spectral distribution of $\ICV$.
\end{itemize}
The two steps rely on two asymptotic results, Theorems \ref{thm:LSD_signal_noise} and \ref{thm:main}, respectively. Roughly speaking, Theorem \ref{thm:LSD_signal_noise} enables us to make inferences about the signal sample covariance based on noisy observations, and Theorem \ref{thm:main} allows us to go further back to the population covariance matrix.

Approach II is more direct. It makes use of some special properties in the setting that we described in Section \ref{ssec:motivation}. The properties allow us to asymptotically eliminate the effect of noise, thus saving us from Step I above and enabling us to take only one step, which relies on Theorem \ref{thm:B_n}.
We also see below that Approach II is more robust, particularly in that it allows for rather general dependence structures in the noise process, both cross-sectional and temporal,  and even  dependence between the  noise and  price process. The drawback is that Approach~II heavily relies on the setting in Section \ref{ssec:motivation}, while Approach~I can be applied to wider situations involving noisy observations.

In the simulation studies, we explain in detail how to generalize the algorithm  proposed by \cite{ElKaroui08} to implement the estimation procedure in practice. We can see that
Approaches I and II both yield satisfactory estimates of the spectral distribution of the targeting ICV matrix. Other algorithms, such as those introduced in \cite{Mestre08}, \cite{BCY10} and \cite{LW2015}, can also be adapted to our setting.

The rest of the paper is organized as follows. Section \ref{sec:main_result} explains the two approaches and the underlying theories.
Section \ref{sec:simulation} demonstrates how to implement the two approaches in practice.
Section \ref{sec:conclusion} concludes.
The proofs are given in the supplementary article \cite{XZ15_supp}.

\medskip \noindent \textbf{Notation.}
 For any $p\times p$ Hermitian matrix $\pSi$ with eigenvalues
$\lambda_1,\ldots,\lambda_p$,  its {empirical
spectral distribution} (ESD) is defined as
\[
   F^{\pSi}(x):=\frac{1}{p} \# \{j: \lambda_j \leq
   x\},\quad \mbox{for } x\in\zz{R}.
\]
The limit of ESD as $p\to\infty$, if it exists, is referred to as the {limiting
spectral distribution}, or LSD for short; see, for example, the book \cite{BS_book}.
For any real  matrix ${\bf A}$, $\|{\bf A}\|=\sqrt{\lambda_{\textrm{max}}({\bf A}{\bf A}^T)}$ denotes its spectral norm, where $\lambda_{\textrm{max}}$ denotes the largest eigenvalue.
For any nonnegative definite matrix $\B$, $\B^{1/2}$ denotes its square root matrix.
For any $z\in\mathbb{C}$, write $\Re(z)$ and $\Im(z)$ as its real and imaginary parts, respectively, and $\bar{z}$ as its complex conjugate. For any distribution $F$, $m_{F}(\cdot)$ denotes its Stieltjes transform, which is defined as
\[
m_{F}(z)= \ \int \ \frac{1}{\lambda-z} \ dF(\lambda), ~~~ {\rm for}~ z\in\mathbb{C}^+{:=}\{z\in\mathbb{C}: \Im(z)>0\}.
\]
In particular,  the Stieltjes transform of $F^{\pSi}$ above, denoted by $m_{\pSi}(\cdot):=m_{F^{\pSi}}(\cdot)$, is given by
\[
m_{\pSi}(z)= \frac{\tr((\pSi - z\I)^{-1})}{p}, ~~~ {\rm for}~ z\in\mathbb{C}^+,
\]
where $\I$ is the identity matrix.
Finally, for any vector $\x$, $|\x|$ stands for its Euclidean norm.

\section{Main results}\label{sec:main_result}

\subsection{Preliminary: pre-averaging method}\label{ssec:preaveraging}
The pre-averaging (PAV) method is introduced in \cite{JLMPV09}, \cite{PV09}, and \cite{CKP10} to deal with microstructure noise. Other approaches include the  two/multi-scales estimators \citep{ZMA05,Zhang06,Zhang11}, realized kernel \citep{BHLS08,BHLS11}, and quasi-maximum likelihood method \citep{Xiu10,AFX10}. We use a slight variant of the PAV approach in this work. First, choose a window length $k$. Then, group the intervals $[(i-1)/n,{i}/{n}]$ for $i=1,\ldots,2k\cdot\lfloor{n}/(2k)\rfloor$ into $m:=\lfloor{n}/(2k)\rfloor$ pairs of non-overlapping windows, each of width $(2k)/n$, where $\lfloor\cdot\rfloor$ represents rounding down to the nearest integer.
Introduce the following notation for any process ${\bf V}=({\bf V}_t)_{t\geq 0}$,
\begin{equation}\label{eq:Delta_V}
\Delta {\bf V}_i
={\bf V}_{i/n}-{\bf V}_{(i-1)/n}, ~~\overline{{\bf V}}_i=\frac{1}{k}\sum_{j=0}^{k-1}{\bf V}_{((i-1)k+j)/n},
~~{\rm and}~~
\Delta\overline{{\bf V}}_{2i}=\overline{{\bf V}}_{2i}-\overline{{\bf V}}_{2i-1}.
\end{equation}
With such notation, the observed return based on the pre-averaged price becomes
\begin{equation}\label{eq:rtn_pav}
\Delta\ol{\Y}_{2i}=\Delta\ol{\X}_{2i} + \Delta\ol{\pvep}_{2i}.
\end{equation}

One key observation is that if $k$ is chosen to be of order $\sqrt{n}$ (which is the order chosen in \cite{JLMPV09}, \cite{PV09}, and \cite{CKP10}), then, in \eqref{eq:rtn_pav}, the ``signal'' $\De\ol\X_{2i}$ and ``noise''~$\De\ol\pvep_{2i}$ can be shown to be of the same order of magnitude.   Observe that with such a chosen window width, the resulting number of windows is only of order~$\sqrt{n}$; hence our statement earlier that the effective sample size is only~$O(\sqrt{n})$ and consequently, even in the one-dimensional case, the optimal rate of convergence for estimating $\ICV$ is only $\sqrt{\sqrt{n}} = n^{1/4}$.

\subsection{Approach I}

\subsubsection{Step 1: From signal-plus-noise back to signal}\label{ssec:App_I_1}

Our starting point is the PAV matrix, which is defined as a multiple of the sample covariance matrix of $\Delta\ol{\Y}_{2i}$, the returns based on the pre-averaged prices:
\begin{equation}\label{PAV}
\aligned
\PAV&:=3\sum_{i=1}^m \left(\Delta\overline{{\bf Y}}_{2i}\right)\left(\Delta\overline{{\bf Y}}_{2i}\right)^T.
\endaligned
\end{equation}
(Coefficient 3 is inherited from \cite{JLMPV09} and comes from the convergence \eqref{eq:conv_noiseless_pav} below.) This is slightly different from the estimator in \cite{JLMPV09}, particularly in that there is no bias correction term involved. This is because (1) in the high-dimensional setting, even with the bias correction, the PAV is still inconsistent, just as in high-dimensional multivariate analysis the sample covariance matrix is inconsistent; and (2) our version of the PAV facilitates further analysis, which leads us all the way back to the target $\ICV$.

The matrix $\PAV$ can be viewed as the sample covariance matrix based on observations $\De\ol{\X}_{2i}+\De\ol{\pvep}_{2i}$, which model the situation of the information vector $\De\ol{\X}_{2i}$ being contaminated by additive noise $\De\ol{\pvep}_{2i}$.
\cite{DS2007a} consider such signal-plus-noise sample covariance matrices as
\[
\S_n=\dfrac{1}{n}\(\A_n+\si_n \bE_n\)\(\A_n+\si_n \bE_n\)^T,
\]
where $\A_n$ indicates a matrix consisting of signals, while $\bE_n$, independent of $(\A_n)$, consists of \hbox{i.i.d.} noise.
Let $\cA_n:= \A_n\A_n^T/n$ be the signal sample covariance matrix.
Under certain regularity conditions, the authors show that if $F^{\cA_n}$ converges to a probability distribution $F^{\cA}$, then so does $F^{\S_n}$. They further show that the LSD of $\S_n$ is determined by $F^{\cA}$ in that its Stieltjes transform $m=m(z)$ uniquely solves the following equation
\begin{eqnarray}\label{GMP_eqn}
~~~~~
m=\int\dfrac{dF^{\cA}(t)}{\dfrac{t}{1+\sigma^2ym}-(1+\sigma^2ym)z+\sigma^2(1-y)}, ~~~~ \mbox{for all } z\in\bC^+,
\end{eqnarray}
where $\si^2,y$ are given in Assumptions \eqref{asm:sigma_n_conv} and \eqref{asm:yn_conv} below.

Our goal in this article, as in many other applications, is to make inferences about signals based on noisy observations; in this case, to make inferences about $\cA_n$ based on $\S_n$. This motivates us to
investigate the problem from a different angle than \cite{DS2007a}. Unlike \eqref{GMP_eqn}, which states how the LSD of $\S_n$ depends on that of~$\cA_n$, we show how the LSD of $\cA_n$ depends on that of $\S_n$; see equation \eqref{eqn:LSD_signal_to_noisy} below. We further explain how such a relation enables us to consistently estimate the ESD of~$\cA_n$ based on $\S_n$.

The relation that we establish is essentially an inverse relation of~\eqref{GMP_eqn}.  Inverting such relations is in general notoriously difficult. For example, the Mar\v{c}enko-Pastur equation, which is similar to equation~\eqref{GMP_eqn} and describes how the LSD of the sample covariance matrix depends on that of the population covariance matrix, is established long time ago in \cite{MP67}, but it was after more than forty years  that researchers realized how the (unobservable) ESD of the population covariance matrix can be recovered based on the (observable) ESD of the sample covariance matrix (\citep{ElKaroui08,Mestre08, BCY10, LW2015} \hbox{etc}). In particular,  \cite{Mestre08} derived an inverse formula for estimating  individual population eigenvalues, under the  assumption that the population covariance matrix admits only finitely many distinct eigenvalues with known multiplicity. Our first result, Theorem~\ref{thm:LSD_signal_noise} below, gives an inverse relation of \eqref{GMP_eqn} that allows the derivation of the ESD of $\cA_n$ based on that of $\S_n$, under rather general assumptions.

We impose the following assumptions on the underlying matrices.
Assumptions \eqref{asm:FA_conv} and \eqref{asm:eps} are from \cite{DS2007a}; in particular, \eqref{asm:FA_conv} is about the convergence of the ESD of the signal sample covariance matrix.
Assumption \eqref{asm:sigma_n_conv} allows the variance of noise to depend on $n$ as in the case of PAV.  Assumption \eqref{asm:yn_conv} is standard in the studies of random matrices.

\begin{compactenum}\setcounter{enumi}{1}
\item[]
\begin{compactenum}
\item\label{asm:FA_conv} ${\bf A}_n$ is $p\times n$,  independent of $\pvep_n$, and with $\cA_n=(1/n)\A_n\A_n^T$,
$F^{\cA_n}\toD F^{\cA}$, where $F^{\cA}$ is a  probability distribution with the Stieltjes transform denoted by $m_{\cA}(\cdot)$;
\item\label{asm:sigma_n_conv} $\sigma_n\geq 0$ with $\lim_{n\to\infty} \sigma_n=\sigma\in[0,\infty) $;
\item\label{asm:eps} $\bE_n=(\ep_{ij})$ is $p\times n$ with the entries $\ep_{ij}$ being \hbox{\hbox{i.i.d.}} and centered with unit variance; and
\item\label{asm:yn_conv} $n=n(p)$ with $y_n=p/n\to y >0$ as $p\to\infty$.
\end{compactenum}
\end{compactenum}

We now present our first result about how the LSD of $\cA_n$ depends on that of $\S_n$.
\begin{thm}\label{thm:LSD_signal_noise}\label{THM1}
Suppose that Assumptions \eqref{asm:FA_conv}-\eqref{asm:yn_conv} hold. Then,
almost surely, the ESD of $\S_n$ converges in distribution to a probability distribution~$F$.
Moreover, if $F$ is supported by a finite interval $[a,b]$ with $a>0$ and possibly has a point mass at 0, then $F^{\cA}$ can be identified as follows.
For all $z\in\bC^+$ such that $m_{\cA}(z)\in D_{\cA}(y,\si^2):= \{\xi\in\bC: ~
z(1-y\si^2 \xi)^2- \si^2(y-1)(1-y\si^2 \xi) \in\bC^+\}$, $m_{\cA}(z)$  uniquely solves the following equation
\begin{eqnarray}
~~~~~~~~
m_{\cA}(z) = \displaystyle\int
\dfrac{dF(\tau)}{\dfrac{\tau}{1-y\si^2m_{\cA}(z)}-z(1-y\si^2m_{\cA}(z))+\si^2(y-1)}.
\label{eqn:LSD_signal_to_noisy}
\end{eqnarray}
\end{thm}

\begin{rmk}\label{rmk:alpha_range}
The restriction on $m_{\cA}(z)$ to be in $D_{\cA}$  is such that the integral on the right hand side of \eqref{eqn:LSD_signal_to_noisy} is well-defined.
Note that because $m_{\cA}(z)\to 0$ and $z m_{\cA}(z)\to -1$ as $\Im(z)\to \infty$,
$m_{\cA}(z)$ does belong to $D_{\cA}$ for all $z$ with $\Im(z)$ sufficiently large.
Furthermore, by the uniqueness of analytic continuation, knowing the values of $m_{\cA}(z)$ for $z$ with $\Im(z)$ sufficiently large is sufficient to determine $m_{\cA}(z)$ for all $z\in\bC^+.$
\end{rmk}

Let us explain how Theorem \ref{thm:LSD_signal_noise} can be used to make inferences about signals based on noisy observations.
\begin{itemize}
\item[(i)]  In practice, we observe noisy observations and can compute ${\bf S}_n$ and hence its ESD.
We can then replace $F$ in \eqref{eqn:LSD_signal_to_noisy} with $F^{{\bf S}_n}$ and solve  for $m_{\cA_n}(z)$.
The empirical version of \eqref{eqn:LSD_signal_to_noisy} can be solved numerically using, for example, the R package ``rootSolve.''
The uniqueness of the solution to equation~\eqref{eqn:LSD_signal_to_noisy} is theoretically justified using analytic tools. For its empirical version, we prove that if $\wh{m_{\cA_n}(z)}$ solves the empirical version, then it is close to the true $m_{\cA_n}$ (see Appendix \ref{appendix:claim}). This property guarantees that even if the empirical version of \eqref{eqn:LSD_signal_to_noisy} admits multiple solutions, they are all close to the true one. Consequently, because $m_{\cA_n}(z)$ fully characterizes the ESD of $\cA_n$, the estimated $m_{\cA_n}(z)$ enables us to consistently estimate the ESD.
In the simulation studies, we explain in detail how to implement this procedure in practice.
\item[(ii)] More importantly, to further estimate the ESD of the population covariance matrix, in the next step to be developed, we need $m_{\cA_n}(z)$. Theorem \ref{thm:LSD_signal_noise} provides such a necessary input. This is an important outcome of establishing the inverse relation \eqref{eqn:LSD_signal_to_noisy}.
\end{itemize}

In practice, if we are only interested in estimating the spectral distribution of the population covariance matrix, then, because in the second step we only need $m_{\cA_n}(z)$, there is actually no need to estimate the ESD of $\cA_n$. In the simulation studies, we still include this part but only for the purpose of illustrating the application of Theorem \ref{thm:LSD_signal_noise}.

We now apply Theorem \ref{thm:LSD_signal_noise} to our PAV matrix.
As mentioned in the summary in Section \ref{ssec:summary}, in Step 1, we  relate the PAV matrix to an intermediate matrix $\cA_m$ defined as follows:
\begin{equation}\label{dfn:A_m}
\cA_m:=3\sum_{i=1}^m \Delta\overline{{\bf X}}_{2i} \cdot (\Delta\overline{{\bf X}}_{2i})^T.
\end{equation}
It differs from the PAV matrix in that it does not involve the noise and can be regarded as a signal sample covariance matrix.
The assumptions under our setting analogous to \eqref{asm:FA_conv}-\eqref{asm:yn_conv} for Theorem \ref{thm:LSD_signal_noise} are then as follows.
\begin{compactenum}\setcounter{enumi}{2}
\item[]
\begin{compactenum}
\item\label{asm:A_conv} the ESD of $\cA_m$ converges to a probability distribution $F^{\cA}$ with the Stieltjes transform denoted by $m_{\cA}(z)$;
\item\label{asm:noise} the noise $(\pvep_i)_{1\le i\le n}$ are independent of $(\X_t)$ and are i.i.d. with zero mean and covariance matrix $\pSi_e=\si_p^2 \I$ for some $\si_p>0$ and $\si_p\to\si_e>0$ as $p\to\infty$;
\item\label{asm:k_PAV} $k=\lfloor\theta\sqrt{n}\rfloor$ for some $\th\in(0,\infty)$, and $m=\lfloor\frac{n}{2k}\rfloor$ satisfies $\lim_{p\to\infty} p/m=y$.
\end{compactenum}
\end{compactenum}

We then have the following Corollary as a direct consequence of Theorem~\ref{thm:LSD_signal_noise}.
\begin{cor}\label{cor:A_PAV}
Suppose that for all $p$, $(\X_t)$ is a $p$-dimensional process satisfying \eqref{diffusion}. Suppose also that Assumptions \eqref{asm:A_conv}-\eqref{asm:k_PAV} hold.
Then, almost surely, the ESD of PAV defined in \eqref{PAV} converges to a probability distribution $F$.
Moreover, if $F$ is supported by a finite interval $[a,b]$ with $a>0$ and possibly has a point mass at 0,  then $F^{\cA}$ can be identified as follows.
For all $z\in\bC^+$ such that $m_{\cA}(z)\in D_{\cA}(y,3\th^{-2}\si_e^2)$,
$m_{\cA}(z)$  uniquely solves the following equation
\begin{equation}\label{eqn:m_A}
m_{\cA}(z)= \displaystyle\int \dfrac{dF(\tau)}{\dfrac{\tau}{1- 3 y\th^{-2}\si_e^2m_{\cA}(z)}-z(1- 3 y\th^{-2}\si_e^2 m_{\cA}(z))+ 3 \th^{-2}\si_e^2(y-1)}.
\end{equation}
\end{cor}

\begin{rmk}\label{rmk:thm_apply_general_case}
Although Corollary \ref{cor:A_PAV} is stated for the case when noise components have the same standard deviations, it can readily be applied to the case when the covariance matrix  ${\pmb\Sigma}_e$  is a general diagonal matrix, say, ${\rm diag}(d_1^2,\ldots,d_p^2)$. To see this, let $d_{max}^2=\max(d_1^2,\ldots,d_p^2)$. We can then artificially add additional $\tilde{\pvep}_i$ to the original observations, where $\tilde{\pvep}_i$ are independent of $\pvep_i$ and are \hbox{i.i.d.} with zero mean  and covariance matrix $\widetilde{{\pmb\Sigma}}_e={\rm diag}(d_{max}^2-d_1^2,\ldots,d_{max}^2-d_p^2)$. The noise components in the modified observations then have the same standard deviation $d_{max}$, and Corollary \ref{cor:A_PAV} can  be applied. Note that the variances, $d_1^2,\ldots,d_p^2$, can be consistently estimated; see, for example, Theorem A.1 in \cite{ZMA05}. A similar remark applies to Theorem \ref{thm:LSD_signal_noise}.
\end{rmk}

\subsubsection{Step 2: From signal to population}\label{ssec:App_I_2}
Step 1 enables us to infer the spectral distribution of $\cA_m$ based on $\PAV$. However, just as in high-dimensional multivariate analysis where the sample covariance matrix is  not consistent, neither is $\cA_m$. This is why we need this second step, which enables us to go further back to the $\ICV$ matrix.

First, we introduce some structural assumptions on the latent process $\X$ in order to go further.
Note that the term $\Delta\overline{{\bf V}}_{2i}$ in \eqref{eq:Delta_V} can be written in a more clear form by using the triangular kernel:
\begin{equation}\label{DelV}
\aligned
\Delta\overline{{\bf V}}_{2i}
=&\frac{1}{k}\sum_{j=0}^{k-1}\left({\bf V}_{((2i-1)k+j)/n}-{\bf V}_{((2i-2)k+j)/n}\right)\\
=&\frac{1}{k}\sum_{j=0}^{k-1} \sum_{\ell=1}^k \Delta {\bf V}_{(2i-2)k+j+\ell}\\
=&\sum_{|j|<k} \left(1-\frac{|j|}{k}\right) \Delta{\bf V}_{(2i-1)k+j}.
\endaligned
\end{equation}
Based on this, it can be shown that if the dimension $p$ is fixed, then,  as $n\to\infty$,
\begin{equation}\label{eq:conv_noiseless_pav}
\sum_{i=1}^m \Delta\overline{{\bf X}}_{2i} \cdot (\Delta\overline{{\bf X}}_{2i})^T \toop  \frac{\ICV}{3},
\mbox{ hence } \cA_m \toop  \ICV.
\end{equation}
It is also easy to verify that
\[
\De\ol{\pvep}_{2i} \eqD \sqrt{\dfrac 2k}  \ \e_i,
\]
where $\e_i$s are \hbox{i.i.d.} random vectors with mean zero and covariance matrix~$\pSi_e$.

Corollary \ref{cor:A_PAV} in Step 1 allows us to consistently estimate the ESD of~$\cA_m$. In light of the convergence~\eqref{eq:conv_noiseless_pav}, it would have been sufficient for us to make inferences about the ICV if  the convergence \eqref{eq:conv_noiseless_pav} also held in the high-dimensional case. Unfortunately, this is not the case, and a further step to go from $\cA_m$ to \ICV\ is needed. Such an inference is generally impossible, as can be seen in the following. \ICV\ is an integral $\int_0^1 {\pmb\Theta}_t{\pmb\Theta}_t^T \, dt$. In the simple situation where $\pmu_t\equiv 0$ and $\pTh_t$ is deterministic, the building blocks in defining $\cA_m$, $\Delta{\bf X}_{i}$, are multivariate normals with mean 0 and covariance matrices $\int_{(i-1)/n}^{i/n}{\pmb\Theta}_t{\pmb\Theta}_t^T\, dt.$ The bottom line is all the $n$ covariance matrices, $\int_{(i-1)/n}^{i/n}{\pmb\Theta}_t{\pmb\Theta}_t^T\, dt$ for $i=1,\ldots,n$, could be very different from the \ICV! We can easily change the $n$ covariance matrices $\int_{(i-1)/n}^{i/n}{\pmb\Theta}_t{\pmb\Theta}_t^T\, dt$ and hence the distributions of $\Delta{\bf X}_{i}$ \emph{without} changing \ICV. And as both the dimension $p$ and observation frequency $n$ go to infinity, there is too much freedom in the underlying distributions, which makes inferences about \ICV\  impossible. Certain structural assumptions are necessary to  turn the impossible possible. The simplest  is to assume that ${\pmb\Theta}_t\equiv {\pmb\Theta}$, in which case $\Delta{\bf X}_{i}$ are \hbox{i.i.d.} The apparent drawback  of this assumption is that it cannot capture stochastic  volatility, which is a stylized feature in financial data. The following class of processes, introduced in \cite{ZL11}, accommodates both stochastic  volatility  and the leverage effect while still making the inference about \ICV\  possible (and the theory is already much more complicated than the \hbox{i.i.d.} observation setting).

\begin{defn}\label{classC}
Suppose that $({\bf X}_t)$ is a $p$-dimensional process satisfying~(\ref{diffusion}).
We say that $({\bf X}_t)$ belongs to Class $\mathcal{C}$ if, almost surely,
there {exist} $(\ga_t)\in D([0,1];\bR)$
and $\pLa$ a $p\times p$ matrix
satisfying tr$(\pLa \pLa^T)=p$ such that
\begin{eqnarray}\label{coval:product}
\pTh_t=\gamma_t \ \pLa,
\end{eqnarray}
where $D([0,1];\bR)$ stands for the space of c\`{a}dl\`{a}g functions from $[0,1]$ to $\bR$.
\end{defn}

\begin{rmk}\label{rmk:ga_lam}
The convention that  tr$(\pLa \pLa^T)=p$  is made to
resolve the non-identifiability built into the formulation \eqref{coval:product}, in which
one can multiply $(\ga_{t})$ and divide $\pLa$ by a same constant
without modifying the process~$(\pTh_t)$. It is thus not a restriction.
\end{rmk}

Class $\mathcal{C}$ incorporates some widely used models as special cases:
\begin{itemize}
\item The simplest case is when the drift $\pmu_t\equiv 0$ and $\ga_t\equiv\ga$, in which case the returns $\De\X_i$ are \hbox{i.i.d.}  $ N(0,\ga^2/n\cdot\pLa \pLa^T)$.
\item More generally, again, when the drift $\pmu_t\equiv 0$ while $(\ga_t)$ is independent of the underlying Brownian motion $(\W_t)$, the returns $\De\X_i$ follow mixed normal distributions.
\begin{itemize}
\item Mixed normal distributions, or their asymptotic equivalent form in the high-dimensional setting,
elliptic distributions (see Section~2 of \cite{ElKaroui13} for the asymptotic equivalence), have been widely used in financial applications. \cite{MFE05} state that ``elliptical distributions ... provided far superior models to the multivariate normal for daily and weekly US stock-return data'' and that ``multivariate return data for groups of returns of a similar type often look roughly elliptical.''
\item More recently, El Karoui, in a series of papers \citep{ElKaroui09,ElKaroui10a,ElKaroui13},  studied the Markowitz optimization problem under the setting that the returns follow mixed normal/elliptic distributions.
\end{itemize}
\item Furthermore, Class $\mathcal{C}$ allows the drift $(\pmu_t)$ to be non-zero and more importantly, the $(\gamma_t)$ process to be stochastic and even dependent on the Brownian motion $(\W_t)$ that drives the price process, thus featuring the so-called leverage effect in financial econometrics. The leverage effect  is an important stylized fact of financial returns and has drawn a great deal of attention in recent years; see, for example, \cite{AFL13} and \cite{WM14}.
\end{itemize}

Observe that if $\(\X_t\)$ belongs to Class $\mathcal{C}$, then the ICV matrix
\begin{eqnarray}
\ICV=\int_0^1 \gamma_t^2 \, dt \cdot\breve{\pSi}, ~~~~~~
{\rm where} ~~ \breve{\pSi}={\pLa}{\pLa}^T .
\label{ICV}
\end{eqnarray}
Furthermore, if the drift process ${\pmb \mu}_t\equiv 0$ and $(\gamma_t)$ is independent of $({\bf W}_t)$, then,
conditional on $(\ga_t)$ and using \eqref{DelV}, we have
\[
\Delta\overline{{\bf X}}_{2i} \ \eqD \
\sqrt{w_i}
\ \breve{{\pmb\Sigma}}^{1/2} \ {\bf Z}_i,
\]
where ${\bf Z}_i=(Z_i^{1},\ldots,Z_i^{p})^T$ consists of independent standard normals and
\begin{eqnarray}
w_i&=&
\sum_{|j|<k} \left(1-\frac{|j|}{k}\right)^2 \int_{\frac{(2i-1)k+j-1}{n}}^{\frac{(2i-1)k+j}{n}} \ \gamma_t^2 \, dt.
\label{wie}
\end{eqnarray}
It follows that
\begin{eqnarray*}
\cA_m=3 \sum_{i=1}^m \Delta\overline{{\bf X}}_{2i} \cdot (\Delta\overline{{\bf X}}_{2i})^T & \eqD & 3\sum_{i=1}^m \ w_i \ \breve{\pSi}^{1/2}  \Z_i  \Z_i^T  \breve{\pSi}^{1/2}.
\end{eqnarray*}

We now explain how to make further inferences about the \ICV\ matrix based on $\cA_m$. Doing so relies on another asymptotic result which relates $\cA_m$ to \ICV.

We impose the following assumptions on the underlying process. They are inherited from Proposition 5 of \cite{ZL11}, and we refer the readers to that article for further background and explanations. Observe in particular that Assumption \eqref{asm:leverage} allows the covolatility process to be dependent on the Brownian motion that drives the price processes. Such dependence allows us to capture the leverage effect. Assumptions \eqref{asm:Sigma_bdd} and \eqref{asm:gamma_conv} concern the spectral norm of the \ICV\ matrix. We do not require the norm to be bounded, which allows, for example, spike eigenvalues.

\smallskip
\noindent{\bf Assumption  C:}
\begin{compactenum}\setcounter{enumi}{3}
\item[]
\begin{compactenum}
\item\label{asm:X_in_C} For all $p$, $({\bf X}_t)$ is a $p$-dimensional process in Class $\mathcal{C}$ for some drift process ${\pmb\mu}_t=(\mu_t^{1},\ldots,\mu_t^{p})^T$ and covolatility process  $({\pmb\Theta}_t)=(\gamma_t{\pmb\Lambda})$;
\item\label{asm:mu_bdd} there exists $C_0<\infty$ such that for all $p$ and all $j=1,\ldots,p$, $|\mu_t^{j}|\le C_0$ for all $t\in [0,1)$ almost surely;
\item\label{asm:Sigma_conv} as $p\to\infty$, the ESD of $\breve{{\pmb\Sigma}}={\pmb\Lambda}{\pmb\Lambda}^T$ converges to a probability distribution $\breve{H}$;
\item\label{asm:Sigma_bdd} there exist $C_1<\infty$ and $\kappa<1/6$ such that for all $p$, $\|\breve{{\pmb\Sigma}}\|\le C_1p^{\kappa}$  almost surely;
\item\label{asm:leverage} there exists  a sequence of index sets~$\mathcal{I}_p$ satisfying $\mathcal{I}_p\subset \{1,\ldots,p\}$ and $\#\mathcal{I}_p=o(p)$ such that $(\gamma_t)$ may depend on       $({\bf W}_t)$ but only on $(W_t^{j}: j\in \mathcal{I}_p)$;
\item\label{asm:gamma_conv}  there exists  $C_2<\infty$ such that  for all $p$ and for all $t\in [0,1)$,    $|\gamma_t|\le C_2$ almost surely, and additionally, almost surely,  $(\ga_t)$
converges uniformly to a nonzero process $(\gamma_t^*)$ that is piecewise continuous with finitely many jumps.
\end{compactenum}
\end{compactenum}

We then have the following result connecting $\cA_m$ with \ICV.

\begin{thm}\label{thm:main}\label{THM2}
Suppose that Assumptions \eqref{asm:X_in_C}-\eqref{asm:gamma_conv}  and \eqref{asm:k_PAV} hold, then as $p\to\infty$,
\begin{compactenum}[(i)]
\item the ESDs of \ICV\ and $\cA_m$  converge to probability distributions $H$ and~$F^{\cA}$ respectively, where
\begin{equation}\label{Hx}
H(x) \ = \ \breve{H}(x/\zeta) \q \mbox{for  all }  x\geq 0 ~ {\rm with} ~ \zeta=\int_0^1(\gamma_t^*)^2\, dt;
\end{equation}
\item $F^{\cA}$ and $H$ are related as follows:
\begin{eqnarray}\label{mz_mainresult}
m_{\cA}(z)=-\dfrac{1}{z}\int\dfrac{\zeta}{\tau M(z)+\zeta} ~ dH(\tau),
\end{eqnarray}
where $M(z)$ and another function $\wt{m}(z)$ uniquely solve the following equations in $\bC^+\times \bC^+$
\begin{equation}\label{Mm_z}
\left\{
\begin{array}{lll}
M(z) &=& -\dfrac{1}{z} \displaystyle\int_0^1 \dfrac{(\ga_s^*)^2}{1+y\wt{m}(z) (\ga_s^*)^2}ds,  \\
\wt{m}(z) &=& {-\dfrac{1}{z} \displaystyle\int \dfrac{\tau}{\tau M(z)+\zeta} dH(\tau).}
\end{array}
\right.
\end{equation}
\end{compactenum}
\end{thm}

Equation \eqref{mz_mainresult} in Theorem \ref{thm:main} forms the basis for us to further estimate the ESD of $\ICV$. It involves an unknown function $M(z)$, which can be solved as follows. First note that multiplying $\wt{m}(z)$ and $M(z)$ on both sides of the first
and second equations in \eqref{Mm_z}, respectively, yields
\[
\left\{
\begin{array}{lll}
M(z)\cdot \wt{m}(z) &=&-\dfrac{1}{yz} + \dfrac{1}{yz} \int_0^1 \dfrac{1}{1+y\wt{m}(z)(\ga_s^*)^2} ~ ds,\\
M(z)\cdot \wt{m}(z) &=&  -\dfrac{1}{z} \int \dfrac{\tau M(z)}{\tau M(z)+\zeta} ~ dH(\tau) = -\dfrac1z -m_\cA(z),
\end{array}
\right.
\]
where the last step is due to  \eqref{mz_mainresult}. It follows that
\begin{equation}\label{Mm_1_2}
-\dfrac1z -m_{\cA}(z) = -\dfrac{1}{yz} +\dfrac{1}{yz} \int_0^1 \dfrac{1}{1+y\wt{m}(z)(\ga_s^*)^2}~ ds,
\end{equation}
and $\wt{m}(z)=-(1/z+m_\cA(z))/M(z).$
Substituting the last expression of $\wt{m}(z)$ into equation \eqref{Mm_1_2} yields
\begin{eqnarray}\label{eqn:sol_Mz}
\int_0^1\dfrac{M(z)}{M(z)-(\ga_s^*)^2 y(z^{-1}+m_\cA(z))} ~ ds =1-y-yzm_\cA(z).
\end{eqnarray}
$M(z)$ is then obtained by plugging in the $m_\cA(z)$ that we derived in Step 1 into \eqref{eqn:sol_Mz} and solving for the solution that is unique in $\bC^+$ by Theorem 1 in \cite{ZL11}.

Having solved $M(z)$, we can then utilize equation \eqref{mz_mainresult} to estimate the ESD of \ICV\  by generalizing the algorithms in \cite{ElKaroui08}, \cite{Mestre08}, \cite{BCY10} and \cite{LW2015} \hbox{etc.} The resulting estimate can  be shown to be consistent by using an argument similar to that for Theorem 2 of \cite{ElKaroui08}. The estimation procedure is explained in detail in the simulation studies.

\subsection{Approach II}\label{ssec:App_II}
The second step in Approach I involves the process~$(\ga_s^*)$ which is unknown in practice. Estimating this process inevitably introduces an additional source of error. Motivated by this consideration, we draw ideas from \cite{ZL11} and develop an alternative approach that overcomes this difficulty. It is also worth mentioning that the alternative approach allows for rather general dependence structures in the noise process, both cross-sectional and temporal,  and even  dependence between the noise and price process. The temporal dependence between  microstructure noise has been documented in recent studies; see, e.g., \cite{hansenlunde06}, \cite{UO09} and \cite{JLZ_noise}.

We shall define a matrix that is an extension of the time-variation adjusted RCV matrix introduced in \cite{ZL11} to our noisy setting.
To start, fix an $\al\in(1/2,1)$ and $\th\in(0,\infty)$, and
let $k=\lfloor\th n^{\al}\rfloor$ and $m=\lfloor n/(2k)\rfloor$.
The time-variation adjusted PAV matrix is then defined as
\begin{equation}
\cB_m:= 3 \dfrac{\sum_{i=1}^m |\De\ol{\Y}_{2i}|^2}{m}\cdot\sum_{i=1}^m \dfrac{\De\ol{\Y}_{2i}(\De\ol{\Y}_{2i})^T}{|\De\ol{\Y}_{2i}|^2} \
= \ 3 \dfrac{\sum_{i=1}^m |\De\ol{\Y}_{2i}|^2}{p}\ \widetilde{\pSi},
\label{eqn:B_n}
\end{equation}
where
\begin{eqnarray}
\widetilde{\pSi}:= \ \dfrac{p}{m} \ \sum_{i=1}^m \dfrac{\De\ol{\Y}_{2i}(\De\ol{\Y}_{2i})^T}{|\De\ol{\Y}_{2i}|^2}.
\label{tSi}
\end{eqnarray}

Note that here, the window length $k$ has a higher order than  in Theorem~\ref{thm:main}.
The reason is that, after pre-averaging, the underlying returns are $O_p(\sqrt{k/n})$ and the noises are $O_p(\sqrt{1/k})$. In Theorem \ref{thm:main}, we balance the orders of the two terms by choosing $k=O(\sqrt{n})$; here we take $k=O(n^{\alpha})$ for some $\alpha>1/2$, which  enables us to asymptotically eliminate the effect of noise.

Next, recall the concept of $\rho$-mixing coefficients.
\begin{defn} \label{rho_corr}
Suppose that $U=(U_k,k\in\zz{Z})$ is a stationary time series.  For $-\infty \le j\le \ell \le \infty$, let $\cF_j^\ell$ be the $\sigma$-field  generated by the random variables $(U_k: j\le k\le\ell)$.
The $\rho$-mixing coefficients  are defined as
\[
\rho(r) = \sup_{f\in\cL^2(\cF_{-\infty}^0), ~ g\in\cL^2(\cF_r^{\infty})} \ \left| {\rm Corr}(f,g)\right|, \q \mbox{for}\q r\in\zz{N},
\]
where, for any probability space $\Om$,  $\cL^2(\Om)$ refers to the space of square-integrable, $\Om$-measurable random variables.
\end{defn}

We now introduce a number of assumptions. Assumption~\eqref{asm:eps_general} below says that we allow for rather general dependence structures in the noise process, both cross-sectional and temporal. We actually do not put any restrictions on the cross-sectional dependence, and even dependence between the  noise and price process is allowed.  Note also that \cite{JLZ_noise} provides an approach to estimate the decay rate of the $\rho$-mixing coefficients.   Assumption~\eqref{asm:leverage_2} concerns the dependence between the covolatility process and the Brownian motion that drives the price processes. Assumption \eqref{asm:vol_bdd} is about the boundedness of individual volatilities.

\begin{compactenum}\setcounter{enumi}{4}
\item[]
\begin{compactenum}
\item\label{asm:eps_general}  For all $j=1,\cdots,p$, the noise $(\vep_i^j)$ is stationary, has mean 0 and  bounded $4\ell$th moments, and has $\rho$-mixing coefficients $\rho^j(r)$ satisfying  $\max_{j=1,\cdots,p}   \rho^j(r)=O(r^{-\ell})$ for some integer $\ell \geq 2$;
\item\label{asm:leverage_2} there exist $0\le \de_1<1/2$ and a sequence of index sets $\mathcal{I}_p$ satisfying $\mathcal{I}_p\subset\{1,\ldots,p\}$ and $\# \mathcal{I}_p=O(p^{\de_1})$ such that $(\ga_t)$ may depend on $(\W_t)$ but only on $(W_t^{j}: j\in\mathcal{I}_p)$;
\item\label{asm:gamma_bdd} there exists  $C_1<\infty$ such that for all $p$, $|\ga_t|\in\(1/C_1,C_1\)$ for all $t\in[0,1)$ almost surely;
\item\label{asm:vol_bdd}  there exists  $C_2<\infty$ such that for all $p$ and  all $j$, the individual volatilities $\si_t=\sqrt{(\ga_t)^2\cdot\sum_{k=1}^p (\Lambda_{jk})^2}\in\(1/C_2,C_2\)$ for all $t\in[0,1]$ almost surely;
\item\label{asm:Sigma_bdd_2} there exist $C_3<\infty$ and $0\le\de_2<1/2$ such that for all $p$, $\|\ICV\|\le C_3 p^{\de_2}$ almost surely;
\item the $\de_1$ in \eqref{asm:leverage_2} and $\de_2$ in \eqref{asm:Sigma_bdd_2} satisfy that $\de_1+\de_2<1/2$;
\item\label{asm:ym_conv}  $k=\lfloor\theta n^{\alpha}\rfloor$ for some $\th\in(0,\infty)$ and $\alpha\in [(3+\ell)/(2\ell+2) ,1)$,   and $m=\lfloor\frac{n}{2k}\rfloor$ satisfy $\lim_{p\to\infty} p/m=y>0$,
where $\ell$ is the integer in  \eqref{asm:eps_general}.
\end{compactenum}
\end{compactenum}

\begin{rmk}\label{rmk:compatibility} Careful readers may have noted that Assumptions \eqref{asm:k_PAV} and \eqref{asm:ym_conv} are mathematically incompatible, as  Assumption \eqref{asm:k_PAV} requires  $p=O(\sqrt{n})$ while  Assumption \eqref{asm:ym_conv} requires $p=O(n^{1-\al})$ for some $\al\in (1/2,1)$. The two assumptions are, however, perfectly compatible in practice when we deal with finite samples. Take the choices of $(p,n,k)$ in the simulation study in Section \ref{ssec:sim_syn} below  for example.
There, we take $(p,n)=(100,\ 23400)$. When applying Corollary \ref{cor:A_PAV} and Theorem \ref{thm:main}, we take $k=\lfloor0.5\sqrt{n}\rfloor=76$, which leads to $y=p/\lfloor n/2k\rfloor\approx 0.7$ in Assumption \eqref{asm:k_PAV};  when applying Theorem \ref{thm:B_n} below, we take $k=\lfloor 1.5 n^{0.6}\rfloor=627$, which gives  $y=p/\lfloor n/2k\rfloor\approx 5.6$ in Assumption \eqref{asm:ym_conv}.
\end{rmk}

We have the following convergence result connecting $\cB_m$ with ICV.

\begin{thm}\label{thm:B_n}\label{THM3}
Suppose that Assumptions (\ref{asm:X_in_C}), (\ref{asm:mu_bdd}), \eqref{asm:Sigma_conv},  \eqref{asm:gamma_conv}, and \eqref{asm:eps_general}-\eqref{asm:ym_conv} hold. Then, as $p\to\infty$,
the ESDs of $\ICV$ and $\cB_m$ converge almost surely to probability distributions $H$ and
$F^\cB$, respectively, where $H$ satisfies \eqref{Hx} and $F^\cB$ is determined by $H$ in that its Stieltjes transform, denoted by $m_\cB(z)$, satisfies the  following (standard) Mar\v{c}enko-Pastur equation
\begin{eqnarray}\label{eqn:mp}
~~~~ ~~  m_\cB(z)
=\int_{\tau\in\bR}\dfrac{1}{\tau\(1-y(1+z m_\cB(z))\)-z} \ dH(\tau),
~~{\rm for}~ z\in\bC^+.
\end{eqnarray}
\end{thm}

Theorem \ref{thm:B_n} states that the LSDs of ICV and $\cB_m$ are related via  the Mar\v{c}enko-Pastur equation.  Several algorithms have been developed to consistently recover $H$ by inverting the Mar\v{c}enko-Pastur equation; see, for example, \cite{ElKaroui08, Mestre08, BCY10, LW2015} \hbox{etc.}
We can therefore consistently estimate the ESD of ICV by using these existing algorithms.


\subsection{A remark about asynchronicity}\label{ssec:asyn}

In multivariate high-frequency data analysis, in addition to microstructure noise, there is another challenge due to asynchronous trading. In practice, different stocks are traded at different times; consequently, the tick-by-tick data are not observed synchronously. There are several existing methods for synchronizing data, such as the refresh times \citep{BHLS11} and previous tick methods \citep{Zhang11}. Asynchronicity is less of an issue than microstructure noise. For example, as pointed out in \cite{Zhang11}, asynchronicity does not induce bias in the two-scales estimator,  and even the asymptotic variance is the same  as if there is no asynchronicity. While a rigorous treatment is beyond the scope of this article, we expect our methods to work  for asynchronous data  as well. The reason, roughly speaking, is as follows. Take the previous tick method for example. Here, we choose a (usually equally spaced) grid of time points $0=t_0<t_1<\ldots<t_n=1$, and for each stock $j$, for each time point~$t_i$, let~$\tau_i^j$ be the latest transaction time before $t_i$. One then acts as if one observes~$Y_{\tau_i^j}^j$ at time $t_i$ for stock $j$. With the original additive model at time~$\tau_i^j$:
\[
Y_{\tau_i^j}^j = X_{\tau_i^j}^j + \eps_{i}^j,
\]
we have at time $t_i$,
\begin{eqnarray}\label{asyn_obs}
Y_{t_i}^j:=Y_{\tau_i^j}^j = X_{t_i}^j +\left(( X_{\tau_i^j}^j - X_{t_i}^j)  + \eps_{i}^j\right).
\end{eqnarray}
In other words, the asynchronicity induces an additional error $( X_{\tau_i^j}^j - X_{t_i}^j)$. The error is, however, diminishingly small as the sampling frequency $n\to\infty$ because $X_{\tau_i^j}^j - X_{t_i}^j=O_p\Big(\sqrt{t_i - \tau_i^j}\Big) = o_p(1)$. In short, asynchronicity  induces an additional error (and violates our model assumption); fortunately, the error is of negligible order compared with the microstructure noise $(\eps_i^j)$. We therefore keep our focus on the model \eqref{eq:signal_noise_additive}.  In simulation studies, in addition to the synchronous observation setting, we consider an asynchronous setting where the observation times for different stocks are independent Poisson processes. We shall see that our methods still work well (see Section \ref{ssec:sim_asyn} for more details).

\section{Simulation studies}\label{sec:simulation}

In this section, we demonstrate how to estimate the ESD of ICV by using Approach I, which uses the PAV matrix,
and Approach II, which uses the alternative matrix  $\cB_m$.

\subsection{When observations are synchronous}\label{ssec:sim_syn}
We first consider a setting where observations are synchronous. To generate the underlying process $\X$, the process $(\ga_t)$ in Definition \ref{classC} is taken to be  a stochastic U-shaped $(\ga_t)$ process  as follows:
\[
d\gamma_t=-\rho (\ga_t-\phi_t)\,dt +\sigma\, d\wt{W}_t, \q \mbox{for}\q t\in [0,1],
\]
where  $\rho=10$, $\sigma=0.05$,
\[
\phi_t \ = \ 2 \sqrt{0.0009+0.0008\cos(2\pi t)},
\]
and $\wt{W}_t=\sum_{i=1}^p W_t^{i}/ \sqrt{p}$ with $W_t^{i}$ being the $i$th component of the Brownian motion $(\W_t)$ that drives the price process. Observe that such a formulation makes $(\ga_t)$ dependent on \emph{all} the  component of the underlying Brownian motion; hence, Assumptions \eqref{asm:leverage} and \eqref{asm:leverage_2} are both violated. However, we shall see  that our methods still work well.
A sample path of $(\ga_t)$ is given in Figure \ref{fig:gamma_t} in the supplementary article of \cite{XZ15_supp}.


Next, the matrix $\breve{\pSi}=\pLa \pLa^T$ is taken to be $\mathbf{UD U^T}$, where $\mathbf{U}$ is a random orthogonal matrix and $\mathbf{D}$ is a diagonal matrix whose diagonal entries are drawn independently from the Beta$(1,3)$ distribution. Such generated $\breve{\pSi}$ does not necessarily have trace $p$, but as we pointed out in Remark \ref{rmk:ga_lam}, the assumption  $\tr(\breve{\pSi}) =p$ is a convention rather than a requirement.
With such  $(\ga_t)$ and $\breve{\pSi}$, the individual daily volatilities are around $3\%$, which is similar to what one observes in practice.
The latent log price process $(\X_t)$ follows
\begin{equation}\label{sim_setup:syn}
d\X_t =\ga_t  \pLa\, d\W_t, \q \mbox{where}\q \pLa =\breve{\pSi}^{1/2}.
\end{equation}
Finally, the noise $(\pvep_i)_{1\le i\le n}$ is taken to be \hbox{i.i.d.} $N(0,0.0002\, \I)$.

In the studies below, the dimension, that is, the number of stocks $p$, is taken to be 100, and the observation frequency $n$ is set to be 23,400, which corresponds to one observation per second on a regular trading day. Note again that because of the presence of noise, the ``effective sample size'' is only of order $\sqrt{n}\approx 153$, which is comparable to our chosen $p=100$.

\subsubsection{Estimation using Approach I: synchronous setting} \label{ssec:sim_App_I}

We start with Approach I, which involves two steps.

In the first step, we replace $F$  in equation \eqref{eqn:m_A} with the ESD of PAV and solve for $m_\cA(z)$ using the R package ``rootSolve.'' The window length~$k$ in defining PAV is set to be $\lfloor 0.5\sqrt{n}\rfloor=76$. As to the $m_\cA(z)$ to be solved, we choose a set of  $z$'s whose real and imaginary parts are equally spaced in the intervals $[-20,0]$ and $[1,20]$, respectively. Denote these $z$'s by $\{z_j\}_{j=1}^J$ and the estimated $m_{\cA_m}(z_j)$ by $\wh{m_{\cA_m}(z_j)}$.
We then need to estimate the ESD of~$\cA_m$ based on $\{\wh{m_{\cA_m}(z_j)}\}_{j=1}^J$, which we do as follows.

Inspired by the nonparametric estimation method proposed in \cite{ElKaroui08}, we approximate $F^{\cA_m}$ with a weighted sum of point masses
\begin{equation}\label{eq:disc_cA}
F^{\cA_m} \approx \sum_{k=1}^K w_k \de_{x_k},
\end{equation}
where $\{x_1 < x_2< \ldots< x_K\}$ is a grid of points to be specified and the $w_k$s are weights to be estimated.
To choose the grid $\{x_k\}_{k=1}^K$,
naturally, we would like $[x_1, x_K]$ to cover the support of $F^{\cA_m}$, which is unknown.
To overcome this difficulty, note that the support of the ESD of PAV always covers that of $\cA_m$; hence, we can choose $x_k$s to be equally spaced between 0 and the largest eigenvalue of PAV, and we are guaranteed that $[x_1, x_K]$ covers the support of $F^{\cA_m}$.

Next we discuss how to estimate the weights $\{w_k\}$ in \eqref{eq:disc_cA}. Observe that the discretization~\eqref{eq:disc_cA}  gives an approximate Stieltjes transform of $F^{\cA_m}$ as $\sum_{k=1}^K \dfrac{w_k}{x_k-z}.$
Let
\[
e_j^{'}:=\wh{m_{\cA_m}(z_j)} -  \sum_{k=1}^K \dfrac{w_k}{x_k-z_j}, \q j=1,\cdots, J
\]
be  the approximation errors.
The weights $\{w_k\}_{k=1}^K$ are then estimated by minimizing the approximation errors:
\begin{equation}\label{eq:est_wt_A}
\argmin\limits_{(w_1,\ldots,w_k)}\max\limits_{j=1,2,\cdots,J} \max \{|\Re(e_j)|, |\Im(e_j)|\}\quad \text{subject}\ \text{to} \quad \sum_{k=1}^K w_k =1 \mbox{ and } w_k\geq 0.
\end{equation}

Next, in Step 2, we estimate  the ESD of ICV.
By plugging in the $\{\wh{m_{\cA_m}(z_j)}\}_{j=1}^J$ obtained in the first step and solving equation \eqref{eqn:sol_Mz}, we obtain $\{\wh{M(z_j)}\}_{j=1}^J$. The estimation of the ESD of \ICV\ is then conducted similarly as above as follows. Discretize the ESD of $\ICV$ as
\begin{eqnarray}\label{dH_ICV}
F^{\ICV} \approx \sum_{k=1}^K c_k \de_{x_k},
\end{eqnarray}
where $c_k$s are again weights to be estimated.   By equation \eqref{mz_mainresult}, we expect that
\[
e_j^{''}:= \wh{m_{\cA_m}(z_j)} + \dfrac{1}{z_j} \cdot \sum_{k=1}^K ~ c_k~ \dfrac{\zeta}{x_k \wh{M(z_j)} +\zeta}
\]
to be small. The $c_k$s are then estimated by minimizing the approximation errors $e_j^{''}$ just as in \eqref{eq:est_wt_A}.


Figures \ref{fig:Thm1-2} and  \ref{fig:Thm1-2-2}  below illustrate the estimation results. The left plot of Figure \ref{fig:Thm1-2} shows three ESDs, those of \ICV, $\cA_m$, and PAV. The three curves are clearly different from each: the difference between PAV and~$\cA_m$ is induced by  noise, while that between $\cA_m$ and \ICV\ is caused by high-dimensionality. Note that we only observe the ESD of PAV, whereas the ESDs of both \ICV\ and $\cA_m$ are underlying.
Our goal is to estimate the  ESD of ICV. As we explained below Theorem \ref{thm:LSD_signal_noise},  such a goal does not require estimating the ESD of $\cA_m$.
Here, we  still estimate this ESD to illustrate the application of Corollary \ref{cor:A_PAV}.
The estimation of the ESD of $\cA_m$ is conducted in the first step, and the result is shown in the right plot of Figure \ref{fig:Thm1-2}. The second step estimates the ESD of \ICV, with the result given in Figure \ref{fig:Thm1-2-2}.
\begin{figure}[H]\centering
\includegraphics[width=.43\textwidth]{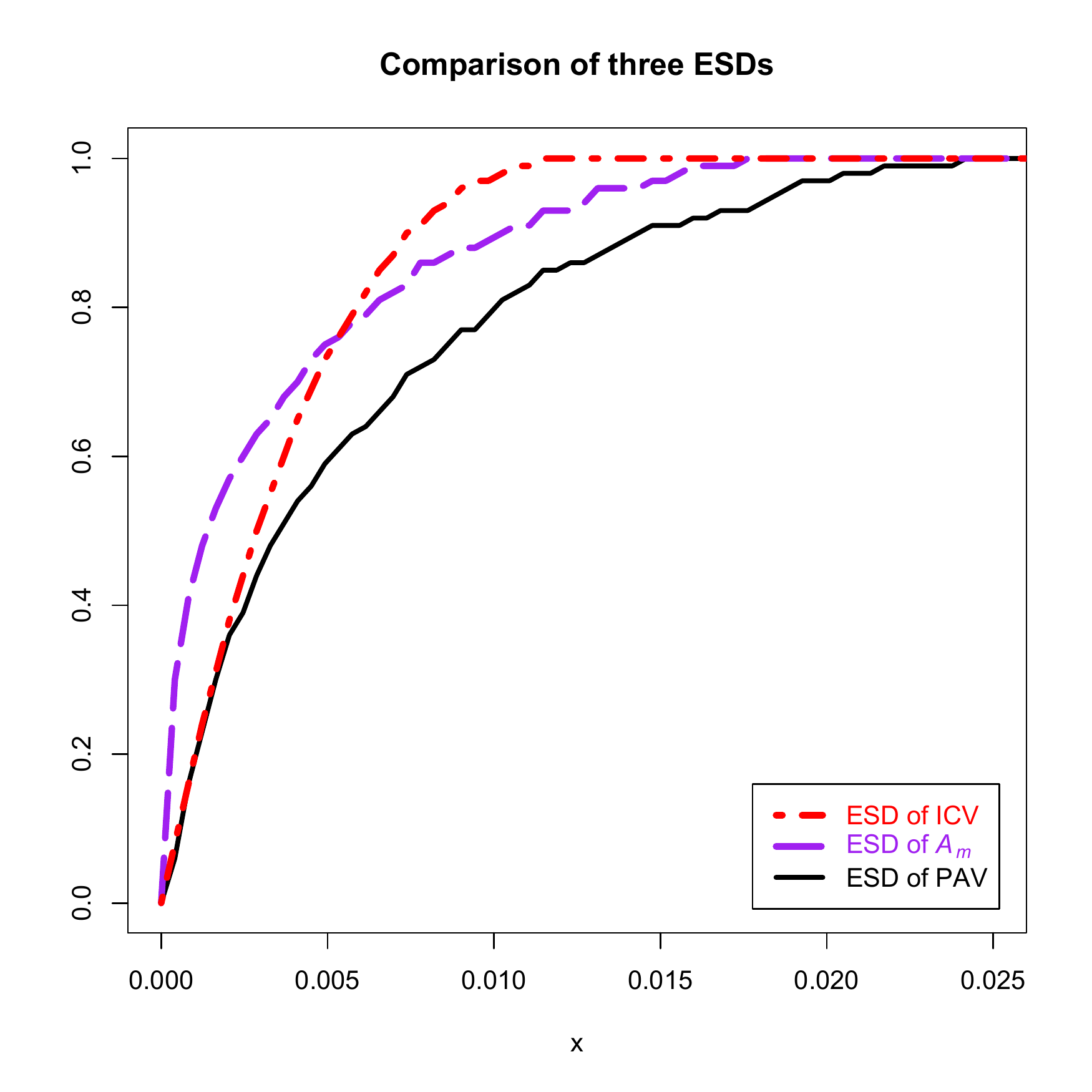}
\includegraphics[width=.43\textwidth]{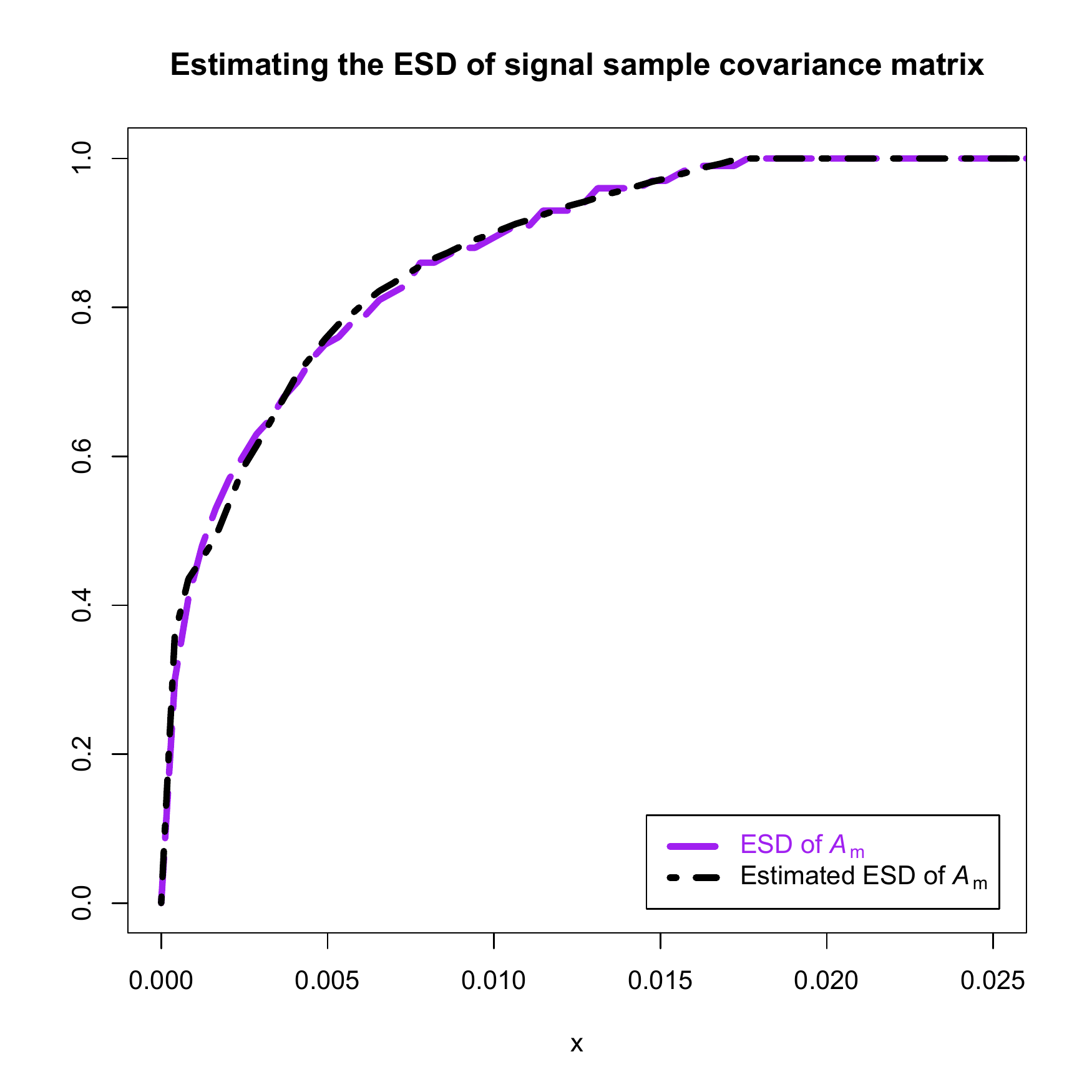}
\caption{Approach I -- Step 1: Estimation of the empirical spectral distribution of the signal sample covariance matrix $\cA_m$ based on synchronous noisy observations under model \eqref{sim_setup:syn}. The dimension  $p=100$, and the observation frequency $n=23400$. }
\label{fig:Thm1-2}
\end{figure}
\begin{figure}[H]\centering
\includegraphics[width=.43\textwidth]{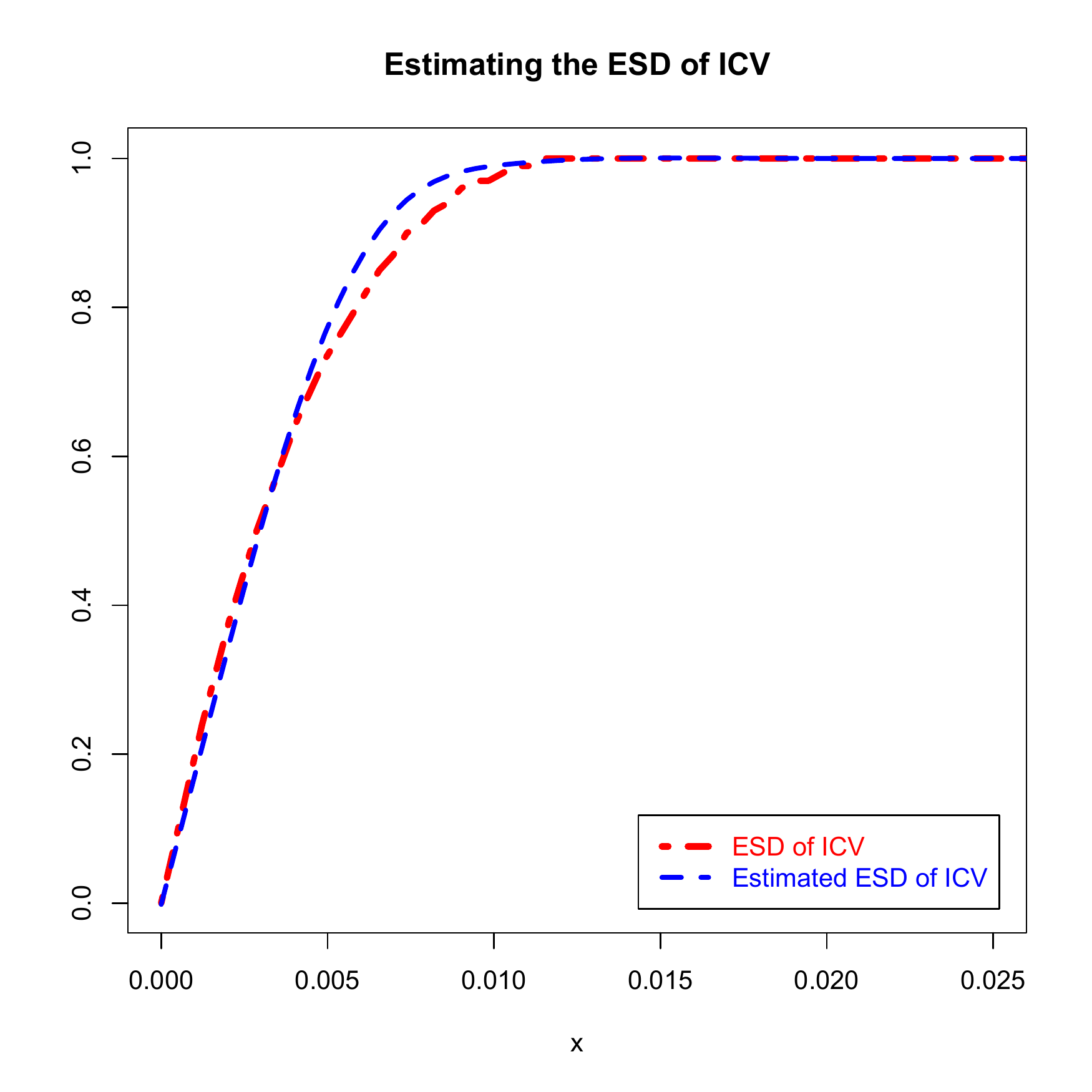}
\caption{Approach I -- Step 2: Estimation of the empirical spectral distribution of targeting ICV matrix.}
\label{fig:Thm1-2-2}
\end{figure}

Figures \ref{fig:Thm1-2} and \ref{fig:Thm1-2-2} show that the ESDs of both $\cA_m$ and \ICV\   can be estimated quite well.

\subsubsection{Estimation using Approach II: synchronous setting} \label{ssec:sim_App_II}

We now apply Approach II to estimate the ESD of \ICV. According to Theorem \ref{thm:B_n}, asymptotically, the ESD of  $\cB_m$ is related to that of ICV through the standard Mar\v{c}enko-Pastur equation. This allows us to directly apply existing algorithms that are developed to invert the Mar\v{c}enko-Pastur equation to estimate the ESD of \ICV, and in the below we adopt the algorithm proposed in \cite{ElKaroui08}.

Specifically, set the window length $k$ in defining $\cB_m$ to be $\lfloor1.5n^{0.6}\rfloor=627$. Discretize the ESD of ICV as \eqref{dH_ICV}.
According to Theorem \ref{thm:B_n}, the Stieltjes transform of the ESD of $\cB_m$, denoted by $m_{\cB_m}(z)$, should approximately satisfy equation \eqref{eqn:mp} with $H$ replaced with the ESD of ICV. In other words, we again expect the approximation errors
\[
e_j^{'''}:= m_{\cB_m}(z_j) -\sum_{k=1}^K \dfrac{c_k}{x_k(1-y(1+z_jm_{\cB_m}(z_j)))-z_j}
\]
to be small. Thus, again, we estimate the weights $c_k$'s by minimizing the approximation errors $e_j^{'''}$ as in \eqref{eq:est_wt_A}.

The estimation results are given in Figure \ref{fig:Thm3}. Again, we see from the left plot that the ESD of $\cB_m$ clearly differs from the (latent unobserved) ESD of  \ICV, yet the right plot shows that we can estimate this latent distribution well.

\begin{figure}[H]\centering
\includegraphics[width=.43\textwidth]{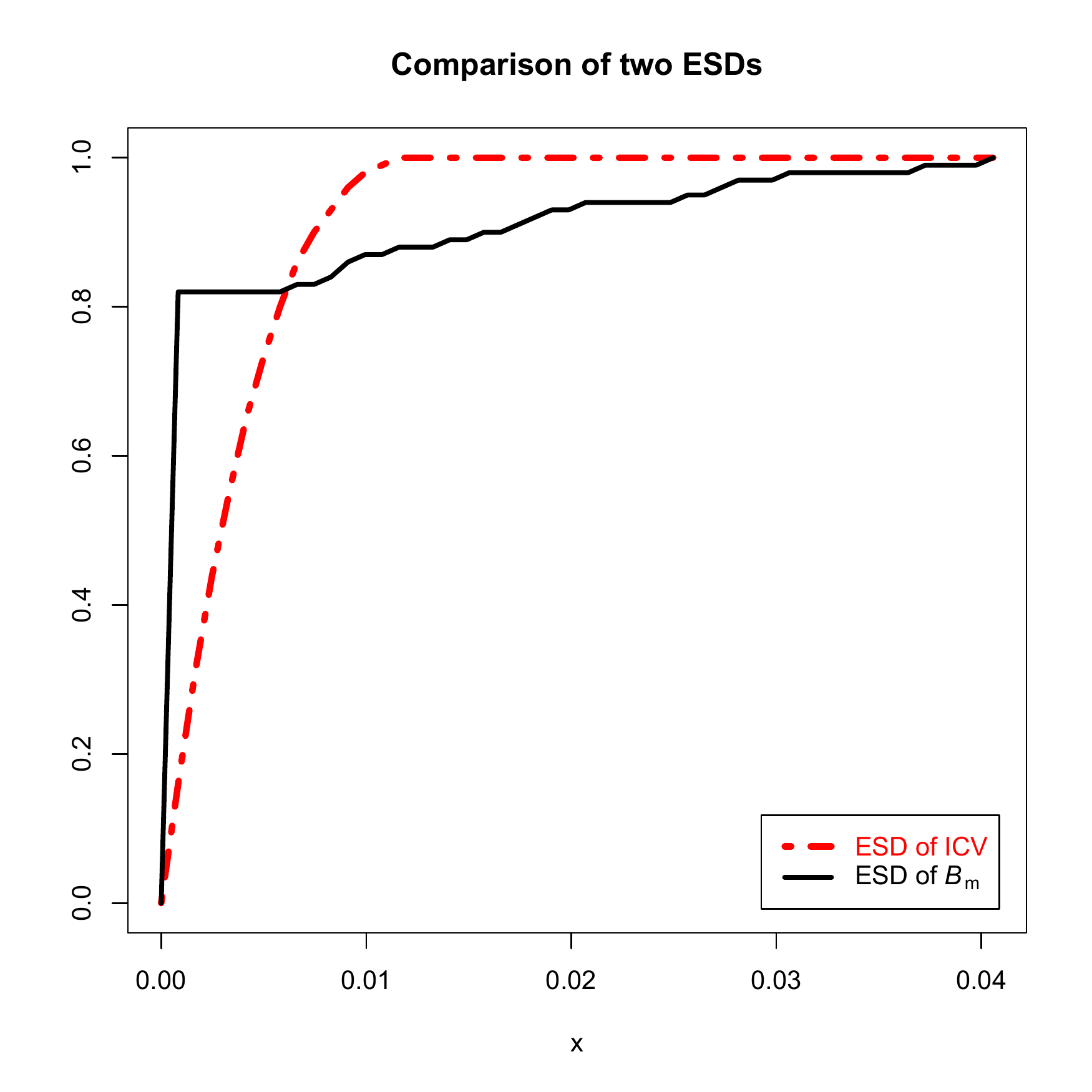}
\includegraphics[width=.43\textwidth]{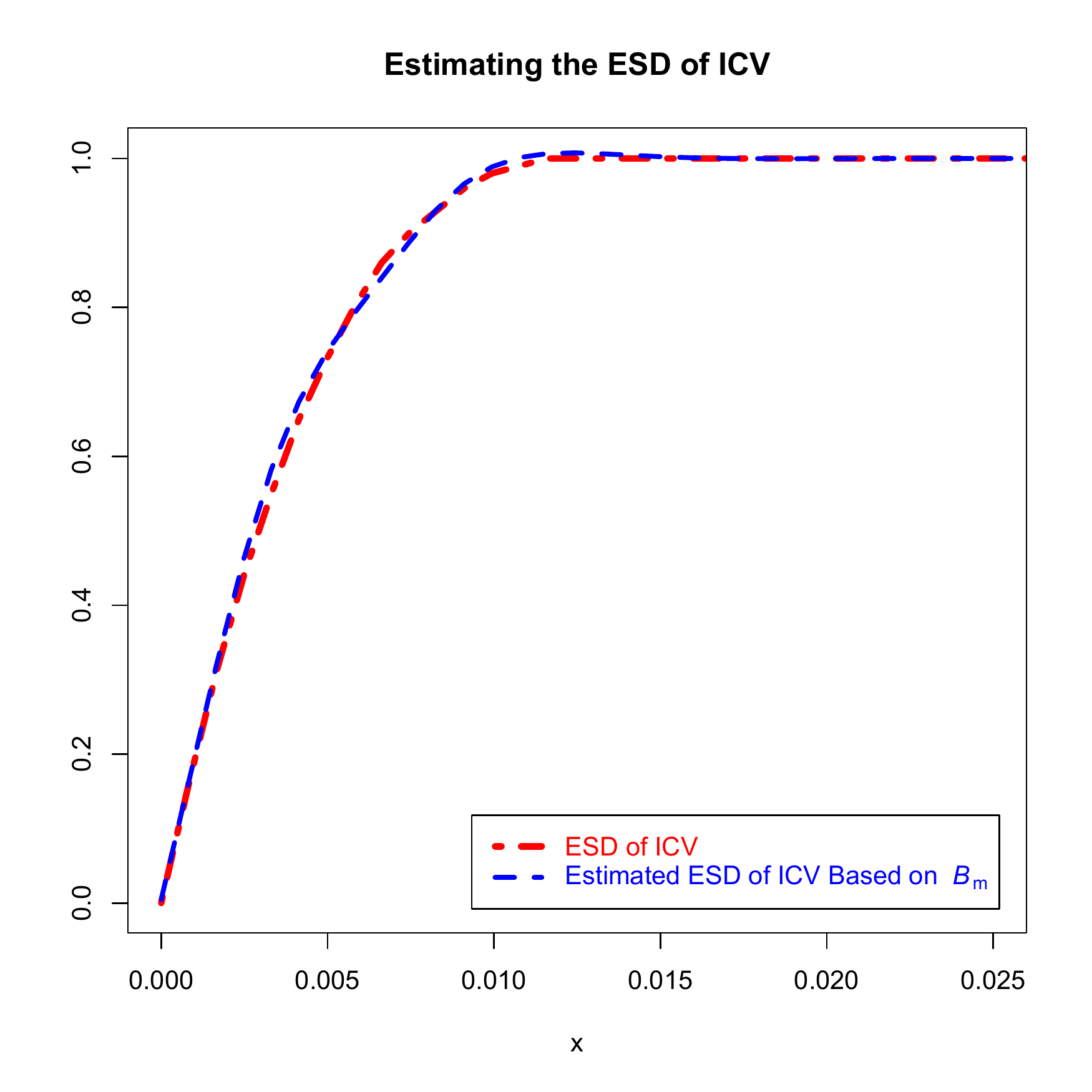}
\caption{Approach II: Estimation of the empirical spectral distribution of the targeting ICV matrix based on synchronous noisy observations under model \eqref{sim_setup:syn}. The dimension  $p=100$ and the observation frequency $n=23400$.}
\label{fig:Thm3}
\end{figure}

\subsection{When observations are asynchronous}\label{ssec:sim_asyn}
We now consider a setting where the observations are asynchronous. More specifically, for each stock $j=1,\ldots,p$, we simulate a Poisson process of rate
23,400 denoted by $\{\eta_i^j\}_{i=0,1,\ldots}.$ Because the Poisson processes $\{\eta_i^j\}_{i=0,1,\ldots}$ are to be generated independently, almost surely,
$\eta_{i_1}^{j_1} \neq \eta_{i_1}^{j_2}$ for all $1\leq j_1\neq j_2\leq p$ and $0<i_1, i_2$, namely, the observation times are all different for different stocks.
Figure \ref{fig:poisson} in the supplementary article  \cite{XZ15_supp} shows the observation times for three stocks generated in such a way during the first ten seconds.  The observation times are highly irregularly spaced:  there can be several seconds without a single observation, while there can also be several observations within a single second. Furthermore, because observation times for different stocks are generated independently, different stocks are observed in a rather unsynchronized manner, making the estimation of covariances difficult. For this reason, a synchronization procedure needs to be carried out before we apply either Approach I or II.

Before we discuss how to synchronize data, we first continue with the simulation design. To generate the latent process $(X_t)$, because of the asynchronicity and high-dimensionality (we are dealing with $p=100$ independent Poisson processes and, consequently, roughly $23,400 \times 100$ distinct observation times), there is a real technical difficulty in incorporating interactions among component processes in the data generating process. We adopt the following simplified setting to facilitate the simulation. Observe that the results in the previous subsection are achieved when the component processes have dependence, so we believe our mathods would still work when there is asynchronicity \emph{and} dependence. The simplified setting is as follows:
\begin{equation}\label{sim_setup:asyn}
d\X_t =\phi_t  \mathbf{D}^{1/2}\, d \W_t,
\end{equation}
where both $(\phi_t)$ and $\mathbf{D}$ are as in the previous subsection.

Our observations are
\[
   Y^j_{\eta_i^j} =  X^j_{\eta_i^j} + \eps^j_i,
\]
where $\eps^j_i$ are \hbox{i.i.d.} $N(0,0.0002)$.

Now we discuss how to synchronize data. We adopt the previous tick method explained in Section \ref{ssec:asyn}. More specifically, we choose an equally spaced grid $\{t_i\}\subseteq[0,1]$, and for each $t_i$, for each $j=1,\ldots,p$,  let
\[
\tau_i^j = \max\{\eta_k^j:  \eta_k^j \leq t_i\}.
\]
We then proceed as if we observe $Y^j_{\tau_i^j}$ at time $t_i$. As we explained in Section~\ref{ssec:asyn}, because $\tau_i^j \neq t_i$, such a synchronization procedure introduces an additional error $X^j_{\tau_i^j} - X^j_{t_i}$.

\subsubsection{Estimation using Approach I: asynchronous setting} \label{sssec:sim_asyn_I}
The additional error $X^j_{\tau_i^j} - X^j_{t_i}$ that the synchronization procedure induces depends on the latent process. For this reason, our independence assumption between the noise and the latent process (\ref{asm:noise}) is violated. To alleviate this problem, we synchronize less frequently so that the signals $X^j_{\tau_i^j}-X^j_{\tau_{i-1}^j}$ tend to be bigger and better approximate the true signals $X^j_{t_i}-X^j_{t_{i-1}}$. More specifically, we choose the equally spaced grid to be
$\{t_i=4i/23,400 \}$ -- in other words, we synchronize  once  every four seconds. Then, following the  estimation procedure in Section~\ref{ssec:sim_App_I}, we have the following results.

\begin{figure}[H]\centering
\includegraphics[width=.43\textwidth]{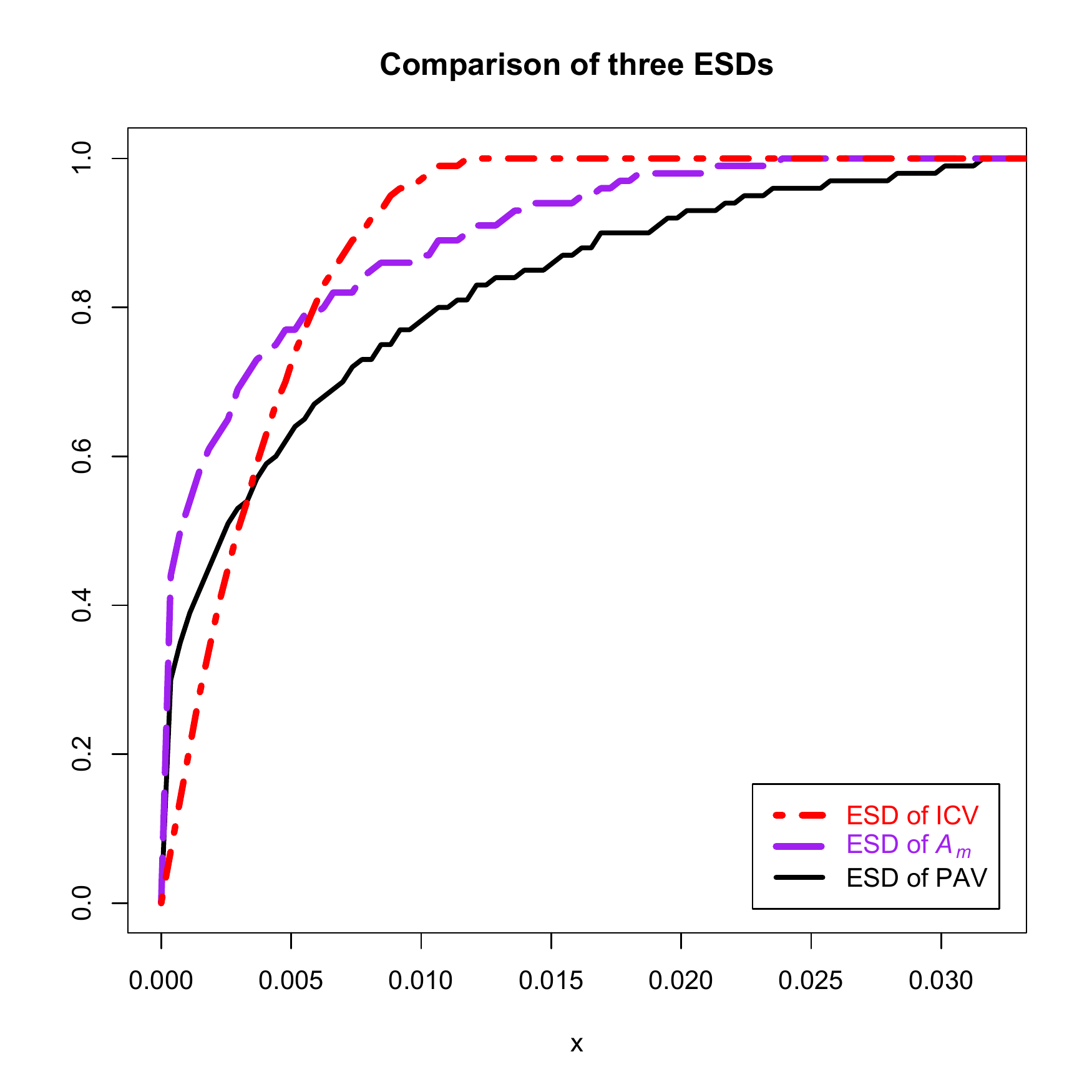}
\includegraphics[width=.43\textwidth]{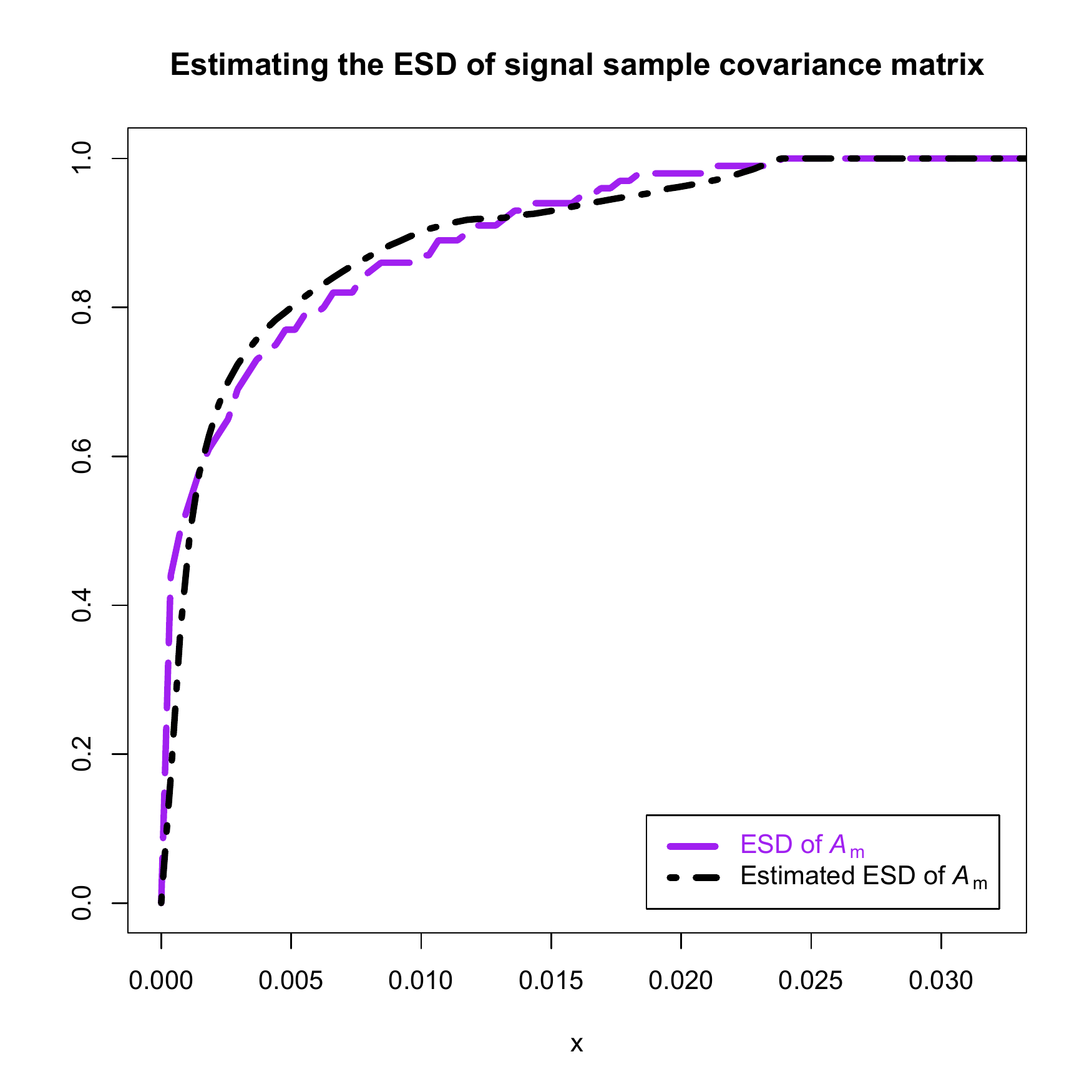}
\caption{Approach I -- Step 1: Estimation of the empirical spectral distribution of the signal sample covariance matrix $\cA_m$ based on asynchronous noisy observations under model~\eqref{sim_setup:asyn}. The dimension is $p=100$. The synchronization frequency is 4 seconds, which leads to $n=23400/4=5850$ observations. }
\label{fig:Thm1-2-Asyn}
\end{figure}
\begin{figure}[H]\centering
\includegraphics[width=.43\textwidth]{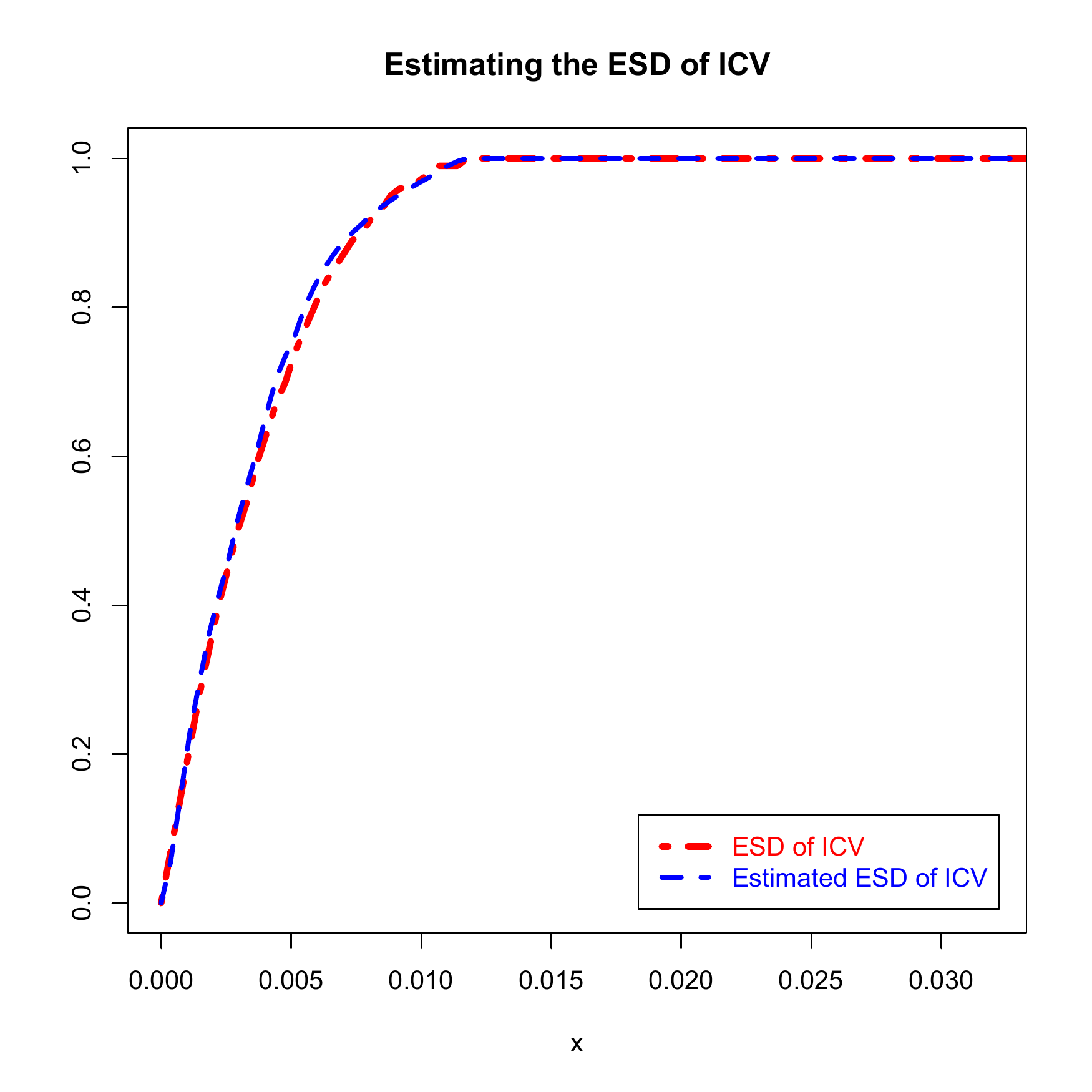}
\caption{Approach I -- Step 2: Estimation of the empirical spectral distribution of the targeting ICV matrix.}
\label{fig:Thm1-2-2-Asyn}
\end{figure}

We see that in such a highly asynchronous noisy observation setting, Approach~I still works quite well.

\subsubsection{Estimation using Approach II: asynchronous setting} \label{sssec:sim_asyn_II}
Approach II relies on Theorem \ref{thm:B_n}, which allows for dependence in the noise process and between the  noise and  price process. For this reason, it is more robust than Approach I, and we can synchronize more frequently. In the estimation below
we choose to synchronize once every second; that is, the time grid is taken to be $\{t_i=i/23400\}$.
Then,  following the  estimation procedure in Section \ref{ssec:sim_App_II}, we obtain the following results.

\begin{figure}[H]\centering
\includegraphics[width=.43\textwidth]{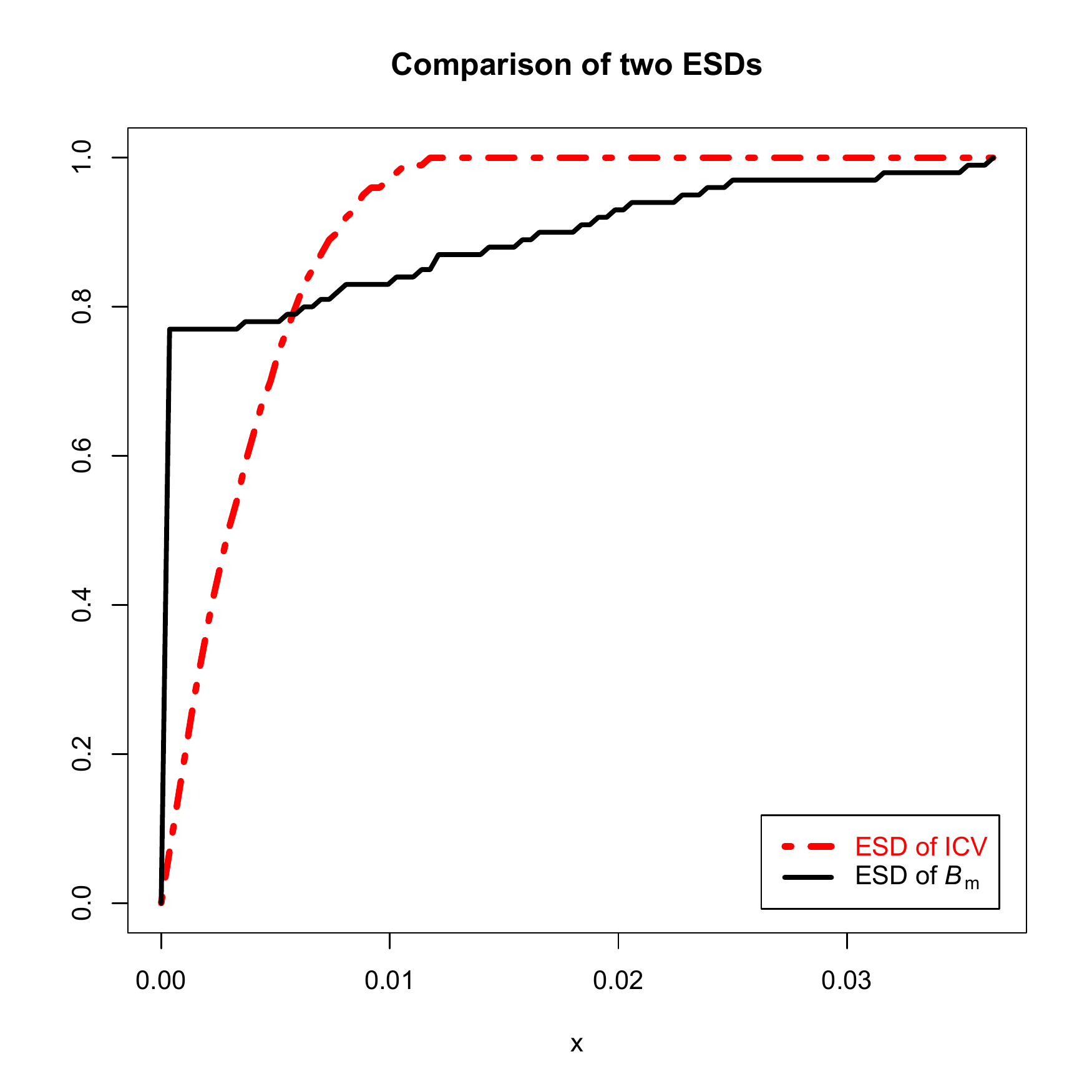}
\includegraphics[width=.43\textwidth]{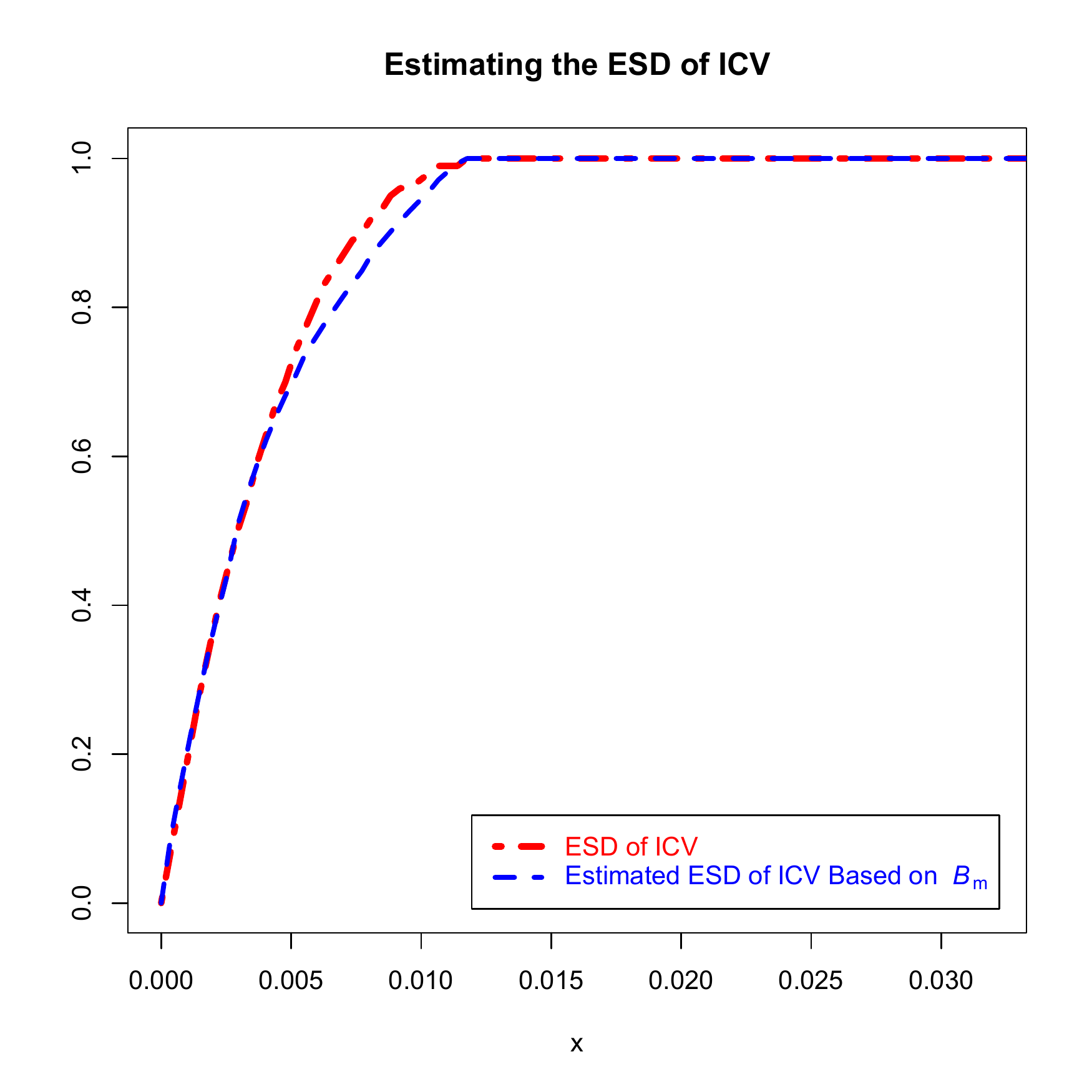}
\caption{Approach II: Estimation of the empirical spectral distribution of the targeting ICV matrix based on asynchronous noisy observations under model \eqref{sim_setup:asyn}. The dimension is $p=100$. The synchronization frequency is  one second, which leads to $n=23,400$ observations. }
\label{fig:Thm3-Asyn}
\end{figure}

Again, we see that in such an asynchronous noisy observation setting, Approach~II works quite well.

\subsection{Discussions about the two approaches}
The two approaches have their own pros and cons.
\begin{itemize}
\item Approach I is more widely applicable to noisy observation situations. Moreover, in our particular application, because the window width in defining the \PAV\ matrix is $O(\sqrt{n})$, which is of lower order than that for $\cB_m$ in Approach II, Approach I essentially has a larger ``effective sample size.'' This approach, however, is more sensitive to the model assumptions. In particular, in the asynchronous observation setting, because of the additional error introduced by asynchronicity, we may need to synchronize less frequently.
\item Approach II is more direct because it only involves a one-step estimation procedure. It is also more robust because it allows for rather general dependence structures in the noise process, both cross-sectional and temporal, and even dependence between the  noise and  price process.  For this reason, in the asynchronous setting in Section \ref{ssec:sim_asyn}, we can use a higher synchronization frequency than for Approach I.
A major drawback of Approach II is that it relies heavily on some special properties of the particular setting under study and hence may not be applicable to other noisy observation situations.
\end{itemize}

Finally, while in the estimation above we largely adapt the algorithms proposed by \cite{ElKaroui08} to fit our setting,
other algorithms such as those in \cite{Mestre08}, \cite{BCY10} and \cite{LW2015}  can  also be adapted.


\section{Conclusion\label{sec:conclusion} }
Motivated by the inference about the spectral distribution of the \ICV\  matrix based on high-frequency noisy data,
\begin{compactitem}
\item we establish an asymptotic relationship that describes how the spectral distribution of the signal sample  covariance matrices depends on that of the sample covariance matrices constructed from noisy observations;
\item using further a (generalized) connection between the spectral distribution of the signal sample covariance matrices and that of the population covariance matrix, we propose a two-step procedure that can consistently estimate the spectral distribution of ICV for a class  of diffusion processes;
\item  we further develop an alternative approach  that possesses several desirable properties: it is more robust, it eliminates the effects of microstructure noise, and the asymptotic relationship that enables the consistent estimation of the  spectral distribution of ICV  is the standard Mar\v{c}enko-Pastur equation.
 \item numerical studies demonstrate that our proposed methods work well, under both synchronous and asynchronous observation settings.
\end{compactitem}


\section*{Acknowledgements}
We are very grateful to the editor, the associate editor, and the anonymous referees for their  valuable comments and constructive suggestions, which led to a substantial improvement of this paper.

\bigskip

\noindent
Ningning Xia: School of Statistics and Management, Shanghai, Key Laboratory of Financial Information Technology, Shanghai University of Finance and Economics,
777 Guo Ding Road, China, 200433.  xia.ningning@mail.shufe.edu.cn

\bigskip

\noindent
Xinghua Zheng: Department of Information Systems Business Statistics and Operations
Management, Hong Kong University of Science and Technology, Clear Water Bay, Kowloon,
Hong Kong. xhzheng@ust.hk



\newpage
\setcounter{page}{1}
\begin{supplement}[id=suppA]
  \stitle{Supplement to ``On the inference about the spectral distribution of high-dimensional covariance matrix based on high-frequency noisy observations
''}

The proposition/lemma/equation \hbox{etc.} numbers below refer to the main article \cite{XZ15}.
\appendix

\renewcommand{\baselinestretch}{1.2}
\setcounter{equation}{0}
\renewcommand{\theequation}{\thesection.\arabic{equation}}

\section{Proof of Theorem~\ref{THM1}}\label{appendix:pf_thm_LSD_signal_noise}

Theorem \ref{thm:LSD_signal_noise} is a consequence of the following proposition.

\begin{prop}\label{prop:LSD_signal_noise}
Under the assumptions of Theorem \ref{thm:LSD_signal_noise}, there exists  $K^*>0$ such that almost surely, for all
\[
z\in\mathbb{C}^{*}:=\left\{z\in\mathbb{C}^+: \Im(z)> K^*\right\},
\]
we have
\begin{equation}\label{eqn:prop_LSD_signal_noise}
\lim_{p\to\infty}\left[
\dfrac{1}{p}{\rm tr}\left(\dfrac{1}{1+\delta_n}\cA_n-z\I\right)^{-1}
-\dfrac{1}{p}{\rm tr}\bigg({\bf S}_n-(z-t_n\sigma_n^2)\I\bigg)^{-1}
\right]=0,
\end{equation}
where for all sufficiently large $p$,  $t_n$ is the unique solution to the equation
\begin{equation}\label{eqn:t_n}
t_n \ = \ y_n-1+y_n(z-t_n\si_n^2) \ \dfrac{1}{p}{\rm tr}\(\S_n-(z-t_n\si_n^2)\I \)^{-1},
\end{equation}
in the set
\begin{equation}\label{dfn:D}
\sD:=\left\{t\in\bC: 0\le\Im(t)\le \dfrac{\Im(z)}{2(\si+1)^2}\right\},
\end{equation}
and
\begin{equation}
\de_n \ = \ y_n\si_n^2 \ \dfrac{1}{p}{\rm tr}\(\S_n-(z-t_n\si_n^2) \I \)^{-1}.
\label{delta}
\end{equation}
\end{prop}

In Section \ref{ssec:pf_thm_LSD_signal_noise} below, we show how Proposition \ref{prop:LSD_signal_noise} leads to Theorem \ref{thm:LSD_signal_noise}.

As to the proof of Proposition \ref{prop:LSD_signal_noise}, we shall use the following results from  \cite{DS2007a}.
By Theorem 1.1 therein, the sequence $\{F^{\S_n}\}$ converges weakly to a probability distribution  $F$.  Moreover, by using the same truncation and centralization technique as in \cite{DS2007a}, we may assume that
\begin{compactenum}\setcounter{enumi}{5}
 \item[]
 \begin{compactenum}
  \item\label{asm:eps_bd} $|\ep_{11}|\le a\log(n)$ for some $a>2$,
  \item\label{asm:eps_mean_var} $E\ep_{11}=0$, $E|\ep_{11}|^2=1$, \q and
  \item\label{asm:A_bd}  $\|(1/n)\A_n\A_n^T\|\le \log(n)$.\\
\end{compactenum}
\end{compactenum}

In addition to equation \eqref{eqn:t_n}, we shall also study its limiting equation
\begin{equation}\label{eqn:t}
t= \ y-1+y(z-t\sigma^2)m(z-t\sigma^2),
\end{equation}
where $m(\cdot)$ is  the Stieltjes transform of  the probability distribution  $F$.

It is shown in \cite{DS2007b} that the distribution $F$ admits a continuous density on $\zz{R}\setminus\{0\}$.  Because we assume that
$F$ is supported by a finite interval $[a,b]$ with $a>0$ and possibly has a point mass at~0, we conclude that $F$ admits a bounded density~$f$ supported by  $[a,b]$ and possibly a point mass at zero.

\subsection{Properties of $t_n$ and $t$}\label{ssec:t_n}

\begin{lem}\label{lem:tn_exist}
There exists  $K_1>0$ such that for all
$z\in\bC_1:=\left\{z=u+iv:  v> K_1 \right\}$,
for all sufficiently large $n$, equation (\ref{eqn:t_n}) admits a unique solution in $\sD$.
\end{lem}
\begin{proof}
Rewrite equation (\ref{eqn:t_n}) as
\begin{equation}\label{re_t}
\aligned
t_n+1 &=
y_n+y_n\int \ \dfrac{z-t_n\sigma_n^2}{x-z+t_n\sigma_n^2} \ dF^{S_n}(x) \\
&= \ y_n\int \ \dfrac{x}{x-z+t_n\sigma_n^2} \ dF^{S_n}(x).
\endaligned
\end{equation}

First, under the assumptions of Theorem \ref{thm:LSD_signal_noise}, by Theorem 1.1 in \cite{Bai2012_noise}, if we let $[a_n,b_n]$ be an interval containing the support of~$F^{\S_n}$, then we may assume that  for all large $n$, $b_n\leq \wt{b}:=b+1$.
Let $\wt{\si}=\si+1$, $\wt{y}=y+1$  and $K_1=2\wt{\si}\sqrt{\wt{y}\wt{b}}$. Because $\si_n\to\si$ and $y_n\to y$, we have for all large $n$ and for all $t=t_1 + it_2\in\sD$,
\begin{equation}\label{eqn:v_t_lower_bd}
\si_n< \wt{\si}, ~ y_n< \wt{y},~~~ {\rm and}~v-t_2\si_n^2\ge v-t_2\wt{\si}^2\geq v/2>0.
\end{equation}

Define
\[
G(t) \ = \ y_n \int \ \dfrac{x}{x-z+t\sigma_n^2} \ dF^{S_n}(x)-1, ~~~~~ {\rm for ~all}~~ t\in \mathscr{D}.
\]
We apply the Banach fixed point theorem to show that  for all sufficiently large $n$, there exists a unique point $t^*\in \mathscr{D}$ such that $G(t^*)=t^*$. The desired conclusion then follows.

\ul{Step (i)}: we prove that the mapping $G$ is defined from $\mathscr{D}$ to $\mathscr{D}$. From the definition of $G(t)$ and that of $t\in\sD$,   we have
\begin{eqnarray*}
\Im(G(t))&=&y_n\int_{a_n}^{b_n}\frac{x(v-t_2\sigma_n^2)}{(x-u+t_1\sigma_n^2)^2+(v-t_2\sigma_n^2)^2} \ dF^{S_n}(x)\\
&=&\dfrac{y_n}{v-t_2\sigma_n^2}\int_{a_n}^{b_n}\dfrac{x}{1+\left(\frac{x-u+t_1\sigma_n^2}{v-t_2\sigma_n^2}\right)^2} \ dF^{S_n}(x),
\end{eqnarray*}
and hence for all sufficiently large $n$, by \eqref{eqn:v_t_lower_bd}, we have
\begin{eqnarray*}
0<\Im(G(t))
< \dfrac{\wt{y}\wt{b}}{v-t_2\wt{\si}^2}
\le \dfrac{v}{2\wt{\si}^2},
\end{eqnarray*}
where the last inequality follows from the fact that for any  $z\in\bC_1$,
\[
\dfrac{\wt{y}\wt{b}}{v-t_2\wt{\si}^2}-\dfrac{v}{2\wt{\si}^2}
\le \dfrac{2\wt{y}\wt{b}}{v}-\dfrac{v}{2\wt{\si}^2}
= \dfrac{4\wt{\si}^2\wt{y}\wt{b}-v^2}{2\wt{\si}^2v}
\le 0.
\]

\ul{Step (ii)}: we shall show that $G: \mathscr{D} \to \mathscr{D}$ is a contraction mapping. In fact, for any two points $t$, $t^{\prime} \ \in \mathscr{D}$,
\begin{eqnarray*}
G(t)-G(t^{\prime})&=& y_n\int_{a_n}^{b_n} \
\left(\dfrac{x}{x-z+t\sigma_n^2}-\dfrac{x}{x-z+t^{\prime}\sigma_n^2}\right)
\ dF^{S_n}(x)\\
&=& (t-t^{\prime}) \ y_n\sigma_n^2 \int_{a_n}^{b_n} \
\dfrac{-x}{(x-z+t\sigma_n^2)(x-z+t^{\prime}\sigma_n^2)}
\ dF^{S_n}(x)\\
&:=& (t-t^{\prime}) \ q(t,t^{\prime}).
\end{eqnarray*}
By the Cauchy-Schwartz inequality and \eqref{eqn:v_t_lower_bd} we get that  for all sufficiently large $n$, for all $t,t^{\prime}\in\mathscr{D}$,
\[
\aligned
|q(t,t^{\prime})|
\le& \left(\int_{a_n}^{b_n} \ \dfrac{\sigma_n^2 y_n x}{|x-z+t\sigma_n^2|^2} \ dF^{S_n}(x)\right)^{1/2}
\left(\int_{a_n}^{b_n} \ \dfrac{\sigma_n^2 y_n x}{|x-z+t^{\prime}\sigma_n^2|^2} \ dF^{S_n}(x)\right)^{1/2}\\
\le& \left(\dfrac{\sigma_n^2y_nb_n}{(v-\Im(t)\sigma_n^2)^2}\right)^{1/2}
\left(\dfrac{\sigma_n^2y_nb_n}{(v-\Im(t^{\prime})\sigma_n^2)^2}\right)^{1/2}\\
<&\left(\dfrac{\wt{\si}^2\wt{y}\wt{b}}{(v-\Im(t)\wt{\si}^2)^2} \right)^{1/2} \left(\dfrac{\wt{\si}^2\wt{y}\wt{b}}{(v-\Im(t^{\prime})\wt{\si}^2)^2} \right)^{1/2}\\
\le&\left(\dfrac{\wt{\si}^2\wt{y}\wt{b}}{v^2/4} \right)^{1/2} \left(\dfrac{\wt{\si}^2\wt{y}\wt{b}}{v^2/4} \right)^{1/2},
\endaligned
\]
which is strictly smaller than 1 when $z\in\bC_1$.
Therefore, the mapping $G$ is contractive  in~$\mathscr{D}$, and the Banach fixed point theorem guarantees the existence of a unique solution to equation (\ref{eqn:t_n}).
\end{proof}

\begin{lem}\label{lem:t_properties}
Suppose that $t$ solves equation \eqref{eqn:t} for $z\in\bC^+$. Write $t=t_1+it_2$ and $z=u+iv$.
Then, $0<t_2<v/\si^2$; moreover,
as $v\to\infty$, uniformly in $u$, one has $t_2\to 0$ and $t_1\to -1$.
\end{lem}
\begin{proof}
Taking imaginary parts on both sides of equation (\ref{eqn:t}) yields
\begin{eqnarray}\label{expression_t2}
t_2= y \int_a^b \dfrac{x(v-t_2\si^2)}{|x-z+t\si^2|^2} \ dF(x).
\end{eqnarray}
It is then straightforward to verify that
$t_2>0$ and $v-t_2\si^2>0$.
Furthermore, since
\begin{align}
\label{sim_t2}
t_2=&\dfrac{y}{v-t_2\si^2} \ \int_a^b \dfrac{x}{1+\(\dfrac{x-u+t_1\si^2}{v-t_2\si^2}\)^2} \ dF(x)\\
\notag\le&\dfrac{yb}{v-t_2\si^2},
\end{align}
when $v\geq 2\si\sqrt{yb}$, we have
\begin{equation}\label{eq:t_2_possibilities}
{\rm either} ~~~~~
t_2 \ \geq  \ \dfrac{ \ v \ + \ \sqrt{v^2-4\sigma^2 yb}}{2\sigma^2} ~~~~~ \textrm{or} ~~~~~
t_2 \ \le \ \dfrac{ \ v \ - \ \sqrt{v^2-4\sigma^2 yb}}{2\sigma^2}.
\end{equation}

Denote $w=u-t_1\si^2$ and $\th=v-t_2\si^2$.
By \eqref{sim_t2}, if $F$ admits a bounded density $f$ and possibly a point mass at 0, then
\begin{eqnarray*}
t_2&=&\dfrac{y}{\theta} \ \int_a^b \ \dfrac{x}{ 1+\(\dfrac{x-w}{\theta}\)^2 } \  f(x) \ dx\\
&=& y \ \int_{\frac{a-w}{\theta}}^{\frac{b-w}{\theta}} \ \dfrac{ \ w+\theta l \ }
{ \ 1+l^2 \ } \ f(w+\theta l) \ dl.
\end{eqnarray*}
Because $f(\cdot)$ is bounded  and $x=w+\theta l \ \in(a,b)$ when $l\in (\frac{a-w}{\th},\frac{b-w}{\th})$,  there exists a constant $C$ such that
\[
t_2 \ \le \ C\int_{\frac{a-w}{\theta}}^{\frac{b-w}{\theta}} \
\dfrac{1}{1+l^2} \ dl \  \le \ C\int_{-\infty}^{+\infty} \ \dfrac{dl}{1+l^2} \ = \ C \ \pi.
\]
This, combined with \eqref{eq:t_2_possibilities}, implies that
\begin{eqnarray}\label{eqn:t2_upper_bd}
t_2 \ \le \ \dfrac{v \ - \ \sqrt{v^2-4\sigma^2 yb}}{2\sigma^2},\q\mbox{for all } v \mbox{ large enough}.
\end{eqnarray}
In particular,  uniformly in $u$,
\begin{equation}\label{eqn:t2_to_0_as_v_infty}
t_2\to 0 \mbox{ and } v-t_2\sigma^2\to\infty,\q \mbox{ as }v\to\infty.
\end{equation}
Moreover, from (\ref{eqn:t}) we get
\[
t+1 \ = \ y \ + \ y\int \ \dfrac{z-t\sigma^2}{ \ x-z+t\sigma^2 \ } \ dF(x) \
=y\int\dfrac{x}{ \ x-z+t\sigma^2 \ }  \  dF(x).
\]
Thus as $v\to\infty$,
\begin{eqnarray*}
|t_1+1| \le |t+1|
\le  y\int_a^b\dfrac{x}{ | \Im(x-z+t\sigma^2) | } \ dF(x)
\le  \dfrac{C}{v-t_2\sigma^2} \to ~ 0,
\end{eqnarray*}
also uniformly in $u$.
\end{proof}

\begin{lem}\label{lem:t_exist}
There exists $K_2\geq K_1$ such that for any $z\in\bC_2:=\left\{z=u+iv: v>K_2\right\}$, equation \eqref{eqn:t} admits a unique solution.
\end{lem}
\begin{proof}
First, by the same proof as for Lemma \ref{lem:tn_exist}, one can show that for all $z=u+iv$ with $v\geq K_1$,  equation \eqref{eqn:t} admits a unique solution in $\sD$. Moreover, by Lemma \ref{lem:t_properties}, if $t=t_1+it_2$ solves \eqref{eqn:t}, then $t_2>0$; furthermore, we can find a constant $K_2$ such that
if $t$ solves~\eqref{eqn:t} for~$z$ with $v(=\Im(z))\geq K_2,$ then we must have $t_2\le{v}/{(2\wt{\si}^2)}$. The latter two properties imply that for all  $z$ with $v\geq K_2,$ the solution to \eqref{eqn:t} must lie in $\sD$. Redefining $K_2=\max(K_1,K_2)$ if necessary, we see that for all $z\in\bC_2$,  \eqref{eqn:t} admits a unique solution.
\end{proof}

\begin{lem}\label{lem:t_analytic}
There exists $K_3\geq K_2$ such that the solution $t=t(z)$ to \eqref{eqn:t} is analytic on $\bC_3:=\left\{z=u+iv: v>K_3\right\}$.
\end{lem}
\begin{proof}
Define a function $G$ as
\[
G(z,t)=t-(y-1)-y(z-t\si^2)m(z-t\si^2), \ (z,t)\in \bC^+\times\bC^+ \mbox{ with } \Im(z-t\si^2)>0.
\]
That $t(z)$ solves \eqref{eqn:t} is equivalent to $G(z,t(z))=0$.
Write $z=u+iv$ and $t=t_1+it_2$.
By taking the partial derivative with respect to $t$ we get
\[
\dfrac{\partial G}{\partial t}=1+y\si^2
\int\dfrac{x}{\(x-(z-t\si^2)\)^2} ~ dF(x).
\]
Note that
\[
\left|\int\dfrac{x}{\(x-(z-t\si^2)\)^2} ~ dF(x)\right| ~
\le ~ \dfrac{b}{(v-t_2\si^2)^2},
\]
which, by \eqref{eqn:t2_to_0_as_v_infty}, goes to zero as $v\to\infty$.
Thus there exists a constant $K_3>0$ such that for all $z\in \bC_3$, ${\partial G}/{\partial t}(z,t(z))\neq 0$. It follows from the implicit function theorem and Lemma~\ref{lem:t_exist}
that $t=t(z)$ is  analytic on $\bC_3$.
\end{proof}

\begin{lem}\label{lem:tlimit}
Suppose that $t_n$ solves equation \eqref{eqn:t_n} for $z\in\bC_2$; then,  $\Im(t_n)>0$ and $\Im(z-t_n\si_n^2)>0$. moreover, if $t_n$ is the unique solution in the  set $\sD$, then,  with probability one, as $n\to\infty$, $t_n$ converges to a nonrandom complex number $t$ that uniquely solves equation \eqref{eqn:t}.
\end{lem}
\begin{proof}
Write $z=u+iv$ and $t_n=t_{n1}+it_{n2}$.
Similar to the proof of Lemma \ref{lem:t_properties}, taking imaginary parts on both sides of equation (\ref{eqn:t_n}), one can easily show that
$t_{n2}>0$ and $v-t_{n2}\si_n^2>0$.

Next we show that $\{t_n\}$ is tight; in other words, for any $\eps>0$, there exists $C>0$, such that for all sufficiently large $n$,  $P\left(|t_n|>C\right)<\eps$.
Because $0<t_{n2}<v/\sigma_n^2$, it suffices to show that $\{|t_{n1}|\}$ is tight.

Let $\ul{\S}_n= \dfrac{1}{n}(\A_n+\si_n\bE_n)^T(\A_n+\si_n\bE_n)$, and let $\ul{m}_n(z)$ be the Stieltjes transform of the ESD $F^{\ul{\S}_n}$. The spectra of $\S_n$ and $\ul{\S}_n$ differ by $|p-n|$ number of zero eigenvalues;  hence,
$
F^{\ul{\S}_n}=(1-y_n) I_{[0,\infty)} +y_n F^{\S_n},
$
and
\begin{eqnarray}
\ul{m}_n(z)=-\dfrac{1-y_n}{z}+y_n m_n(z).
\label{mmn}
\end{eqnarray}
Thus, equation (\ref{eqn:t_n}) can also be expressed as
\begin{eqnarray*}
t_n&=&y_n-1+y_n(z-t_n\si_n^2)m_n(z-t_n\si_n^2) \nonumber\\
&=&(z-t_n\si_n^2)\ul{m}_n(z-t_n\si_n^2).
\end{eqnarray*}
Taking real parts on both sides yields
\[
\Re(t_{n})=\int\dfrac{x(u-\Re(t_{n})\si_n^2)-|z-t_n\si_n^2|^2}{|x-z+t_n\si_n^2|^2} \ dF^{\ul{\S}_n}(x).
\]
Solving for $\Re(t_{n})$ yields
\begin{equation}\label{re_tn}
\Re(t_{n}) = \dfrac{\displaystyle\int \dfrac{xu-|z-t_{n}\si_{n}^2|^2}{|x-z+t_{n}\si_{n}^2|^2}\ dF^{\ul{\S}_{n}}(x)}
{1+\si_{n}^2\displaystyle\int \dfrac{x}{|x-z+t_{n}\si_{n}^2|^2}\ dF^{\ul{\S}_{n}}(x)}
\end{equation}

Now suppose that $\{t_{n1}=\Re(t_n)\}$ is not tight; then, with positive probability, there exists a subsequence $\{n_k\}$ such that $|\Re(t_{n_k})|\to\infty$.
By \eqref{re_tn}, we have
\begin{eqnarray*}
|\Re(t_{n_k})|
&\le& \int_{a_{n_k}}^{b_{n_k}} \dfrac{x|u|+|z-t_{n_k}\si_{n_k}^2|^2}{|x-z+t_{n_k}\si_{n_k}^2|^2} ~ dF^{\ul{\S}_{n_k}}(x).
\end{eqnarray*}
However, as $k$ goes to infinity, if  $|\Re(t_{n_k})|\to\infty$, because  $\{F^{\ul{\S}_{n_k}}\}$ is tight and $\si_{n_k}\to\si>0$, one gets that the RHS goes to 1. This contradicts the supposition that $|\Re(t_{n_k})|\to\infty$.

Next, for any convergent subsequence $\{t_{n_k}\}$ in  set $\sD$, by \eqref{eqn:v_t_lower_bd}, for all sufficiently large $n_k$, we have $v-\Im(t_{n_k})\si_{n_k}^2\geq v/2$. We can then apply the dominated convergence theorem  to conclude that the limit point of $\{t_{n_k}\}$ must satisfy equation \eqref{eqn:t}.
By Lemma \ref{lem:t_exist}, the solution is unique; hence, the whole sequence $\{t_n\}$ converges to the unique solution to equation~\eqref{eqn:t}.
\end{proof}


\subsection{Some further preliminary results}\label{ssec:prelim}
Let $K^*=\max\{K_1,K_2,K_3\}$$(=K_3)$ for  $K_1$, $K_2$ and $K_3$ as defined in Lemmas \ref{lem:tn_exist}, \ref{lem:t_exist} and \ref{lem:t_analytic}, respectively. Also define
$\bC^*=\{z\in\bC^+: \Im(z)>K^*\}$.
Below we work with $z\in\bC^*$.

Let ${\bf a}_j$ and $\pep_j$, $j=1,\ldots,n$, be the $j$th column of ${\bf A}_n$ and  $\bE_n$, and let ${\bf b}_j=\sigma_n \pep_j$.
Denote ${\pmb {\pmb\xi}}_j=({\bf a}_j+{\bf b}_j)/\sqrt{n}$ so that ${\bf S}_n=\sum_{j=1}^n{\pmb {\pmb\xi}}_j{\pmb {\pmb\xi}}_j^T$.
For any complex number $t_n$ such that $\Im(z-t_n\si_n^2)> 0$,  define
\begin{eqnarray}
&&~~~{\bf R}_n  =  {\bf S}_n-(z-t_n\sigma_n^2)  \I ,\q\q\q\q \delta_n =\dfrac{\sigma_n^2}{n}{\rm tr}({\bf R}_n^{-1})
=y_n\si_n^2\dfrac{1}{p}{\rm tr}(\R_n^{-1}), \nonumber\\
&&~~~{\bf S}_{nj}  =  {\bf S}_n-{\pmb \xi}_j{\pmb \xi}_j^T \ = \ \sum_{k\neq j}{\pmb\xi}_k{\pmb\xi}_k^T , ~~\q
{\bf R}_{nj}  =  {\bf S}_{nj}-(z-t_n\sigma_n^2)  \I, \label{betaj}\\
&&~~~{\bf B}_n  =  \dfrac{1}{1+\delta_n}\dfrac{1}{n}{\bf A}_n{\bf A}_n^T-z  \I, \mbox{ and }
~ \beta_j  =  \dfrac{1}{ \ 1+{\pmb\xi}_j^T{\bf R}_{nj}^{-1}{\pmb\xi}_j \ }. \nonumber
\end{eqnarray}

According to equation (2.2) in \cite{SB95}, we have
\begin{eqnarray}\label{xiRn}
{\pmb\xi}_j^T \ \R_n^{-1}
= \ \dfrac{{\pmb\xi}_j^T{\bf R}_{nj}^{-1}}{ \ 1+{\pmb\xi}_j^T{\bf R}_{nj}^{-1}{\pmb\xi}_j \ }=\beta_j{\pmb\xi_j}^T\R_{nj}^{-1}.
\end{eqnarray}
Thus,  using the identity
$
\A^{-1}-\B^{-1}=\A^{-1}(\B-\A)\B^{-1},
$
we obtain
\begin{eqnarray}
{\bf R}_n^{-1} \ = \ {\bf R}_{nj}^{-1}-{\bf R}_n^{-1}{\pmb\xi}_j{\pmb\xi}_j^T{\bf R}_{nj}^{-1} \
= \ {\bf R}_{nj}^{-1}-\beta_j{\bf R}_{nj}^{-1}{\pmb\xi}_j{\pmb\xi}_j^T{\bf R}_{nj}^{-1}.
\label{RRbeta}
\end{eqnarray}

Next, we introduce another definition of $t_n$ as the solution to the following equation
\begin{eqnarray}\label{dfn:t_alternative}
t_n = \ -\dfrac{1}{n}\sum_{j=1}^n \beta_j \
= \ -\dfrac{1}{n}\sum_{j=1}^n \ \dfrac{1}{ \ 1+{\pmb\xi}_j^T{\bf R}_{nj}^{-1}{\pmb\xi}_j \ }.
\end{eqnarray}
We claim that the definition of $t_n$ in (\ref{dfn:t_alternative}) is equivalent to the earlier definition of defining $t_n$ to be the solution to equation \eqref{eqn:t_n}. In fact, 
write
\[
\R_n+z  \I \ = \ \sum_{j=1}^n \ {\pmb\xi}_j{\pmb\xi}_j^T+t_n\sigma_n^2 \ \I.
\]
Right-multiplying both sides by ${\bf R}_n^{-1}$ and using \eqref{xiRn} yield
\[
\I+z \ {\bf R}_n^{-1}=\sum_{j=1}^n \ {\pmb\xi}_j{\pmb\xi}_j^T{\bf R}_n^{-1}+t_n\sigma_n^2 \ {\bf R}_n^{-1}
=\sum_{j=1}^n \ \dfrac{{\pmb\xi}_j{\pmb\xi}_j^T{\bf R}_{nj}^{-1}}{ \ 1+{\pmb\xi}_j^T{\bf R}_{nj}^{-1}{\pmb\xi}_j \ } +t_n\sigma_n^2 \
 {\bf R}_n^{-1}.
\]
Taking the trace on both sides and dividing by $n$, one gets
\begin{eqnarray}\label{beta_tn}
y_n+z \ \dfrac{1}{n}{\rm tr}({\bf R}_n^{-1})
&=& 1-\dfrac{1}{n}\sum_{j=1}^n
\dfrac{1}{ \ 1+{\pmb\xi}_j^T{\bf R}_{nj}^{-1}{\pmb\xi}_j \ }
+t_n\sigma_n^2 \ \dfrac{1}{n}{\rm tr}({\bf R}_n^{-1}) \nonumber\\
&=&1-\dfrac{1}{n}\sum_{j=1}^n \beta_j
+t_n\sigma_n^2 \ \dfrac{1}{n}{\rm tr}({\bf R}_n^{-1}).
\end{eqnarray}
This shows that if $t_n$ satisfies \eqref{dfn:t_alternative}, then $t_n$ satisfies equation \eqref{eqn:t_n}.
On the other hand, if $t_n$ satisfies equation \eqref{eqn:t_n}, from \eqref{beta_tn} we have
\[
-\dfrac{1}{n}\sum_{j=1}^n \beta_j
= y_n-1+(z-t_n\si_n^2)\dfrac{1}{n}\tr(\R_n)^{-1}
= t_n,
\]
namely, $t_n$ satisfies \eqref{dfn:t_alternative}.

We proceed to analyze the difference in \eqref{eqn:prop_LSD_signal_noise}.
Because
\begin{eqnarray*}
{\bf S}_n-\dfrac{1}{1+\delta_n} \ \dfrac{1}{n}{\bf A}_n{\bf A}_n^T
&=&\dfrac{1}{n}\sum_{j=1}^n \ ({\bf a}_j+{\bf b}_j)({\bf a}_j+{\bf b}_j)^T-\dfrac{1}{1+\delta_n} \ \dfrac{1}{n}\sum_{j=1}^n  {\bf a}_j{\bf a}_j^T\\
&=&\dfrac{1}{n}\sum_{j=1}^n \ \left(\dfrac{\delta_n}{1+\delta_n} \ {\bf a}_j{\bf a}_j^T+{\bf a}_j{\bf b}_j^T+{\bf b}_j{\bf a}_j^T+{\bf b}_j{\bf b}_j^T\right),
\end{eqnarray*}
and recall that ${\bf R}_{n}  =  {\bf S}_{n}-(z-t_n\sigma_n^2)  \I$ and
${\bf B}_n  =  \dfrac{1}{1+\delta_n}\dfrac{1}{n}{\bf A}_n{\bf A}_n^T-z  \I$,
we have
\[
\aligned
\Delta
:=&\dfrac{1}{p}{\rm tr}\left({\bf B}_n^{-1} - {\bf R}_{n}^{-1}\right)\\
=&\dfrac{1}{p}{\rm tr}\({\bf B}_n^{-1} \left({\bf S}_n-\dfrac{1}{1+\delta_n}\dfrac{1}{n}{\bf A}_n{\bf A}_n^T+t_n\sigma_n^2 \I\right)
{\bf R}_{n}^{-1}\)\\
=&\dfrac{1}{np}\sum_{j=1}^n \ \Bigg\{
\dfrac{\delta_n}{1+\delta_n} \ {\bf a}_j^T{\bf R}_{n}^{-1}{\bf B}_n^{-1}{\bf a}_j
+ \ {\bf b}_j^T{\bf R}_{n}^{-1}{\bf B}_n^{-1}{\bf a}_j
+ \ {\bf a}_j^T{\bf R}_{n}^{-1}{\bf B}_n^{-1}{\bf b}_j
+ \ {\bf b}_j^T{\bf R}_{n}^{-1}{\bf B}_n^{-1}{\bf b}_j\Bigg\}\\
& + \ \dfrac{t_n\sigma_n^2}{p}{\rm tr}\left({\bf R}_{n}^{-1}{\bf B}_n^{-1}\right).
\endaligned
\]

Using (\ref{RRbeta}), we obtain
\begin{eqnarray*}
\Delta&=&\dfrac{1}{np} \ \sum_{j=1}^n \ \bigg[ \ \dfrac{\delta_n}{1+\delta_n} \ {\bf a}_j^T{\bf R}_{nj}^{-1}{\bf B}_n^{-1}{\bf a}_j \ - \ \dfrac{\delta_n}{1+\delta_n} \ \beta_j \  {\bf a}_j^T{\bf R}_{nj}^{-1}{\pmb\xi}_j{\pmb\xi}_j^T{\bf R}_{nj}^{-1}{\bf B}_n^{-1}{\bf a}_j\\
&&+ \ {\bf b}_j^T{\bf R}_{nj}^{-1}{\bf B}_n^{-1}{\bf a}_j \ - \ \beta_j \ {\bf b}_j^T{\bf R}_{nj}^{-1}{\pmb\xi}_j{\pmb\xi}_j^T{\bf R}_{nj}^{-1}{\bf B}_n^{-1}{\bf a}_j\\
&&  + \ {\bf a}_j^T{\bf R}_{nj}^{-1}{\bf B}_n^{-1}{\bf b}_j \ - \ \beta_j \ {\bf a}_j^T{\bf R}_{nj}^{-1}{\pmb\xi}_j{\pmb\xi}_j^T{\bf R}_{nj}^{-1}{\bf B}_n^{-1}{\bf b}_j\\
&&+ \ {\bf b}_j^T{\bf R}_{nj}^{-1}{\bf B}_n^{-1}{\bf b}_j \ - \ \beta_j \ {\bf b}_j^T{\bf R}_{nj}^{-1}{\pmb\xi}_j{\pmb\xi}_j^T{\bf R}_{nj}^{-1}{\bf B}_n^{-1}{\bf b}_j \ \bigg]\\
&&+ \ \dfrac{t_n\sigma_n^2}{p} \ {\rm tr}({\bf R}_n^{-1}{\bf B}_n^{-1}).
\end{eqnarray*}
Define
\begin{eqnarray}\label{etaj}
\begin{array}{cc}
\rho_j=\dfrac{1}{n}{\bf a}_j^T{\bf R}_{nj}^{-1}{\bf a}_j,
& \hat{\rho}_j=\dfrac{1}{n}{\bf a}_j^T{\bf R}_{nj}^{-1}{\bf B}_n^{-1}{\bf a}_j, \\
w_j=\dfrac{1}{n}{\bf b}_j^T{\bf R}_{nj}^{-1}{\bf b}_j,
& \hat{w}_j=\dfrac{1}{n}{\bf b}_j^T{\bf R}_{nj}^{-1}{\bf B}_n^{-1}{\bf b}_j, \\
\eta_j=\dfrac{1}{n}{\bf a}_j^T{\bf R}_{nj}^{-1}{\bf b}_j,
& \hat{\eta}_j=\dfrac{1}{n}{\bf a}_j^T{\bf R}_{nj}^{-1}{\bf B}_n^{-1}{\bf b}_j,\\
\gamma_j=\dfrac{1}{n}{\bf b}_j^T{\bf R}_{nj}^{-1}{\bf a}_j,
&\hat{\gamma}_j=\dfrac{1}{n}{\bf b}_j^T{\bf R}_{nj}^{-1}{\bf B}_n^{-1}{\bf a}_j.
\end{array}
\end{eqnarray}
Certainly, $\eta_j=\ga_j$, but introducing $\ga_j$ makes the computations below more clear.

Recall that ${\pmb\xi}_j=({\bf a}_j+{\bf b}_j)/\sqrt{n}$, and so
$\beta_j^{-1}=1+\rho_j+w_j+\eta_j+\ga_j$.
We can then rewrite $\Delta$ as
\begin{eqnarray*}
\Delta&=&\dfrac{1}{p}\sum_{j=1}^n\beta_j\bigg(
\dfrac{\delta_n}{1+\delta_n}\hat{\rho}_j(1+\rho_j+\eta_j+\gamma_j+w_j)
-\dfrac{\delta_n}{1+\delta_n}(\rho_j+\eta_j)(\hat{\rho}_j+\hat{\ga}_j)\\
&&+\hat{\gamma}_j(1+\rho_j+\eta_j+\gamma_j+w_j)-(\ga_j+w_j)(\hat{\ga}_j+\hat{\rho}_j)\\
&&+\hat{\eta}_j(1+\rho_j+\eta_j+\gamma_j+w_j)-(\rho_j+\eta_j)(\hat{\eta}_j+\hat{w}_j)\\
&&+\hat{w}_j(1+\rho_j+\eta_j+\gamma_j+w_j)-(\gamma_j+w_j)(\hat{\eta}_j+\hat{w}_j)\bigg)\\
&&+\dfrac{t_n\sigma_n^2}{p} {\rm tr}({\bf R}_n^{-1}{\bf B}_n^{-1})\\
&=&\dfrac{1}{p}\sum_{j=1}^n \beta_j\left(\dfrac{1}{1+\delta_n}\hat{\rho}_j(\delta_n-\gamma_j-w_j)+
\hat{\gamma}_j\left(1+\dfrac{1}{1+\delta_n}(\rho_j+\eta_j)\right)+\hat{\eta}_j+\hat{w}_j\right)\\
&&+\dfrac{t_n\sigma_n^2}{p} {\rm tr}({\bf R}_n^{-1}{\bf B}_n^{-1})\\
&:=&\Delta_1+\Delta_2+\Delta_3,
\end{eqnarray*}
where
\begin{align}\label{eq:Delta}
\notag \Delta_1
\notag         =&\dfrac{1}{p(1+\delta_n)}\sum_{j=1}^n\beta_j\hat{\rho}_j(\delta_n-w_j)
-\dfrac{1}{p(1+\delta_n)}\sum_{j=1}^n\beta_j\hat{\rho}_j\gamma_j,\\
      \Delta_2
       =&\dfrac{1}{p}\sum_{j=1}^n\beta_j\hat{\gamma}_j\left(1+\dfrac{1}{1+\delta_n}(\rho_j+\eta_j)\right)
+\dfrac{1}{p}\sum_{j=1}^n\beta_j\hat{\eta}_j, \q \mbox{and}\\
\notag  \Delta_3
\notag         =&\dfrac{1}{p}\sum_{j=1}^n\beta_j\left(\hat{w}_j-\dfrac{\sigma_n^2}{n} {\rm tr}({\bf R}_n^{-1}{\bf B}_n^{-1})\right),
\end{align}
and in the last equality we used the equivalent definition \eqref{dfn:t_alternative} of $t_n$.

\begin{lem}\label{4Lems}
Suppose that $t_n$ solves equation (\ref{eqn:t_n}) for $z=u+iv\in\bC^*$; then,
 \begin{compactenum}[(i)]
  \item\label{betabound}  for all $j=1,\ldots,n$, $|\beta_j|$ is bounded by $\dfrac{|z-t_n\sigma_n^2|}{v-\Im(t_{n})\sigma_n^2}$;
  \item\label{Bbound} $\|\B_n^{-1}\|$ is bounded by $v^{-1}$;
  \item\label{rv4} the random variables $\varpi_j$  satisfy
$$\max_{1\le j\le n} E|\varpi_j|^4\le \ \dfrac{C(\log n)^6}{n^2(v-t_{n2}\si_n^2)^4},$$
where $\varpi_j$ can be any of $\eta_j$, $\hat{\eta}_j$, $\gamma_j$ and $\hat{\gamma}_j$ defined in (\ref{etaj}), and $C$ is a constant independent of $n$;
  \item\label{ww}  the random variables $w_j$ and $\hat{w}_j$ satisfy
\begin{eqnarray*}
&&\max_{1\le j\le n} E\left|w_j-\dfrac{\sigma_n^2}{n}{\rm tr}({\bf R}_n^{-1})\right|^4
\le \ \dfrac{C(\log n)^8}{n^2(v-t_{n2}\sigma_n^2)^4}, \\
&&\max_{1\le j\le n} E\left|\hat{w}_j-\dfrac{\sigma_n^2}{n}{\rm tr}({\bf R}_n^{-1}{\bf B}_n^{-1})\right|^4
\le \ \dfrac{C(\log n)^8}{n^2v^4(v-t_{n2}\sigma_n^2)^4}.
\end{eqnarray*}
 \end{compactenum}
\end{lem}
\begin{proof}
We first prove  \eqref{betabound}.
Write $t_n=t_{n1}+it_{n2}$. Note that
\begin{eqnarray*}
&&\Im\left\{(z-t_n\sigma_n^2){\pmb\xi}_j^T{\bf R}_{nj}^{-1}{\pmb\xi}_j\right\}\\
=&&\Im\left\{{\pmb\xi}_j^T\left(\dfrac{1}{z-t_n\sigma_n^2}{\bf S}_{nj}-\I\right)^{-1}{\pmb\xi}_j\right\}\\
=&&\dfrac{1}{2i}{\pmb\xi}_j^T\left[\left(\dfrac{1}{z-t_n\sigma_n^2}{\bf S}_{nj}-\I\right)^{-1}
-\left(\dfrac{1}{\overline{z-t_n\sigma_n^2}}{\bf S}_{nj}-\I\right)^{-1}\right]{\pmb\xi}_j\\
=&&\dfrac{v-t_{n2}\sigma_n^2}{|z-t_n\sigma_n^2|^2} \ {\pmb\xi}_j^T\left(\dfrac{1}{z-t_n\sigma_n^2}{\bf S}_{nj}-\I\right)^{-1}
{\bf S}_{nj}\left(\dfrac{1}{\overline{z-t_n\sigma_n^2}}{\bf S}_{nj}-\I\right)^{-1}{\pmb\xi}_j \\
\geq && 0,
\end{eqnarray*}
where the last inequality is because of  Lemma \ref{lem:tlimit}. Therefore,
\begin{eqnarray*}
|\beta_j|&=& \dfrac{|z-t_n\sigma_n^2|}{ \ |(z-t_n\sigma_n^2)(1+{\pmb\xi}_j^T{\bf R}_{nj}^{-1}{\pmb\xi}_j)| \ } \\
&\le& \dfrac{|z-t_n\sigma_n^2|}{ \ |\Im\{(z-t_n\sigma_n^2)(1+{\pmb\xi}_j^T{\bf R}_{nj}^{-1}{\pmb\xi}_j)\}| \ }  \\
&\le&   \dfrac{ \ |z-t_n\sigma_n^2| \ }{ \ v-t_{n2}\sigma_n^2 \ }.
\end{eqnarray*}

As to \eqref{Bbound}, note that any eigenvalue of $\B_n=\dfrac{1}{n(1+\delta_n)}{\bf A}_n{\bf A}_n^T-z\I$
can be expressed as $\la^B=\dfrac{1}{1+\de_n}\la-z$, where $\la$ is an eigenvalue of $\dfrac{1}{n}\A_n\A_n^T$. We have
\[
|\lambda^B|\geq |\Im(\lambda^B)|=\left|\dfrac{\Im(\delta_n)}{|1+\delta_n|^2}\lambda+v\right|\geq v,
\]
where the last step follows from the fact that
$
\Im(\delta_n)=y_n\sigma_n^2\Im(m_n(z-t_n\sigma_n^2)) \ >0,
$
because of Lemma \ref{lem:tlimit}.

Now we prove \eqref{rv4}. We shall only establish the inequality for $\eta_j(=\ga_j)$; the other two variables $\hat{\eta}_j$ and $\hat{\gamma}_j$ can be handled in a similar way by using~\eqref{Bbound}.

Because for any Hermitian matrix $\A$ and $z\in\bC^+$, $\|(\A-z\I)^{-1}\|\le 1/\Im(z)$, we have by Lemma \ref{lem:tlimit}  that
\begin{eqnarray}\label{bound_RRj}
~~~~\|{\bf R}_n^{-1}\|\le \dfrac{1}{(v-t_{n2}\sigma_n^2)}, ~~~{\rm and}~~~
\max_{1\le j\le n}\|{\bf R}_{nj}^{-1}\|\le \dfrac{1}{(v-t_{n2}\sigma_n^2)}.
\end{eqnarray}

Recall that ${\bf b}_j=\sigma_n \pep_j$, and $\pep_j$ satisfies   $E(\pep_j\pep_j^T)=\I$.
The strengthened assumption \eqref{asm:A_bd} implies that $|{\bf a}_j|\le C\sqrt{n\log n}$.
Note also that $\pep_j$ is independent of ${\bf R}_{nj}^{-1}$ and ${\bf a}_j$. Moreover,  using Lemma \ref{xtrx} in  Appendix \ref{appendix:lemmas},
Assumption \eqref{asm:eps_bd} and \eqref{bound_RRj}, we get
\begin{eqnarray*}
E|\eta_j|^4&=&\dfrac{1}{n^4}E|{\bf a}_j^T{\bf R}_{nj}^{-1}{\bf b}_j|^4
=\dfrac{\sigma_n^4}{n^4}E|{\bf a}_j^T{\bf R}_{nj}^{-1}{\bf \pep}_j|^4\\
&=&\dfrac{\sigma_n^4}{n^4}E\left({\bf \pep}_j^T\bar{{\bf R}}_{nj}^{-1}{\bf a}_j{\bf a}_j^T{\bf R}_{nj}^{-1}{\bf \pep}_j\right)^2\\
&\le&\dfrac{2\sigma_n^4}{n^4}\left(
E|{\bf \pep}_j^T\bar{{\bf R}}_{nj}^{-1}{\bf a}_j{\bf a}_j^T{\bf R}_{nj}^{-1}{\bf \pep}_j-{\bf a}_j^T{\bf R}_{nj}^{-1}\bar{{\bf R}}_{nj}^{-1}{\bf a}_j|^2
+E({\bf a}_j^T{\bf R}_{nj}^{-1}\bar{{\bf R}}_{nj}^{-1}{\bf a}_j)^2
\right)\\
&\le&\dfrac{C}{n^4}E|\ep_{11}|^4\times E\left({\bf a}_j^T{\bf R}_{nj}^{-1}\bar{{\bf R}}_{nj}^{-1}{\bf a}_j\right)^2\\
&\le& \dfrac{C(\log n)^6}{n^2(v-t_{n2}\si_n^2)^4}.
\end{eqnarray*}

Finally, we prove \eqref{ww}.
Using \eqref{asm:eps_bd}, \eqref{bound_RRj},  \eqref{Bbound}, Lemma \ref{xtrx} and Lemma~2.6 in \cite{SB95}, we obtain
\begin{eqnarray*}
&&E\left|w_j-\dfrac{\sigma_n^2}{n}{\rm tr}({\bf R}_n^{-1})\right|^4\\
&\le& C\(E\left|\dfrac{\sigma_n^2}{n}\pep_j^T{\bf R}_{nj}^{-1}\pep_j-\dfrac{\sigma_n^2}{n}{\rm tr}({\bf R}_{nj}^{-1})\right|^4
+E\left|\dfrac{\sigma_n^2}{n}{\rm tr}({\bf R}_{nj}^{-1}-{\bf R}_n^{-1})\right|^4\)\\
&\le&\dfrac{C}{n^4}\left|E\((\log n)^4{\rm tr}\({\bf R}_{nj}^{-1}\bar{{\bf R}}_{nj}^{-1}\)\)^2
+(\log n)^8 E{\rm tr}\({\bf R}_{nj}^{-1}\bar{{\bf R}}_{nj}^{-1}\)^2\right|
+\dfrac{C}{n^4(v-t_{n2}\si_n^2)^4}\\
&\le&\dfrac{C(\log n)^8}{n^2(v-t_{n2}\sigma_n^2)^4}.
\end{eqnarray*}
The result for $\hat{w}_j$ can be proved similarly.
\end{proof}

\subsection{Proof of Proposition \ref{prop:LSD_signal_noise}}
\label{ssec:pf_prop_LSD_signal_noise}
\begin{proof}[Proof of Proposition \ref{prop:LSD_signal_noise}] First, the existence and uniqueness of $t_n$ have been established in Lemma \ref{lem:tn_exist}.

Next, to show \eqref{eqn:prop_LSD_signal_noise}, we recall the $\De_j, j=1,2,3$ defined in \eqref{eq:Delta}.
The proof is completed if we show $\De_j\to 0$ almost surely for all $j=1,2,3$.

By
\eqref{bound_RRj},
\eqref{asm:A_bd}
and Lemma \ref{4Lems} \eqref{Bbound}, there exists a constant $C$ such that
\begin{eqnarray}
~~~
\max_{j=1,\ldots,n}  |\rho_j| \le \ \dfrac{C\log(n)}{v-t_{n2}\sigma_n^2}, \q \mbox{and}\q
\max_{j=1,\ldots,n} |\hat{\rho}_j| \le \ \dfrac{C\log(n)}{v(v-t_{n2}\sigma_n^2)}.
\label{rhoj}
\end{eqnarray}
Moreover, by Lemmas  \ref{lem:tlimit} and \ref{lem:t_properties} and  the convergence of $\{F^{\S_n}\}$, we have that as $p\to\infty$,
\begin{eqnarray}\label{eqn:de_conv}
\de_n=y_n\si_nm_n(z-t_n\si_n^2) ~ \to ~ \de=\de(z)=y\si^2m(z-t\si^2),
\end{eqnarray}
and $\Im(\de)>0$.
In particular, for all sufficiently large $n$, we have
\begin{eqnarray}
\dfrac{1}{ \ |1+\delta_n| \ } \ \le \ \dfrac{2}{\liminf_n\Im(\de_n)} <\infty.
\label{delta0}
\end{eqnarray}

We now show that $\De_3\to 0$ almost surely.
Using Markov's inequality and H\"{o}lder's inequality, for any $\eps>0$, we have
\begin{eqnarray*}
{\rm P}\(|\De_3|\geq \eps\)&\le& \dfrac{1}{\eps^4} E
\left|\dfrac{1}{p}\sum_{j=1}^n\beta_j \(\hat{w}_j-\dfrac{\sigma_n^2}{n}{\rm tr}({\bf R}_n^{-1}{\bf B}_n^{-1})\)\right|^4\\
&\le& \dfrac{n^3}{p^4\eps^4}\sum_{j=1}^nE|\beta_j|^4
\left|\hat{w}_j-\dfrac{\sigma_n^2}{n}{\rm tr}({\bf R}_n^{-1}{\bf B}_n^{-1})\right|^4\\
&\le&\dfrac{C(\log n)^8}{n^2\eps^4v^4(v-t_{n2}\si_n^2)^8}\cdot|z-t_n\si_n^2|^4,
\end{eqnarray*}
where the last step follows from Lemma \ref{4Lems}  \eqref{betabound} and \eqref{ww}.
Thus, $\De_3\to 0$ almost surely by Lemmas \ref{lem:tlimit} and   \ref{lem:t_properties} and the Borel-Cantelli Lemma.

Similarly we can prove that $\De_j\to 0$ almost surely for $j=1,2$ by using Lemmas \ref{lem:tlimit}, \ref{lem:t_properties}, and \ref{4Lems} and inequalities (\ref{rhoj}) and (\ref{delta0}).
\end{proof}


\subsection{Proof of Theorem \ref{thm:LSD_signal_noise}}
\label{ssec:pf_thm_LSD_signal_noise}

\begin{proof}[Proof of Theorem \ref{thm:LSD_signal_noise}]

We first show that equation (1.1) in \cite{DS2007a} can be derived from Proposition \ref{prop:LSD_signal_noise}. For any fixed $z\in\bC^*$, by Proposition~\ref{prop:LSD_signal_noise}, Lemmas \ref{lem:tlimit}, \ref{lem:t_properties}, and Lemma \ref{4Lems} \eqref{Bbound},  and the dominated convergence theorem, we obtain
\begin{eqnarray}\label{rem}
m(z-t\si^2) \ = \ \int\dfrac{1}{(1+\de)^{-1}x-z} \  dF^{\cA}(x),
\end{eqnarray}
where $t$ is the unique solution to equation \eqref{eqn:t}, $\de=y\si^2m(z-t\si^2)$,
and $m(\cdot)$ is the Stieltjes transform of the probability distribution $F$.
Moreover, if we let $\ga(z)=z-t(z)\si^2$, then by~\eqref{eqn:t}  and \eqref{eqn:de_conv}, we have
$$t=y-1+y\ga m(\ga), ~~~~~~~~~~~\de=y\si^2m(\ga),$$
and
\[
z=\ga+t\si^2=\ga+\ga y\si^2 m(\ga)+\si^2(y-1).
\]
Substituting the expressions of $t$, $\de$, and $z$ in terms of $\ga$ into equation (\ref{rem}) yields
\begin{equation}\label{rem_limit}
m(\ga)=\int \dfrac{dF^{\cA}(x)}{\dfrac{x}{1+y\si^2m(\ga)}-\ga(1+y\si^2m(\ga))-\si^2(y-1)},
\end{equation}
where $\ga\in\bC_{\ga}:=\{\ga=z-t(z)\si^2: z\in\bC^*\}$.  Note further that by Lemma~\ref{lem:t_analytic}, $\ga(z)$ is analytic on $\bC^*$. It then follows from the uniqueness of analytic continuation that equation (\ref{rem_limit}) holds for every $\ga\in\bC^+$ -- in other words, equation~(1.1) in \cite{DS2007a} holds.

In the following, we show that equation (\ref{eqn:LSD_signal_to_noisy}) in Theorem \ref{thm:LSD_signal_noise} holds.

For any $z\in\bC^*$, denote $\al(z)=z(1+\de(z))$, where  $\de(z)=y\si^2 m(\ga)$ and $\ga=z-t\si^2$.  We further define
\[
d(\ga)=1+y\si^2m(\ga) (=1+\de(z)), ~~~~{\rm and}~~~~ g(\al)=1-y\si^2m_{\cA}(\al).
\]
We  show the following facts:
\begin{compactenum}\setcounter{enumi}{6}
\item[]
\begin{compactenum}
\item\label{F1} $g(\al)=1/d(\ga)$,
\item\label{F2} $\al=\ga d^2(\ga)+\si^2(y-1)d(\ga)$, or equivalently, \\ $\ga=\al g^2(\al)-\si^2(y-1)g(\al)$.
\end{compactenum}
\end{compactenum}

In fact, we can rewrite equation \eqref{rem} as
\begin{eqnarray*}
m_{\cA}\(\al\)
=(1+\de)^{-1} m(\ga).
\end{eqnarray*}
Noting that $\de=y\si^2m(\ga),$
we have
\[
y\si^2 m_{\cA}(\al)=\dfrac{\de}{1+\de}, \mbox{ and hence } g(\al)=\frac{1}{1+\delta} = \frac{1}{d(\ga)},
\]
namely, \eqref{F1} holds.
In addition, $y\si^2m_{\cA}(\al)=1-{1}/{(1+\de)}$ implies $\al\in\bC^+$  because
$\de=y\si^2 m(z-t\si^2)\in \bC^+$ by Lemma \ref{lem:t_properties}.

We now show \eqref{F2}.
Let $\beta=t\si^2(1+\de)$. Then
\begin{eqnarray}\label{ga}
\ga=z-t\si^2=\dfrac{\al-\beta}{1+\de}.
\end{eqnarray}
By substituting \eqref{ga} and $\de=y\si^2 m(\ga)$ into equation \eqref{eqn:t}, we obtain
\[
\dfrac{\beta}{\si^2(1+\de)}=y-1+\dfrac{\de(\al-\beta)}{\si^2(1+\de)}.
\]
That is,
\[
\beta=\si^2(y-1)+\dfrac{\de}{1+\de}\al.
\]
Therefore,
\begin{eqnarray*}
\ga&=&\dfrac{\al-\beta}{1+\de}
= \dfrac{\al}{(1+\de)^2}-\dfrac{\si^2(y-1)}{1+\de}\\
&=&\al g^2(\al)- \si^2(y-1)g(\al),
\end{eqnarray*}
namely,  \eqref{F2} holds.

Next, by \eqref{rem_limit} and the definitions of $\al$ and $d(\ga)$ and \eqref{F2}, we have
\begin{eqnarray*}
m(\ga)
&=&d(\ga)\int\dfrac{1}{x-\al}dF^{\cA}(x).
\end{eqnarray*}
Using the facts
\eqref{F1} and \eqref{F2},
we obtain
\begin{equation}\label{g_a}
\aligned
m_{\cA}(\al)&=\int\dfrac{dF^A(x)}{x-\al}
=\dfrac{1}{d(\ga)}\int\dfrac{1}{\tau-\ga}dF(\tau) \\
&=\int\dfrac{g(\al)}{\tau-\al g^2(\al)+\si^2(y-1)g(\al)} dF(\tau) \\
&=\int\dfrac{1}{\dfrac{\tau}{g(\al)}-\al g(\al)+\si^2(y-1)} dF(\tau).
\endaligned
\end{equation}
By plugging in the expression of $g(\al)$,  we see that for all
$\al=\al(z)=z(1+\de(z))$,
$m_{\cA}(\al)$ satisfies
\[
~~~~~~~ m_{\cA}(\al)=\int\dfrac{dF(\tau)}{\dfrac{\tau}{1-y\si^2 m_{\cA}(\al)}-\al(1-y\si^2 m_{\cA}(\al))+\si^2(y-1)}.
\]
It follows from the uniqueness of analytic continuation that the above equation holds for all $\al\in\bC^+$ such that
$m_{\cA}(\al)\in D_{\cA}.$

It remains to show that the solution to equation \eqref{eqn:LSD_signal_to_noisy} is unique in $D_{\cA}$.
Suppose that $m\in D_{\cA}$ satisfies equation \eqref{eqn:LSD_signal_to_noisy}.  Define
\begin{eqnarray}\label{ga_j}
\ga=z(1-y\si^2m )^2-\si^2(y-1)(1-y\si^2 m)\in \bC^+.
\end{eqnarray}
By \eqref{eqn:LSD_signal_to_noisy} and \eqref{ga_j}, we have
$
m=(1-y\si^2m)m(\ga).
$
Hence,
\begin{eqnarray}\label{m_ga_j}
m=\frac{m(\ga)}{1+y\si^2m(\ga)}, \q\mbox{and}\q  m(\ga)=\dfrac{m}{1-y\si^2m}.
\end{eqnarray}
The second identity implies that
\begin{eqnarray}\label{mm_ga_j}
1+y\si^2m(\ga)=\dfrac{1}{1-y\si^2m}.
\end{eqnarray}
Using \eqref{ga_j} and \eqref{mm_ga_j}, we have
\begin{eqnarray}\label{al_ga_j}
z&=&\dfrac{\ga}{(1-y\si^2m)^2}+\dfrac{\si^2(y-1)}{1-y\si^2m}\nonumber\\
&=&\ga(1+y\si^2m(\ga))^2+\si^2(y-1)(1+y\si^2m(\ga)).
\end{eqnarray}
Because the Stieltjes transform $m(\ga)$ is uniquely determined by equation~\eqref{rem_limit},  we obtain
\begin{eqnarray*}
m(\ga)
&=&\int\dfrac{dF^{\cA}(x)}{\dfrac{x}{1+y\si^2m(\ga)}-\ga(1+y\si^2m(\ga))-\si^2(y-1)}\\
&=&(1+y\si^2m(\ga)) \cdot m_{\cA}(z).
\end{eqnarray*}
It then follows from the first identity in \eqref{m_ga_j}  that $m = m_{\cA}(z)$ -- in other words, $m_{\cA}(z)$ is the unique solution in $D_{\cA}$.
\end{proof}


\section{Proof of Theorem~\ref{THM2}}

\begin{proof}[Proof of  Theorem \ref{thm:main}]
The convergence of $F^{\ICV}$ follows easily from Assumption~\eqref{asm:Sigma_conv} and the fact that
\[
F^{\ICV}(x) \ = \ F^{\bpSi} \left(\frac{x}{\int_0^1\ga_t^2\,dt}\right) ~~~~~~~ {\rm for~all~} x\geq 0.
\]

Next, by Theorem 3.2 in \cite{DS2007b}, the assumption that $F$ has a bounded support implies that $H$ also has a bounded support. Thus, Assumption (A.iii$'$) in \cite{ZL11}, which requires that $H$ has a finite second moment, is satisfied.

We proceed to show the convergence of $\cA_m$.
As discussed in
Section  \ref{ssec:App_I_2},
if the diffusion process $({\bf X}_t)$ belongs to Class~$\mathcal{C}$, the drift process ${\pmb \mu}_t\equiv 0$, and $(\gamma_t)$ is independent of $({\bf W}_t)$, then, conditional on $\{\gamma_t\}$, we have
\begin{eqnarray}
\Delta\overline{{\bf X}}_{2i} \ \eqD \
\sqrt{w_i}
\ \breve{{\pmb\Sigma}}^{1/2} \ {\bf Z}_i,
\label{rx}
\end{eqnarray}
where $w_i$ is as in \eqref{wie}
and is independent of ${\bf Z}_i$, and ${\bf Z}_i=(Z_i^{1},\ldots,Z_i^{p})^T$ consists of independent standard normals.
Hence, $\cA_m$ has the same distribution as  $\wt{\cA}_m$ defined as
\begin{eqnarray}
\wt{\cA}_m &=& 3 \sum_{i=1}^m w_i \ \breve{\pSi}^{1/2} \Z_i \Z_i^T \breve{\pSi}^{1/2}.
\label{Sm}
\end{eqnarray}

{\bf Claim 1}. Without loss of generality, we can assume that the drift process ${\pmb\mu}_t\equiv 0$ and $(\gamma_t)$ is independent of $({\bf W}_t)$.

First, whether the drift term $({\pmb\mu}_t)$ vanishes or not does not affect the LSD of $\cA_m$.
To see this, note that $\De\ol{\X}_{2i}=\wt{\V}_i+\wt{\Z}_i$, where
\begin{eqnarray}\label{Vi}
\wt{\V}_i= \sum_{|j|<k}\(1-\dfrac{|j|}{k}\)\int_{((2i-1)k+j-1)/n}^{((2i-1)k+j)/n} {\pmb\mu}_t\, dt,
\end{eqnarray}
and
\begin{eqnarray}\label{Zi}
\wt{\Z}_i=\pLa \cdot \sum_{|j|<k}\(1-\dfrac{|j|}{k}\) \int_{((2i-1)k+j-1)/n}^{((2i-1)k+j)/n} \ga_t\, d\W_t.
\end{eqnarray}
Because all the entries of $\wt{\V}_i$ are of order $O(k/n)=o(1/\sqrt{p})$, by  Lemma \ref{lemma1} in  Appendix \ref{appendix:lemmas},
$\cA_m$ and
$ 3\sum_{i=1}^m \wt{\Z}_i (\wt{\Z}_i)^T$
have the same LSD.

Next, by the same argument as in the beginning of the Proof of Theorem~1 in \cite{ZL11}, we can assume without loss of generality that $(\ga_t)$ is independent of $(\W_t)$.
It follows that $\cA_m$ and $\wt{\cA}_m$ have the same LSD.

{\bf Claim 2}.  $\max_{i,n} (mw_i)$ is bounded, and  there exists a piecewise continuous process $(w_s)$ with finitely many jumps such that
\begin{eqnarray}\label{eqn:wt_conv}
\lim_{n\to\infty}\sum_{i=1}^m\int_{((2i-2)k)/n}^{2ik/n} \ | 3 mw_i-w_s| \ ds=0.
\end{eqnarray}

Using the boundedness of $(\ga_t)$ assumed in \eqref{asm:gamma_conv}   and  $k=\lfloor\theta\sqrt{n}\rfloor$, one can easily show that  $\max_{i,n} (mw_i)$ is bounded.

Next we show that \eqref{eqn:wt_conv} is satisfied for
$w_s=(\ga_s^*)^2$.
Define
\begin{eqnarray*}
w_i^*&=&
\sum_{|j|<k} \(1-\dfrac{|j|}{k}\)^2 ~
 \int_{((2i-1)k+j-1)/n}^{((2i-1)k+j)/n} (\ga_t^*)^2 \, dt.
\end{eqnarray*}

Suppose that $(\ga_t^*)$ has $J$ jumps for $J\geq 1$. For each $j=1,\ldots,J$, there exists an $\ell_j$ such that the $j$th jump falls in the interval $[{(2\ell_j-2)k}/{n}, \ {(2\ell_j k)}/{n})$. Then
\begin{eqnarray*}
&& \sum_{i=1}^m\int_{((2i-2)k)/n}^{2ik/n} \ | 3 mw_i-w_s| \ ds\\
&=&\sum_{\ell_j\in\{\ell_1,\ldots,\ell_J\}}
\int_{((2\ell_j-2)k)/n}^{2\ell_jk/n} \ | 3 mw_{\ell_j}-w_s| \ ds\\
&&+\sum_{i\not\in\{\ell_1,\ldots,\ell_J\}}
\int_{((2i-2)k)/n}^{2ik/n} \ | 3 mw_i-w_s| \ ds\\
&:=&\De_1+\De_2.
\end{eqnarray*}
Because  $(mw_{\ell_j})$ and $|\ga_s^*|$ are both bounded, for any $\varepsilon>0$ and for any sufficiently large $n$, we have
\[
|\De_1|\le \dfrac{2k}{n}\cdot JC < \varepsilon.
\]
For the second term $\De_2$, because $(\ga_t^*)$ is continuous in $[{(2i-2)k}/{n},{(2ik)/{n}}]$ when $i\not\in\{\ell_1,\ldots,\ell_J\}$, and by \eqref{asm:gamma_conv}, $(\ga_t)$ uniformly converges to $(\ga_t^*)$,  for any $\varepsilon>0$ and for sufficiently large $n,p$, we have
\[
|\ga_t^*-\ga_{(2i-2)k/n}^*|<\varepsilon \mbox{ for all } t\in \left[\frac{(2i-2)k}{n},{\frac{2ik}{n}}\right], \mbox{ and } |\ga_t-\ga_t^*|<\varepsilon\mbox{ for all } t.
\]
Moreover, because $|\ga_t|\le C_2$, for all large $n$, we have
\begin{eqnarray*}
|\De_2|
&\le& \sum_i\int_{(2i-2)k/n}^{2ik/n} | 3mw_i- 3 mw_i^*|ds \\
&+&\sum_i\int_{(2i-2)k/n}^{2ik/n} \left| 3 mw_i^*-  \frac{ 3 m}{n}  \(\ga_{(2i-2)k/n}^*\)^2 \cdot  \sum_{|j|<k} \(1-\dfrac{|j|}{k} \)^2 \right| ds\\
&+&\sum_i\int_{(2i-2)k/n}^{2ik/n} \left| \frac{ 3 m}{n}   \(\ga_{(2i-2)k/n}^*\)^2\cdot  \sum_{|j|<k} \(1-\dfrac{|j|}{k} \)^2 -(\ga_s^*)^2 \right| ds\\
&\le&  3 m^2\cdot\dfrac{2k}{n}\cdot  \frac{2k(k+1)(2k+1)}{6k^2}\cdot \frac{(2C_2\varepsilon)}{n}\\
&+& 3 m^2\cdot\dfrac{2k}{n}\cdot \frac{2k(k+1)(2k+1)}{6k^2}\cdot \frac{(2C_2\varepsilon)}{n}\\
&+& \sum_i\int_{(2i-2)k/n}^{2ik/n}  \(\ga_{(2i-2)k/n}^*\)^2 ds\cdot \left| \frac{ 3 m}{n}  \cdot  \sum_{|j|<k} \(1-\dfrac{|j|}{k} \)^2 -1 \right|\\
&+&
\sum_i\int_{(2i-2)k/n}^{2ik/n} \left|\(\ga_{(2i-2)k/n}^*\)^2-(\ga_s^*)^2\right| ds\\
&\le& C\varepsilon.
\end{eqnarray*}
This completes the proof of \eqref{eqn:wt_conv}.

Finally, because $F^{\breve{\pSi}}\to\breve{H}$ and $\breve{H}(x/\zeta)=H(x)$ for $x\geq 0$, using Claim~2 and applying Theorem 1 in \cite{ZL11}, we conclude that
the ESD of~$\cA_m$ converges to $F^{\cA}$, whose Stieltjes transform satisfies
\begin{eqnarray}
m_{\cA}(z)&=&-\dfrac{1}{z}\int\dfrac{1}{\tau M(z)+1}d\breve{H}(\tau) \nonumber\\
&=&-\dfrac{1}{z}\int\dfrac{\zeta}{\tau M(z)+\zeta}dH(\tau),
\label{m_A}
\end{eqnarray}
where $M(z)$, together with another function $\wt{m}(z)$, uniquely solve the following equations in $\bC^+\times \bC^+$
\[
\left\{
\begin{array}{lll}
M(z) &=& -\dfrac{1}{z} \displaystyle\int_0^1 \dfrac{w_s}{1+y\wt{m}(z) w_s}ds,  \\
\wt{m}(z) &=& -\dfrac{1}{z} \displaystyle\int \dfrac{\tau}{\tau M(z)+1} d\breve{H}(\tau)
{=-\dfrac{1}{z} \displaystyle\int \dfrac{\tau}{\tau M(z)+\zeta} dH(\tau)}.
\end{array}
\right.
\]

\end{proof}


\section{Proof of  Theorem \ref{THM3}}

The convergence of the ESD of $\ICV$ has been proved in Theorem \ref{thm:main}. The rest of Theorem \ref{thm:B_n} is a direct consequence of the following two convergence results.

\begin{lem}\label{pthm3_a}
Under the assumptions of Theorem \ref{thm:B_n}, we have
\[
\lim_{p\to\infty} 3 \dfrac{\sum_{i=1}^m |\De\ol{\Y}_{2i}|^2}{p}
=\zeta, \q \mbox{almost surely}.
\]
\end{lem}

\begin{prop}\label{pthm3}
Under the assumptions of Theorem \ref{thm:B_n},
 $F^{\widetilde{\pSi}}$ converges almost surely, and the limit $\widetilde{F}$
is determined by $\breve{H}$ in that its Stieltjes transform $m_{\widetilde{F}}(z)$ satisfies the following equation
\begin{equation}\label{eqn:B_n}
m_{\widetilde{F}}(z) \ = \ \int_{\tau\in\mathbb{R}}\dfrac{1}{\tau\(1-y(1+zm_{\widetilde{F}}(z))\)-z} \ d\breve{H}(\tau), ~~~ {\rm for~ all}~ z\in\mathbb{C}^+.
\end{equation}
\end{prop}

We prove Proposition \ref{pthm3} first, and then give the proof of Lemma~\ref{pthm3_a}.

\begin{proof}[Proof of Proposition \ref{pthm3}]
We first show the convergence of $F^{\tSi}$.
The main reason that we choose~$k$ in such a way that $k/\sqrt{n}\to\infty$ is to make the noise term negligible. To be more specific, by choosing $k=\lfloor\th n^{\alpha}\rfloor$ for some $\alpha\in [(3+\ell)/(2\ell+2) ,1)$ where $\ell$ is the integer in Assumption \eqref{asm:eps_general},
 we shall show that
\[
\tpSi=y_m\sum_{i=1}^m\dfrac{\De\ol{\Y}_{2i}(\De\ol{\Y}_{2i})^T}{|\De\ol{\Y}_{2i}|^2}
~~~{\rm and }~~~
\widetilde{\tpSi}:=y_m\sum_{i=1}^m\dfrac{\De\ol{\X}_{2i}(\De\ol{\X}_{2i})^T}{|\De\ol{\X}_{2i}|^2}
\]
have the same LSD. This follow if we can show that
\begin{eqnarray}
\max_{i=1,\ldots,m} \left|\dfrac{|\De\ol{\Y}_{2i}|^2}{|\De\ol{\X}_{2i}|^2}-1\right| \to 0 ~~~~~ {\rm almost ~ surely},
\label{prop1}
\end{eqnarray}
and
\begin{eqnarray}
y_m\sum_{i=1}^m\dfrac{\De\ol{\Y}_{2i}(\De\ol{\Y}_{2i})^T}{|\De\ol{\X}_{2i}|^2}~~{\rm and ~~} \widetilde{\tpSi}~~~{\rm have~ the ~ same ~ LSD}.
\label{prop2}
\end{eqnarray}

We start with \eqref{prop1}. Because
\begin{eqnarray*}
\left|\dfrac{|\De\ol{\Y}_{2i}|^2}{|\De\ol{\X}_{2i}|^2}-1\right|
&=&\left|\dfrac{|\De\ol{\X}_{2i}|^2+|\De\ol{\pvep}_{2i}|^2+2(\De\ol{\X}_{2i})^T(\De\ol{\pvep}_{2i})}{|\De\ol{\X}_{2i}|^2}-1\right|\\
&\le&\(\dfrac{|\De\ol{\pvep}_{2i}|}{|\De\ol{\X}_{2i}|}\)^2+2\dfrac{|\De\ol{\pvep}_{2i}|}{|\De\ol{\X}_{2i}|},
\end{eqnarray*}
 to prove \eqref{prop1}, it  suffices to show
\[
\max_{1\le i\le m} ~
\dfrac{|\De\ol{\pvep}_{2i}|}{|\De\ol{\X}_{2i}|} \to 0 ~~~~~~~~ {\rm almost ~surely}.
\]
Below, we shall prove the following slightly stronger result:
\begin{eqnarray}\label{ep_x}
\max_{1\le i\le m,1\le j\le p} ~
\dfrac{\sqrt{p}~|\De\ol{\varepsilon}_{2i}^{j}|}{|\De\ol{\X}_{2i}|}~ \to~ 0 ~~~~~~ {\rm almost ~surely},
\end{eqnarray}
where for any vector ${\bf a}$,  $a^{j}$ denotes its $j$th entry.

Note further that for \eqref{prop2},  using  Lemma \ref{lemma1} in  Appendix \ref{appendix:lemmas},  to prove \eqref{prop2}, it also suffices to show \eqref{ep_x}.

We now prove \eqref{ep_x}. We start with the denominator term $\De\ol{\X}_{2i}$.
We have $\De\ol{\X}_{2i}=\wt{\V}_i+\wt{\Z}_i$ for $\wt{\V}_i$ and $\wt{\Z}_i$ as defined in \eqref{Vi} and \eqref{Zi}, respectively.
Write $\wt{\Z}_i$ as $
\sqrt{w_i}
~ \pLa \Z_i$, where $w_i$ is defined in \eqref{wie} and
\begin{eqnarray*}
\Z_i=
\sum_{|j|<k}  \frac{1}{\sqrt{w_i}} \(1-\dfrac{|j|}{k}\)\int_{((2i-1)k+j-1)/n}^{((2i-1)k+j)/n} \ga_t\, d\W_t.
\end{eqnarray*}

 By Assumption \eqref{asm:leverage_2}, for all $j\notin \mathcal{I}_p$, $Z_i^{j}$ are \hbox{i.i.d.} $N(0,1)$. By using the same trick as in the proof of (3.34) in \cite{ZL11}, we have
\begin{equation}\label{eq:mv_normal_norm}
\max_{1\le i\le m} ~
\left|\dfrac{1}{p}|\pLa \Z_i|^2 - 1\right| ~\to ~ 0 ~~~~~~{\rm almost~surely}.
\end{equation}
Note that
\begin{eqnarray*}
|\De\ol{\X}_{2i}|^2
~=~|\wt{\V}_i+\wt{\Z}_i|^2
~\geq~|\wt{\V}_i|^2+|\wt{\Z}_i|^2-2|\wt{\V}_i||\wt{\Z}_i|.
\end{eqnarray*}

Assumption \eqref{asm:gamma_bdd} implies that for  all $i$, there exists $\wt{C}_1 $ such that
\[
|w_i| ~\geq ~ \wt{C}_1\dfrac{k}{n}.
\]
Therefore, by Assumption \eqref{asm:ym_conv}, there exists $C>0$ such that
\[
|\wt{\Z}_i|^2~=~
{|w_i|}
~|\pLa\Z_i|^2
~\geq~\dfrac{C}{p}|\pLa\Z_i|^2,
\]
which, together with \eqref{eq:mv_normal_norm}, implies that there exists $\delta_1>0$ such that for all sufficiently large $n$,
\[
\min_{1\le i\le m} ~ |\wt{\Z}_i|^2 \geq \delta_1.
\]
Moreover, by Assumption \eqref{asm:mu_bdd},
$|\wt{V}_i^{j}|\le C k/n$ for all $i,j$;  hence, $\max_i |\wt{\V}_i|=O(\sqrt{p}\times k/n)$, which, by Assumption \eqref{asm:ym_conv}, is $O(\sqrt{1/m})=o(1)$.   Therefore, there exists a constant $\delta>0$ such that, almost surely, for sufficiently large all~$n$,
\begin{equation}\label{eq:delta_X_bdd}
\min_{1\le i\le m} ~ |\De\ol{\X}_{2i}|^2 ~\geq~ \delta.
\end{equation}

It remains to prove that
\begin{eqnarray}\label{lim_eps}
\max_{1\le i\le m,1\le j\le p} ~
\sqrt{p}~ | \De\ol{\varepsilon}_{2i}^{j}|~\to~ 0,  ~~~ {\rm almost ~ surely.}
\end{eqnarray}
Observe that if we can show that there exists $C>0$ such that
\begin{eqnarray}\label{dev_lim_eps}
\max_{1\le i\le m,1\le j\le p} E|\De\ol{\vep}_{2i}^j|^{2\ell} \ \le \ Ck^{-\ell},
\end{eqnarray}
where $\ell $ is the integer in  Assumption~\eqref{asm:ym_conv}, then, for any $\varepsilon>0$, by Markov's inequality, we have
\begin{eqnarray*}
P\(\max_{1\le i\le m,1\le j\le p} \sqrt{p}~|\De\ol{\varepsilon}_{2i}^{j}|\geq\varepsilon\)
&\le& \sum_{1\le i\le m,1\le j\le p} \dfrac{p^\ell~ E|\De\ol{\varepsilon}_{2i}^{j}|^{2\ell}}{\varepsilon^{2\ell}}\\
&\le& \frac{C mp\cdot p^{\ell}}{k^\ell\varepsilon^{2\ell}} =
O\(\frac{1}{n^{(2+2\ell)\alpha - 2 - \ell}}\),
\end{eqnarray*}
where the last equation is due to Assumption \eqref{asm:ym_conv}. Because\\
$\al > (3+\ell)/(2\ell+2)$ by Assumption \eqref{asm:ym_conv} again, we have
$(2+2\ell)\alpha - 2 - \ell > 1$; hence, by the Borel-Cantelli Lemma, \eqref{lim_eps} holds.

We now show \eqref{dev_lim_eps}, which is a Marcinkiewicz-Zygmund type inequality.
We use Theorem 1 in \cite{DL99} to prove \eqref{dev_lim_eps}.
For that, we need to verify that $C_{r,2\ell}=O(r^{-\ell})$, where
\[
C_{r,2\ell}:= \max_{j=1,\ldots,p}\ \max_{1<M<2\ell} \sup_{(i_1,\cdots,i_{2\ell})\in\Theta_{r,M,2\ell}} \left| \cov( \vep_{i_1}^j \cdots \vep_{i_M}^j, \ \vep_{i_{M+1}}^j  \cdots \vep_{i_{2\ell}}^j ) \right|,
\]
where
$\Theta_{r,M,2\ell}=\{(i_1,\cdots,i_{2\ell}): i_1\le\ldots \le i_M < i_M+r \le i_{M+1}\le \ldots \le i_{2\ell}\}.$

We now verify that $C_{r,2\ell}=O(r^{-\ell})$.
For any $j$ and for any $(i_1,\cdots,i_{2\ell})\in\Theta_{r,M,2\ell}$, using the definition of $\rho$-mixing coefficients we have
\begin{eqnarray*}
&&\left| \cov( \vep_{i_1}^j  \cdots  \vep_{i_M}^j,\  \vep_{i_{M+1}}^j  \cdots \vep_{i_{2\ell}}^j ) \right| \\
&\le & \rho^j(r) \cdot \sqrt{ \var(\vep_{i_1}^j  \cdots  \vep_{i_M}^j) \cdot \var(\vep_{i_{M+1}}^j  \cdots \vep_{i_{2\ell}}^j) } \\
&\le & \rho^j(r) \cdot \sqrt{E((\vep_{i_1}^j)^2 \cdots (\vep_{i_M}^j)^2) \cdot E((\vep_{i_{M+1}}^j)^2\cdots (\vep_{i_{2\ell}}^j)^2 )}.
\end{eqnarray*}
By H\"{o}lder's inequality, we have
\[
E((\vep_{i_1}^j)^2 \cdots (\vep_{i_M}^j)^2) \ \le \ \(E(\vep_{i_1}^j)^{2M} \)^{1/M} \cdots \(E(\vep_{i_M}^j)^{2M} \)^{1/M}.
\]
And, similarly, for $E((\vep_{i_{M+1}}^j)^2 \cdots (\vep_{i_{2\ell}}^j)^2 )$.
Because $(\vep_i^j)$s have bounded $4\ell$th moments  
and $\max_{j=1,\cdots,p}  \rho^j(r)=O(r^{-\ell})$ by Assumption \eqref{asm:eps_general}, we get $C_{r,2\ell}=O(r^{-\ell})$.

Finally, by using an argument  similar to the last part of the proof of Proposition 8 in \cite{ZL11} (see pp.3142--3143), we have that $\widetilde{\tpSi}$
has the same LSD as
\[
\widetilde{\S}:= \ \dfrac{1}{m}\sum_{i=1}^m \pLa \Z_i \Z_i^T \pLa^T,
\]
where $\Z_i$ consists of independent standard normals.
It is well known that the LSD of $\widetilde{\S}$ is determined by \eqref{eqn:B_n}; hence, by the previous arguments, so is that of $\widetilde{\pSi}$.
\end{proof}

We now prove Lemma \ref{pthm3_a}.
\begin{proof}[Proof of Lemma \ref{pthm3_a}]
We have
\[
 \sum_{i=1}^m |\De\ol{\Y}_{2i}|^2
=\sum_{i=1}^m |\De\ol{\X}_{2i}|^2+ 2 \sum_{i=1}^m \De\ol{\X}_{2i}^T  \De\ol{\pvep}_{2i} + \sum_{i=1}^m|\De\ol{\pvep}_{2i}|^2.
\]
The convergence \eqref{lim_eps} and $p/m\to y$ imply that $\sum_{i=1}^m|\De\ol{\pvep}_{2i}|^2/p\to 0$ almost surely.
To prove the lemma, it then suffices to show that
\begin{equation}\label{subpav_conv}
  3\frac{\sum_{i=1}^m |\De\ol{\X}_{2i}|^2}{p}\to \zeta\q\mbox{almost surely},
\end{equation}
and
\begin{equation}\label{subpav_error_negligible}
  \frac{\sum_{i=1}^m \De\ol{\X}_{2i}^T  \De\ol{\pvep}_{2i}}{p}\to 0\q\mbox{almost surely}.
\end{equation}

We start with \eqref{subpav_conv}. Write $\De\ol{\X}_{2i}=\wt{\V}_i+\wt{\Z}_i$  as in the proof of Proposition~\ref{pthm3}. The convergence \eqref{eq:mv_normal_norm} implies that
\[
\frac{\sum_{i=1}^m |\wt{\Z}_i|^2}{p} = \sum_{i=1}^m w_i + \mbox{error},
\]
where the error term converges to 0 almost surely. By Riemann integration and Assumption \eqref{asm:gamma_conv} it is easy to show that  $\sum_{i=1}^m w_i\to \zeta/3$, so we get
\[
3\frac{\sum_{i=1}^m |\wt{\Z}_i|^2}{p} \to \zeta \q\mbox{almost surely}.
\]
Furthermore, by using the bound that $\max_i |\wt{\V}_i|=O(\sqrt{p}\times k/n)$ one can easily show that
\[
\frac{2\sum_{i=1}^m |\wt{\V}_i||\wt{\Z}_i| +  \sum_{i=1}^m |\wt{\V}_i|^2}{p} \to   0\q\mbox{almost surely}.
\]
We therefore get \eqref{subpav_conv}.

Finally, \eqref{subpav_error_negligible} follows from \eqref{lim_eps} and \eqref{subpav_conv}.
\end{proof}


\section{Verifying $\wh{m_{\cA_n}(z)} \approx m_{\cA_n}(z)$}\label{appendix:claim}

Suppose that $\wh{m_{\cA_n}(z)}\in D_{\cA}(y_n,\si_n^2)\\ (= \{\xi\in\bC: ~
z(1-y_n\si_n^2 \xi)^2- \si_n^2(y_n-1)(1-y_n\si_n^2 \xi) \in\bC^+\})$
satisfies the empirical version of equation~\eqref{eqn:LSD_signal_to_noisy}; in other words,
\begin{equation}\label{eqn:LSD_signal_to_noisy_emp}
\wh{m_{\cA_n}(z)} = \displaystyle\int
\dfrac{dF^{\S_n}(\tau)}{\dfrac{\tau}{1-y_n\si_n^2 \wh{m_{\cA_n}(z)}}-z(1-y_n\si_n^2 \wh{m_{\cA_n}(z)})+\si_n^2(y_n-1)}.
\end{equation}

First, we claim that $\{\wh{m_{\cA_n}(z)}\}$ is tight (regardless of whether $\wh{m_{\cA_n}(z)}$ is the unique solution or not).
Suppose to the contrary that $\{\wh{m_{\cA_n}(z)}\}$ is not tight; then, with positive probability, there exists a subsequence $\{n_k\}$ such that $|\wh{m_{\cA_{n_k}}(z)}|\to\infty$. However, by the tightness of $\{F^{\S_n}\}$, along such a subsequence, the right-hand side of \eqref{eqn:LSD_signal_to_noisy_emp} would converge to 0, while the left-hand side blows up, a contradiction.

Next we show that any limit of $\wh{m_{\cA_n}(z)}$ as $n\to\infty$, denoted by $\wt{m}$, has to be $m_\cA(z)$. Because $\wh{m_{\cA_n}(z)}\in D_{\cA}(y_n,\si_n^2)$, the limit $\wt{m}$ satisfies
\[
\Im\left(z(1-y\si^2 \wt{m})^2- \si^2(y -1)(1-y\si^2 \wt{m})\right) \geq 0.
\]
We want to show that $\wt{m}\in D_{\cA}(y,\si^2)$; in other words, the equality sign cannot hold. This can be be seen as follows.  First, by rewriting \eqref{eqn:LSD_signal_to_noisy_emp} as
\begin{eqnarray*}
\wh{m_{\cA_n}(z)} &=& (1-y_n\si_n^2 \wh{m_{\cA_n}(z)}) \\
&& \cdot \displaystyle\int
\dfrac{dF^{\S_n}(\tau)}{\tau-z(1-y_n\si_n^2 \wh{m_{\cA_n}(z)})^2+\si_n^2(y_n-1)(1-y_n\si_n^2 \wh{m_{\cA_n}(z)})},
\end{eqnarray*}
we see that $(1-y_n\si_n^2 \wh{m_{\cA_n}(z)})$ cannot converge to 0 (because otherwise the right-hand side would converge to 0 while the left-hand side would converge to
$1/(y\si^2)\neq 0$). It follows by taking the limits on both sides of the equation above that $\wt{m}$ satisfies
\[
\frac{\wt{m}}{1-y\si^2\wt{m}}=\displaystyle\int
\dfrac{dF(\tau)}{\tau-z(1-y\si^2 \wt{m})^2- \si^2(y -1)(1-y\si^2 \wt{m})}.
\]
Now, if $\Im\left(z(1-y\si^2 \wt{m})^2- \si^2(y -1)(1-y\si^2 \wt{m})\right)=0$, then the right-hand side would be a real number;  consequently, $\wt{m}$ has to be a real number as well. However, because we have just proved that $(1-y\si^2 \wt{m})\neq 0$, if $\wt{m}$ is a real number, then $\Im\left(z(1-y\si^2 \wt{m})^2- \si^2(y -1)(1-y\si^2 \wt{m})\right)$ cannot be zero, a contradiction.

To sum up, we have shown that $\wt{m}$ satisfies equation~\eqref{eqn:LSD_signal_to_noisy} and is inside $D_{\cA}(y,\si^2)$. Therefore, by the uniqueness of the solution to  equation~\eqref{eqn:LSD_signal_to_noisy} inside $D_{\cA}(y,\si^2)$, we have $\wt{m}=m_\cA(z)$. Consequently, $\wh{m_{\cA_n}(z)}\approx m_\cA(z)$ when $n$ (and $p$) are large.

Finally, because $m_{\cA_n}(z) \to m_\cA(z)$, we have that $\wh{m_{\cA_n}(z)}\approx m_{\cA_n}(z)$.

\renewcommand{\baselinestretch}{1.2}
\section{Two lemmas}\label{appendix:lemmas}
\setcounter{equation}{0}
\renewcommand{\theequation}{\thesection.\arabic{equation}}

\begin{lem}(Lemma 2.7 in \cite{BS98} ).
Let ${\bf X}=(X_1,\ldots,X_n)^T$ be a vector where the $X_i$s are centered \hbox{i.i.d.} random variables with unit variance. Let ${\bf A}$ be an $n\times n$ deterministic complex matrix. Then, for any $p\geq 2$,
\[
E\left|{\bf X}^T{\bf A}{\bf X}-\tr{\bf A}\right|^p\le C_p
\left(
\left(E|X_1|^4\tr{\bf A}{\bf A}^*\right)^{p/2}+E|X_1|^{2p}\tr({\bf A}{\bf A}^*)^{p/2}
\right).
\]
\label{xtrx}
\end{lem}

\begin{lem}(Lemma 1 in \cite{ZL11}).
Suppose that for each $p$, ${\bf v}_l=(v_l^{1},\ldots,v_l^{p})^T$ and ${\bf w}_l=(w_l^{1},\ldots,w_l^{p})^T$, $l=1,\ldots,m$, are all $p$-dimensional vectors. Define
\[
\wt{{\bf S}}_m=\sum_{l=1}^m({\bf v}_l+{\bf w}_l)({\bf v}_l+{\bf w}_l)^T
~~ \mbox{and} ~~
{\bf S}_m=\sum_{l=1}^m{\bf w}_l({\bf w}_l)^T.
\]
If the following conditions are satisfied,
\begin{compactenum}[(i)]
  \item $m=m(p)$ with $\lim_{p\to\infty} p/m=y>0$,
  \item there exists a sequence $\varepsilon_p=o(1/\sqrt{p})$ such that for all $p$ and all $l$, all the entries of ${\bf v}_l$ are bounded by $\varepsilon_p$ in absolute value;
  \item $\limsup_{p\to\infty} \tr({\bf S}_m)/p<\infty$ almost surely.
\end{compactenum}
Then, $L(F^{\wt{S}_m},F^{S_m})\to 0$ almost surely, where for any two probability distribution functions $F$ and $G$, $L(F,G)$ denotes the Levy distance between them.
\label{lemma1}
\end{lem}

\section{ Two figures in the simulation studies}\label{appendix:2_figures}
Figure \ref{fig:gamma_t} plots a sample path of $(\ga_t)$.
\begin{figure}[H]
\begin{center}
\includegraphics[width=9cm]{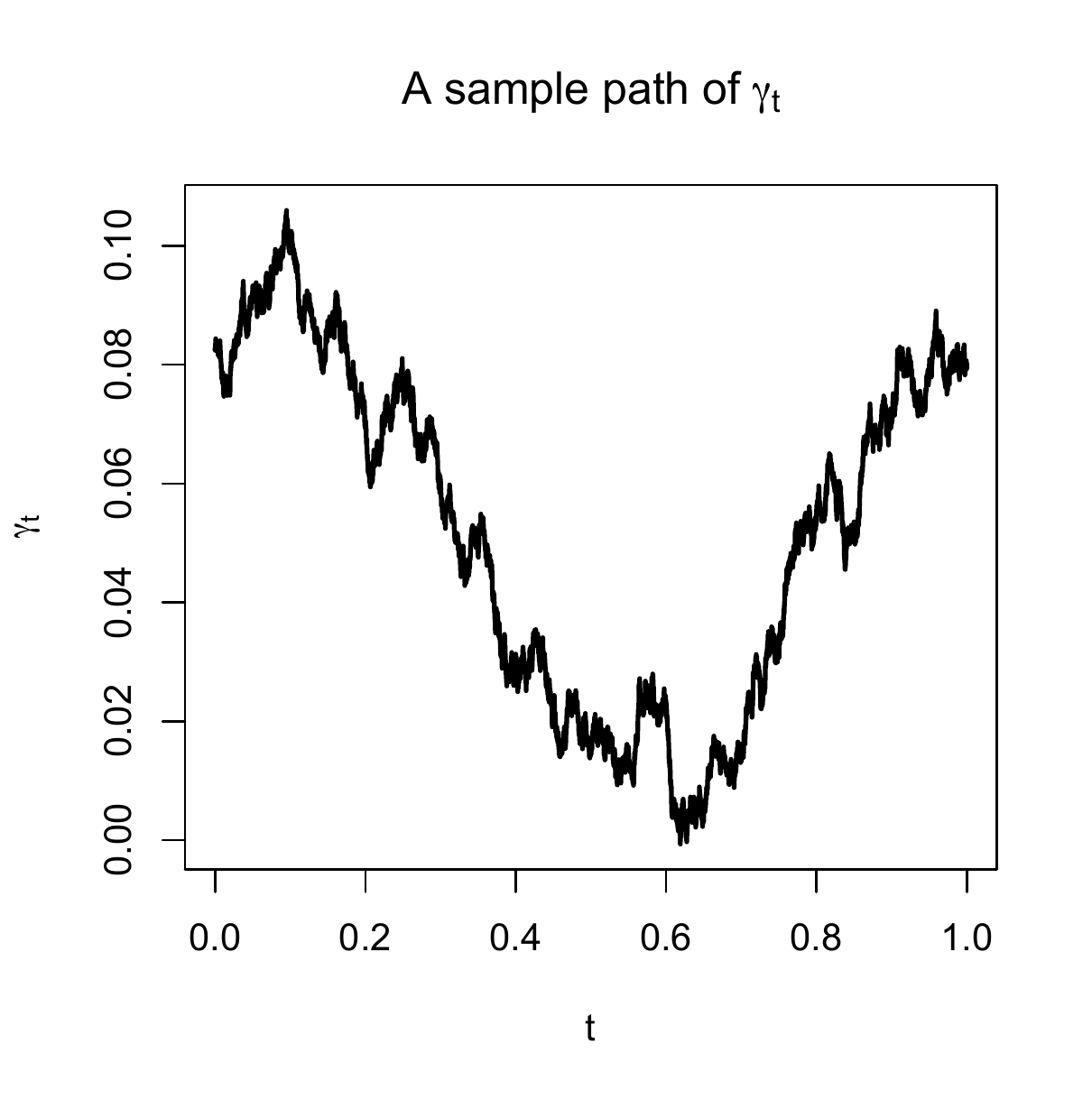}
\end{center}
\caption{A sample path of the process $(\ga_t)$.}
\label{fig:gamma_t}
\end{figure}

Figure \ref{fig:poisson} shows a realization of observation times for three stocks under the asynchronous observation setting.
\begin{figure}[H]
\includegraphics[width=12cm]{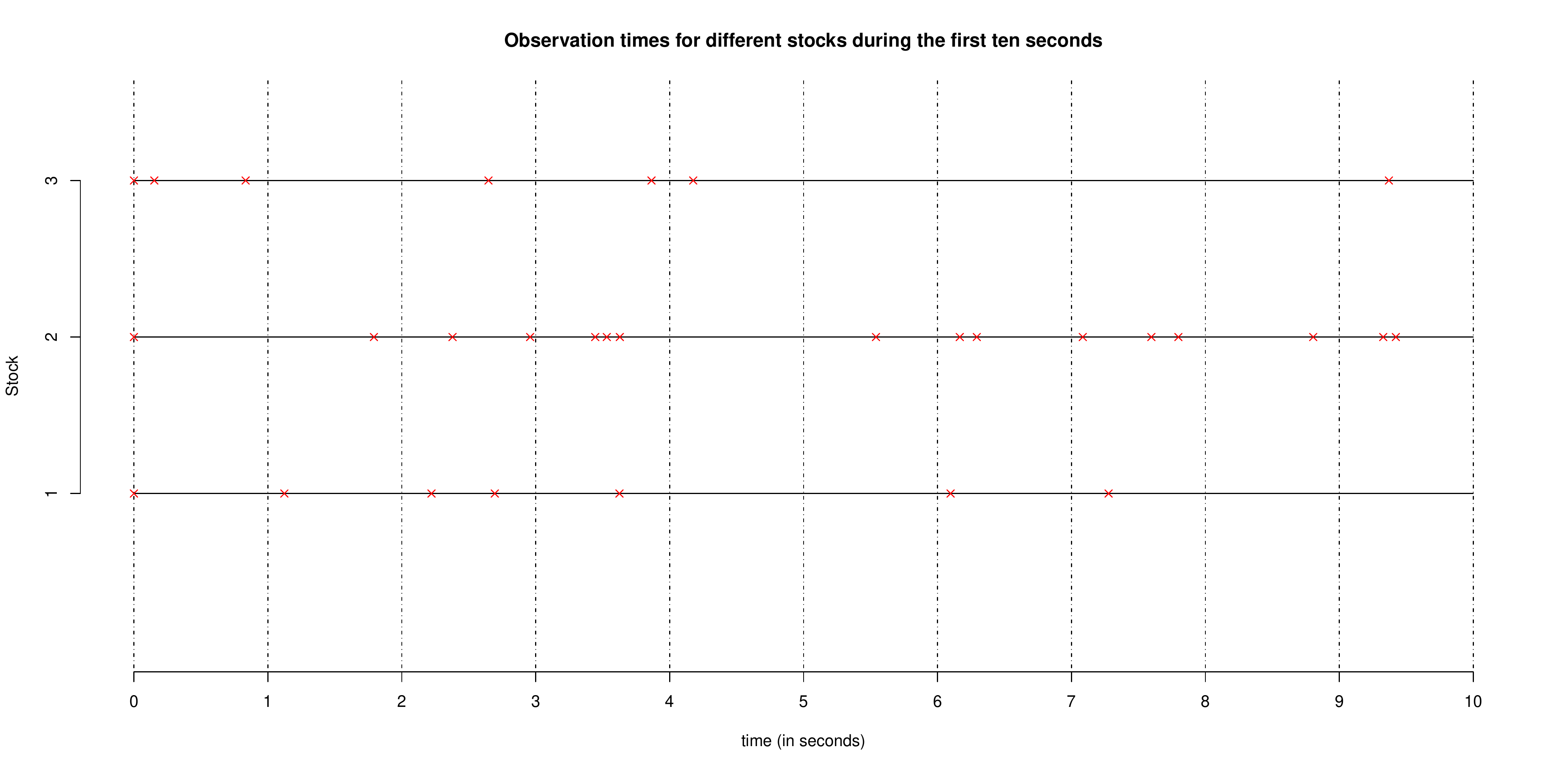}
\caption{One realization of the observation times during the first 10 seconds for three stocks. The observation times are generated as independent Poisson processes with rate 23,400.}
\label{fig:poisson}
\end{figure}

\end{supplement}

\end{document}